\documentclass[11pt,final]{amsart}

\usepackage{booktabs}
\usepackage[toc,page]{appendix}
\usepackage{enumerate}
\usepackage{mathrsfs}
\usepackage{fancyhdr}
\usepackage{latexsym, amsfonts, amssymb, amsmath,  amsthm}
\usepackage{amsopn, amstext, amscd,pifont}
\usepackage{amssymb,bm, cite}
\usepackage[all]{xy}
\usepackage{color}
\usepackage{graphicx,hyperref}
\usepackage{float}
\usepackage{listings}
\usepackage[T1]{fontenc}

\usepackage{manfnt}

\usepackage{xcolor}
\hypersetup{
    colorlinks, linkcolor={red!80!black},
    citecolor={blue!80!black}, urlcolor={blue}
}
\usepackage[color]{showkeys}
\definecolor{refkey}{gray}{.75}
\definecolor{labelkey}{gray}{.75}

\usepackage{graphicx,pgf}
\usepackage{float}

\usepackage{caption}
\usepackage{subcaption}

\newtheorem{thm}{Theorem}[section]
\newtheorem{lem}[thm]{Lemma}
\newtheorem{pro}[thm]{Proposition}

\newtheorem{cor}[thm]{Corollary}

\newtheorem{rem}[thm]{Remark}

\numberwithin{equation}{section}

\numberwithin{equation}{section}
\numberwithin{figure}{section}

\let\oldtocsection=\tocsection
\let\oldtocsubsection=\tocsubsection
\let\oldtocsubsubsection=\tocsubsubsection
\renewcommand{\tocsection}[2]{\hspace{0em}\oldtocsection{#1}{#2}}
\renewcommand{\tocsubsection}[2]{\hspace{2em}\oldtocsubsection{#1}{#2}}
\renewcommand{\tocsubsubsection}[2]{\hspace{4.5em}\oldtocsubsubsection{#1}{#2}}

\allowdisplaybreaks

\begin{document}
\title{Scalar Green function bounds for instantaneous shock location and one-dimensional stability of viscous shock waves}
\author{Yingwei Li}
\address{Yingwei Li, Indiana University, 831 East Third Street, Rawles Hall, Bloomington, Indiana 47405, U.S.A.}
\email{YL37@umail.iu.edu}
\date{\today}

\begin{abstract}
In this paper, we investigate and prove the nonlinear stability of viscous shock wave solutions of a scalar viscous conservation law \eqref{0.1}, using the methods developed for general systems of conservation laws by Howard, Mascia, Zumbrun and others, based on instantaneous tracking of the location of the perturbed viscous shock wave. In some sense, this paper extends the treatment in a previous expository work of Zumbrun ["Instantaneous shock location ..."] on Burgers equation to the general case, giving an exposition of these methods in the simplest setting of scalar equations. In particular we give by a rescaling argument a simple treatment of nonlinear stability in the small-amplitude case.
\end{abstract}

\maketitle



\section{Introduction}

The problem we consider here is the nonlinear stability of viscous shock wave solutions of scalar viscous conservation laws. With full details, we use the methods developed for general systems of conservation laws in \cite{ZH,MaZ2,MaZ3} in the simplest setting of a scalar conservation law \eqref{0.1}. We also extend these methods to the small-amplitude limiting case, recovering by a simple rescaling argument uniform bounds obtained by Liu and Zeng (\cite{LZe}, system case). This is also in a sense a sequel to the paper \cite{Z1} in which these techniques were exposed for Burgers equations. The work \cite{Z1} focused on the nonlinear stability argument and in particular the method of approximately locating the shock profile, restricting to Burgers equation for which explicit linearized Green function bounds are available through the Hopf-Cole transform. The current paper effectively completes the exposition by showing how similar Green function bounds can be obtained for general scalar equations by the use of inverse Laplace transform and complex stationary phase estimates.

Our treatment extends the study of scalar shock stability in \cite{Ho} by related techniques.  The main new aspects here are the formulation of the resolvent kernel formula in terms of dual and forward modes in a way generalizing conveniently to the system case, and improved bounds on ``excited'' and $y$-derivative terms sufficient to close the ``tracking-'' type  argument that has evolved to treat nonlinear stability in the system case.  These modifications allow the treatment of more general data (merely $L^1\cap L^\infty$ as compared to algebraically decaying data in $x$) than was treated in \cite{Ho}, and, in particular, the sharp treatment of the small-amplitude limit.

\subsection{Problem and equations}

We consider the \textit{scalar viscous conservation law}
\begin{equation} \label{0.1}
u_t + f(u)_x = u_{xx},
\end{equation}
where $u \in C^2(\mathbb{R}^2 \to \mathbb{R}), f \in C^2(\mathbb{R} \to \mathbb{R})$, $u=u(x,t)$.

Let us consider the standing wave solution
$\bar{u} = \bar{u}(x), \lim_{x \to \pm \infty} \bar{u}(x) = \bar{u}_{\pm}$
satisfying the Lax condition
$f^{\prime}(\bar{u}_+)<0<f^{\prime}(\bar{u}_-)$.
Thus $\bar{u}$ satisfies
\begin{eqnarray} \label{0.3}
f(\bar{u}(x))_x = \bar{u}_{xx}
\end{eqnarray}
or equivalently $f^{\prime}(\bar{u})\bar{u}_x = \bar{u}_{xx}$.
Linearizing \eqref{0.1}, about the solution $\bar{u}$, we obtain
\begin{eqnarray} \label{1.1}
v_t=Lv:=v_{xx}-(f^{\prime}(\bar{u})v)_x,\quad v(x,0)=g(x).
\end{eqnarray}
We call this the \textit{homogeneous linearized equation}.

\subsection{Main Result I: Pointwise Green Function Bounds}

Equation \eqref{1.1} can be solved explicitly by Hopf-Cole transformation, to give an exact solution formula
\begin{equation*}
e^{Lt}g=\int_{-\infty}^{+\infty}G(x,t;y)g(y)dy
\end{equation*}
where $G(x,t;y)$ is the Green function associated with the linearized evolution equation \eqref{1.1}.
Following \cite{MaZ} and \cite{ZH}, $G(x,t;y)$ has the following representation.
\begin{equation*}
\begin{aligned}
G(x,t;y)=\frac{1}{2\pi i}\mathrm{P.V.}\int_{\eta-i\infty}^{\eta+i\infty}e^{\lambda t}G_{\lambda}(x,y)d\lambda
        =\frac{1}{2\pi i}\lim_{T \to \infty}\int_{\eta-iT}^{\eta+iT}e^{\lambda t}G_{\lambda}(x,y)d\lambda
\end{aligned}
\end{equation*}
which is valid for $\eta$ sufficiently large.

\begin{thm}[Pointwise Green function bounds] \label{greenfunctionbounds}
Assuming the Lax condition, the Green function $G(x,t;y)$ associated with the linearized evolution equation \eqref{1.1} may be decomposed as
 \begin{equation*}
 G(x,t;y)=E(x,t;y)+S(x,t;y)+R(x,t;y)
 \end{equation*}
where for $y \leq 0:$
 \begin{equation} \label{8.4}
 E(x,t;y)=C\bar{u}^{\prime}(x)\left(\mathrm{errfn}\left(\frac{y+f^{\prime}(\bar{u}_{-})t}{\sqrt{4t}}\right)
 -\mathrm{errfn}\left(\frac{y-f^{\prime}(\bar{u}_{-})t}{\sqrt{4t}}\right)\right),
 \end{equation}
 \begin{equation} \label{8.5}
 \begin{aligned}
 S(x,t;y)
 =(4\pi t)^{-\frac{1}{2}}e^{-\frac{\left(x-y-f^{\prime}(\bar{u}_{-})t\right)^2}{4t}}\left(\frac{e^{-x}}{e^x+e^{-x}}\right)
 +(4\pi t)^{-\frac{1}{2}}e^{-\frac{\left(x-y+f^{\prime}(\bar{u}_{-})t\right)^2}{4t}}\left(\frac{e^{x}}{e^x+e^{-x}}\right),
 \end{aligned}
 \end{equation}
 \begin{equation} \label{8.6}
 R(x,t;y)
 =\mathbf{O}\left(e^{-\eta t}e^{-\frac{|x-y|^2}{Mt}}\right)
 +\mathbf{O}\left((t+1)^{-\frac{1}{2}}e^{-\eta x^+}+e^{-\eta |x|}\right)t^{-\frac{1}{2}}e^{-\frac{(x-y-f^{\prime}(\bar{u}_{-})t)^2}{Mt}},
 \end{equation}
 \begin{equation}
 R_{y}(x,t;y)
 =\mathbf{O}\left(e^{-\eta t}e^{-\frac{|x-y|^2}{Mt}}\right)
 +\mathbf{O}\left((t+1)^{-\frac{1}{2}}e^{-\eta x^+}+e^{-\eta |x|}\right)t^{-1}e^{-\frac{(x-y-f^{\prime}(\bar{u}_{-})t)^2}{Mt}}.
 \end{equation}
\end{thm}
Please refer to section \ref{sectiongreenfunctionbounds} for more detailed discussion.

\subsection{Main Result II: Nonlinear Stability of Viscous Shock Solutions}

With the Green function bounds established, we can deduce the nonlinearly orbitally asymptotically stability of the viscous
shock solutions, i.e., a perturbation $\tilde{u}=\bar{u}+u$ remains close to a set of translates of $\bar{u}$. The translation,
$\alpha(t)$, will be defined later on.

\begin{thm}[Stability of viscous shock solutions] \label{mainthm}
Viscous shock solutions $\bar{u}(x)$ of \eqref{0.1} are nonlinearly stable in $L^1\cap L^{\infty}$ and nonlinearly orbitally
asymptotically stable in $L^p$, $p>1$, with respect to initial perturbations $u_0$ that are sufficiently small in
$L^1\cap L^{\infty}$. More precisely, for some $C>0$ and $\alpha \in W^{1,\infty}(t)$,
 \begin{eqnarray*}
 |\tilde{u}(x,t)-\bar{u}(x-\alpha(t))|_{L^p(x)}&\leq&C(1+t)^{-\frac{1}{2}(1-\frac{1}{p})}
 |\tilde{u}-\bar{u}|_{L^1\cap L^{\infty}}|_{t=0},\\
 |\dot{\alpha}(t)|&\leq&C(1+t)^{-\frac{1}{2}}|\tilde{u}-\bar{u}|_{L^1\cap L^{\infty}}|_{t=0},\\
 |\alpha(t)|&\leq&C|\tilde{u}-\bar{u}|_{L^1\cap L^{\infty}}|_{t=0},\\
 |\tilde{u}-\bar{u}|_{L^1\cap L^{\infty}}(t)&\leq&C|\tilde{u}-\bar{u}|_{L^1\cap L^{\infty}}|_{t=0},
 \end{eqnarray*}
for all $t\geq 0$, $1\leq p\leq\infty$, for solutions $\tilde{u}$ of \eqref{0.1} with
$|\tilde{u}-\bar{u}|_{L^1\cap L^{\infty}}|_{t=0}$ sufficiently small.
\end{thm}
The proof is given in section \ref{sectionstability}.

\subsection{Main Result III: Nonlinear Stability of Viscous Shock Solutions in the Small-amplitude Setting}

Here we we consider the nonlinear stability of weak shocks(small-amplitude), i.e. viscous shock solutions with shock strength
$|u^{\epsilon}_{+}-u^{\epsilon}_{-}|=\epsilon \to 0$.

\begin{thm}[Stability of small amplitude viscous shock solutions] \label{mainthmsmall}
For $0<\epsilon\leq 1$, viscous shock solutions $\bar{u}^{\epsilon}(x)$ of \eqref{0.1} are nonlinearly stable in
$L^1\cap L^{\infty}$ and nonlinearly orbitally asymptotically stable in $L^p$, $p>1$, with respect to initial
perturbations $u_0$ with $L^1\cap L^{\infty}$ norm less than of equal to $\eta_0>0$ sufficiently small, where $\eta_0$ is
independent of $0<\epsilon\leq 1$. More precisely, for some $C>0$ independent of $0<\epsilon\leq 1$, there is
$\alpha \in W^{1,\infty}(t)$ such that,
 \begin{eqnarray*}
 |\tilde{u}(x,t)-\bar{u}^{\epsilon}(x-\alpha(t))|_{L^p(x)}&\leq&C(1+t)^{-\frac{1}{2}(1-\frac{1}{p})}
 |\tilde{u}-\bar{u}^{\epsilon}|_{L^1\cap L^{\infty}}|_{t=0},\\
 |\dot{\alpha}(t)|&\leq&C(1+t)^{-\frac{1}{2}}|\tilde{u}-\bar{u}^{\epsilon}|_{L^1\cap L^{\infty}}|_{t=0},\\
 |\alpha(t)|&\leq&C\epsilon^{-1}|\tilde{u}-\bar{u}^{\epsilon}|_{L^1\cap L^{\infty}}|_{t=0},\\
 |\tilde{u}-\bar{u}^{\epsilon}|_{L^1\cap L^{\infty}}(t)&\leq&C|\tilde{u}-\bar{u}^{\epsilon}|_{L^1\cap L^{\infty}}|_{t=0},
 \end{eqnarray*}
for all $t\geq 0$, $1\leq p\leq\infty$, for solutions $\tilde{u}$ of \eqref{0.1} with
$|\tilde{u}-\bar{u}^{\epsilon}|_{L^1\cap L^{\infty}}|_{t=0}\leq \eta_0$.
\end{thm}

\subsection{Discussion and future directions}

The nonlinear stability of weak shocks in the general system case was established using more detailed methods in \cite{LZe}; here we use our pointwise Green function bounds and some rescaling arguments to give a new simplified proof in the scalar case. See section \ref{sectionstabilitysmall} for more details. Our approach gives very simple proofs of stability with sharp rates in the small-amplitude limit, completes the treatment of \cite{Z1}. To investigate the small-amplitude nonlinear stability of the system case using our approach would be very interesting.

\section{The Asymptotic Eigenvalue Equations}

The eigenvalue equation $Lw=\lambda w$ associated with \eqref{1.1} is,
\begin{eqnarray} \label{2.1}
w^{\prime\prime}-(f^{\prime}(\bar{u})w)^{\prime} = \lambda w.
\end{eqnarray}
Written as first-order system in the variable $W=(w,w^{\prime})^{t}$, this becomes
\begin{eqnarray} \label{2.2}
W^{\prime}=\mathbb{A}(x;\lambda)W,
\end{eqnarray}
where
$
\mathbb{A}(x;\lambda):=
  \begin{pmatrix}
   0 & 1\\
   \lambda +f^{\prime\prime}(\bar{u})\bar{u}_x & f^{\prime}(\bar{u})
  \end{pmatrix}.
$
We begin by studying the limiting, constant coefficient systems $L_{\pm}w=\lambda w$ of \eqref{2.1} at $\pm \infty$,
$w^{\prime\prime}-(f^{\prime}(\bar{u}_{\pm})w)^{\prime} = \lambda w$. Or, written as a first-order system,
\begin{eqnarray} \label{2.5}
W^{\prime}=\mathbb{A}_{\pm}(\lambda)W,
\end{eqnarray}
where
$
\mathbb{A}_{\pm}(\lambda):=
  \begin{pmatrix}
   0 & 1\\
   \lambda & f^{\prime}(\bar{u}_{\pm})
  \end{pmatrix}
$,
since $\bar{u}^{\prime}_{\pm}=0$. The normal modes of \eqref{2.5} are $V_j^{\pm}e^{\mu_j^{\pm}x}$, $j=1,2$, where $\mu_j^{\pm},V_j^{\pm}$ are the eigenvalues and eigenvectors of $\mathbb{A}_{\pm}$; they are easily seen to satisfy
$
V_j^{\pm}=
 \begin{pmatrix}
 v_j^{\pm}\\
 \mu_j^{\pm}v_j^{\pm}
 \end{pmatrix}, v_j^{\pm} \in \mathbb{C}
$
and
$(\mu_j^{\pm})^2-f^{\prime}(\bar{u}_{\pm})\mu_j^{\pm}-\lambda=0$.
Then we can solve for $\mu_j^{\pm}$ as
$\mu_j^{\pm}(\lambda)=\frac{f^{\prime}(\bar{u}_{\pm})\pm \sqrt{(f^{\prime}(\bar{u}_{\pm}))^2+4\lambda}}{2}$.

\begin{lem} \label{Lem2.1}
Let $\mu_1(\lambda)=\frac{a-\sqrt{a^2+4\lambda}}{2}$ and $\mu_2(\lambda)=\frac{a+\sqrt{a^2+4\lambda}}{2}$ be two solutions to the equation $\mu^2-a\mu-\lambda=0$, where $a \in \mathbb{R},\lambda \in \mathbb{C}$. Then $\mathrm{Re}\mu_1(\lambda)<0<\mathrm{Re}\mu_2(\lambda)$ or $\mathrm{Re}\mu_1(\lambda)>0>\mathrm{Re}\mu_2(\lambda)$ (i.e. real parts of two solutions have different signs) if and only if $\lambda \in \Lambda$, where
$\Lambda=\left\{\lambda \in \mathbb{C}:a^2\mathrm{Re}(\lambda)+(\mathrm{Im}(\lambda))^2>0\right\}$.
\end{lem}

From this lemma we know that the \textit{region of consistent splitting} $\Lambda$ of system \eqref{2.5} contains the  intersection of the following two sets
$
\Lambda_{\pm}=\left\{\lambda \in \mathbb{C}:(f^{\prime}(\bar{u}_{\pm}))^2 \mathrm{Re}(\lambda)+(\mathrm{Im}(\lambda))^2>0\right\},
$
then the Lax condition
$f^{\prime}(\bar{u}_+)<0<f^{\prime}(\bar{u}_-)$
and Lemma \ref{Lem2.1} gives the two eigenvalues of $\mathbb{A}_{\pm}$ as following,
$\mathrm{Re}\mu_1(\lambda)<0<\mathrm{Re}\mu_2(\lambda)$.

\section{Asymptotic Behavior of the Stationary Solution $\bar{u}$} \label{sectionasy}

The stationary wave solution $\bar{u}(x)$ satisfies \eqref{0.3}: $f(\bar{u}(x))_x = \bar{u}_{xx}$ and the asymptotic conditions
$$
\lim_{x \to +\infty}\bar{u}(x)=\bar{u}_{+}, \quad
\lim_{x \to -\infty}\bar{u}(x)=\bar{u}_{-}, \quad
\lim_{x \to +\infty}\bar{u}^{\prime}(x)=0, \quad
\lim_{x \to -\infty}\bar{u}^{\prime}(x)=0.
$$
We integrate \eqref{0.3} from $x(>0)$ to $+\infty$, to get
$\bar{u}_x=f(\bar{u})-f(\bar{u}_{+})$, rewrite as
\begin{eqnarray} \label{3.6}
\bar{u}_x=F(\bar{u}):=f(\bar{u})-f(\bar{u}_{+}).
\end{eqnarray}
We notice that $\bar{u}=\bar{u}_+$ is a critical point of the ODE (\ref{3.6}), thus for $\bar{u}$ to be a stable solution of ODE (\ref{3.6}), it is necessary that $F^{\prime}(\bar{u}_+) \leq 0$, i.e., $F^{\prime}(\bar{u}_+)=f^{\prime}(\bar{u}_+)\leq 0$. Let $\phi=\bar{u}$, then (\ref{3.6}) reads
\begin{eqnarray} \label{3.7}
\phi_x=F(\phi):=f(\phi)-f(\bar{u}_{+})
\end{eqnarray}

\begin{lem} \label{Lem3.1}
Consider the initial value problem \eqref{3.7} with $\phi(x_0)=\bar{u}_0$. If we assume that $f^{\prime}(\bar{u}_+)<0$, then there are positive constants $\delta>0$ and $\alpha >0$ that are independent of the choice of the initial time $x_0$ such that the solution $x \mapsto \phi(x)$ of the initial value problem \eqref{3.7} satisfies
 \begin{eqnarray} \label{3.8}
 \left|\phi(x)-\bar{u}_+ \right| \leq \left|\phi(x_0)-\bar{u}_+ \right|e^{-\alpha(x-x_0)}
 \end{eqnarray}
for $x \geq x_0$ whenever $\left|\phi(x_0)-\bar{u}_+\right| \leq \delta$.
\end{lem}
This lemma can be proved with standard stable manifold techniques.

Notice that we can rewrite \eqref{3.6} as
$\bar{u}_x=f(\bar{u})-f(\bar{u}_+)=f^{\prime}(\xi)(\bar{u}-\bar{u}_+)$
so we have the following estimate
$
\left|\bar{u}_x \right| \leq \left|f^{\prime} \right|_{L^{\infty}}\left|\bar{u}(x_0)-\bar{u}_+ \right|e^{-\alpha(x-x_0)}
$.
We then can rephrase Lemma \ref{Lem3.1} as the following

\begin{lem} \label{Lem3.2}
Consider the initial value problem \eqref{3.7} with $\bar{u}(x_0)=\bar{u}_0$. If we assume that $f^{\prime}(\bar{u}_+)<0$, then there are positive constants $\delta>0$ and $\alpha >0$ that are independent of the choice of the initial time $x_0$ such that the solution $x \mapsto \bar{u}(x)$ of the initial value problem \eqref{3.7} satisfies
 \begin{eqnarray*}
 \left|\bar{u}(x)-\bar{u}_+ \right| \leq \left|\bar{u}(x_0)-\bar{u}_+ \right|e^{-\alpha(x-x_0)}, \quad
 \left|\bar{u}_x \right| \leq \left|f^{\prime} \right|_{L^{\infty}}\left|\bar{u}(x_0)-\bar{u}_+ \right|e^{-\alpha(x-x_0)}
 \end{eqnarray*}
for $x \geq x_0$ whenever $\left|\bar{u}(x_0)-\bar{u}_+\right| \leq \delta$.
\end{lem}

Similarly, if we integrate \eqref{0.3} from $-\infty$ to some $x<0$, we get
\begin{eqnarray} \label{3.18}
\bar{u}_x=f(\bar{u})-f(\bar{u}_{-})
\end{eqnarray}
Follow some similar arguments,

\begin{lem} \label{Lem3.3}
Consider the initial value problem \eqref{3.18} with $\bar{u}(x_0)=\bar{u}_0$. If we assume that $f^{\prime}(\bar{u}_-)>0$, then there are positive constants $\delta>0$ and $\alpha >0$ that are independent of the choice of the initial time $x_0$ such that the solution $x \mapsto \bar{u}(x)$ of the initial value problem \eqref{3.18} satisfies
 \begin{eqnarray*}
 \left|\bar{u}(x)-\bar{u}_- \right| \leq \left|\bar{u}(x_0)-\bar{u}_- \right|e^{\alpha(x-x_0)}, \quad
 \left|\bar{u}_x \right| \leq \left|f^{\prime} \right|_{L^{\infty}}\left|\bar{u}(x_0)-\bar{u}_- \right|e^{\alpha(x-x_0)}
 \end{eqnarray*}
for $x \leq x_0$ whenever $\left|\bar{u}(x_0)-\bar{u}_-\right| \leq \delta$.
\end{lem}

Combine Lemma \ref{Lem3.2} and \ref{Lem3.3} we know that if we assume the Lax condition $f^{\prime}(\bar{u}_+)<0<f^{\prime}(\bar{u}_-)$ then,

\begin{pro} \label{Pro3.4}
There are positive constants $\delta>0$ and $\alpha >0$ that are independent of the choice of the initial time $x_0$ such that the stationary wave solution $x \mapsto \bar{u}(x)$ satisfies asymptotic description.
 \begin{eqnarray*}
 \left|\bar{u}(x)-\bar{u}_{\pm} \right| \leq \left|\bar{u}(x_0)-\bar{u}_{\pm} \right|e^{-\alpha\left|x-x_0\right|}, \quad
 \left|\bar{u}_x \right| \leq \left|f^{\prime} \right|_{L^{\infty}}\left|\bar{u}(x_0)-\bar{u}_{\pm} \right|e^{-\alpha\left|x-x_0\right|}
 \end{eqnarray*}
for $x \gtrless x_0$ whenever $\left|\bar{u}(x_0)-\bar{u}_{\pm}\right| \leq \delta$.
\end{pro}

\section{The Gap Lemma}

The "Gap Lemma" of \cite{ZH} consists of the idea of relating the behavior near $x=\pm \infty$ of solutions of \eqref{2.2} to that of solutions of the asymptotic systems \eqref{2.5}, in a manner that is analytic in $\lambda$. In this section, we state the general Gap Lemma for a general equation
\begin{eqnarray} \label{4.1}
W^{\prime}=\mathbb{A}(x;\lambda)W
\end{eqnarray}
with the hypothesis
$
\left| \mathbb{A}-\mathbb{A}_{\pm} \right| = \mathbf{O}(e^{-\alpha |x|}) \text{  as  } x \rightarrow \pm \infty.
$

\begin{pro} \label{Pro4.1}
In \eqref{4.1}, let $\mathbb{A}$ be $C^{0,\alpha}$ in $x$ and analytic in $\lambda$, with
$\left|\mathbb{A}(x;\lambda)-\mathbb{A}_-(\lambda)\right|=\mathbf{O}(e^{-\alpha |x|})$ as $x \to -\infty$ for $\alpha >0$, and
$0<\bar{\alpha}<\alpha$. If $V^{-}(\lambda)$ is an eigenvector of $\mathbb{A}_-$ with eigenvalue $\mu(\lambda)$, both analytic in $\lambda$, then there exists a solution $W(x;\lambda)$ of $(\ref{4.1})$ of form
 $W(x;\lambda)=V(x;\lambda)e^{\mu(\lambda) x}$
where $V$ (hence $W$) is $C^{1,\alpha}$ in $x$ and locally analytic in $\lambda$, and for each $j=0,1,2,\ldots$ satisfies
 \begin{equation} \label{4.4}
 \begin{aligned}
 \left(\frac{\partial}{\partial \lambda}\right)^j V(x;\lambda)
 &=\left(\frac{\partial}{\partial \lambda}\right)^j V^-(\lambda)
 +\mathbf{O}\left(e^{-\bar{\alpha}|x|}\left|\left(\frac{\partial}{\partial \lambda}\right)^j V^-(\lambda)\right| \right)\\
 &=\left(\frac{\partial}{\partial \lambda}\right)^j V^-(\lambda)(1+\mathbf{O}(e^{-\bar{\alpha} |x|}))
 \end{aligned}
 \end{equation}
for $x<0$. Moreover, if $\mathrm{Re} \mu(\lambda)> \mathrm{Re} \tilde{\mu}(\lambda)-\alpha$ for all eigenvalues $\tilde{\mu}(\lambda)$ of $\mathbb{A}_-(\lambda)$, then $W$ is uniquely determined by \eqref{4.4}, and \eqref{4.4} holds for $\bar{\alpha}=\alpha$.
\end{pro}

\section{Construction of the Resolvent Kernel} \label{sectioncon}

We now construct an explicit representation for the resolvent kernel, that is, the Green function $G_\lambda(x,y)$ associated with the elliptic operator $(L-\lambda I)$, defined by
\begin{eqnarray*}
(L-\lambda I)G_\lambda (x,y)=\delta_y(x) I,
\end{eqnarray*}
where $\delta_y$ denotes the Dirac delta distribution centered at $y$. Let $\Lambda$ be as defined in section \ref{sectionasy}. It is a standard fact that both the resolvent $(L-\lambda I)^{-1}$ and the Green function $G_{\lambda}(x,y)$ are meromorphic in $\lambda$ on $\Lambda$, with isolated poles of finite order(See \cite{He}). Using our explicit representation, we will show more, that $G_{\lambda}(x,y)$ in fact admits a meromorphic extension to a sector
$\Omega_{\theta}:=\{\lambda: \mathrm{Re}(\lambda) \geq -\theta_{1}-\theta_{2}|\mathrm{Im}(\lambda)|\}; \theta_1,\theta_2 >0.$

We now start to find the Green function for the operator $L-\lambda I$. On $\Lambda$, the subspaces spanned by
\begin{eqnarray} \label{5.10}
\phi^+(x) &=& W_1^+(x;\lambda) =V_1^+(x;\lambda) e^{\mu_1^+ x},\text{ for } x>0 \text{  and,}\\
\phi^-(x) &=& W_2^-(x;\lambda) =V_2^-(x;\lambda) e^{\mu_2^- x},\text{ for } x<0.
\end{eqnarray}
contain all solutions of \eqref{2.2} \textit{decaying} at $x=\pm \infty$. We denote the complementary subspaces of \textit{growing modes} by the subspace spanned by
\begin{eqnarray} \label{5.13}
\psi^+(x) &=& W_2^+(x;\lambda) =V_2^+(x;\lambda) e^{\mu_2^+ x},\text{ for } x>0 \text{  and,}\\
\psi^-(x) &=& W_1^-(x;\lambda) =V_1^-(x;\lambda) e^{\mu_1^- x},\text{ for } x<0.
\end{eqnarray}
where $\mu_1^+ <0<\mu_2^+$ and $\mu_1^- <0<\mu_2^-$.

Eigenfunctions, decaying at both $\pm \infty$, occur precisely when the subspaces Span$\{\phi^+\}$ and Span$\{\phi^-\}$ intersect. This intersection can be detected by the vanishing of their mutual determinant, or equivalently of the \textit{Evans function}, $D_L(\lambda):=\det(\phi^+,\phi^-)\mid_{x=0}=(\phi^+\wedge\phi^-)\mid_{x=0}$.
We now turn to the representation of the Green function $G_\lambda(x,y)$.

\begin{lem} \label{Lem5.1}
Let $H_\lambda(x,y)$ denote the Green function for the adjoint operator $(L-\lambda I)^\ast$. Then,
$G_\lambda(y,x)=H_\lambda(x,y)^\ast$. In particular, for $x \neq y$, the matrix $z=G_\lambda(x,\cdot)$ satisfies
 \begin{eqnarray} \label{5.16}
 z^{\prime\prime}=-z^{\prime}f^{\prime}(\bar{u})+\lambda z.
 \end{eqnarray}
\end{lem}

\begin{proof}
Notice that we have $(L-\lambda I)G_\lambda(x,y)=\delta_y(x)I$, and $(L-\lambda I)^\ast H_\lambda(x,y)=\delta_y(x)I$. So
 \begin{eqnarray*}
 G_\lambda(x_0,y_0)
 &=& \langle \delta_{x_0}(x) I, G_\lambda(x,y_0) \rangle_{x}\\
 &=& \langle (L-\lambda I)^\ast H_\lambda(x,x_0), G_\lambda(x,y_0) \rangle_{x}\\
 &=& \langle H_\lambda(x,x_0),(L-\lambda I)G_\lambda(x,y_0) \rangle_{x}\\
 &=& \langle H_\lambda(x,x_0), \delta_{y_0}(x) \rangle_{x}\\
 &=& H_\lambda(y_0,x_0)^\ast.
 \end{eqnarray*}
\end{proof}

The equation \eqref{5.16} is an adjoint equation of \eqref{2.1}, written as a first order system by setting $Z=(z,z^\prime)$, it becomes
\begin{eqnarray} \label{5.17}
Z^\prime=Z\tilde{\mathbb{A}}(x;\lambda),
\end{eqnarray}
where
$
\tilde{\mathbb{A}}(x;\lambda):=
  \begin{pmatrix}
   0 & \lambda\\
   1 & -f^{\prime}(\bar{u})
  \end{pmatrix}.
$
The following lemma shows a duality relation between solutions of \eqref{2.2} and solutions of \eqref{5.17}.

\begin{lem} \label{Lem5.2}
$Z$ is a solution of \eqref{5.17} if and only if $Z\mathcal{S}W\equiv$ constant for any solution $W$ of \eqref{2.2}, where
$\mathcal{S}=
  \begin{pmatrix}
   -f^{\prime}(\bar{u}) & 1\\
   -1 & 0
  \end{pmatrix}.$
\end{lem}

\begin{proof}
 \begin{eqnarray*}
 (Z\mathcal{S}W)^{\prime}
 &=& (-zf^{\prime}(\bar{u})w-z^{\prime}w+zw^{\prime})^{\prime}\\
 &=& -z^{\prime}f^{\prime}(\bar{u})w-z(f^{\prime}(\bar{u})w)^{\prime}
     -z^{\prime\prime}w-z^{\prime}w^{\prime}+z^{\prime}w^{\prime}+zw^{\prime\prime}\\
 &=& -z^{\prime}f^{\prime}(\bar{u})w-z(f^{\prime}(\bar{u})w)^{\prime}-z^{\prime\prime}w+zw^{\prime\prime}\\
 &=& z[w^{\prime\prime}-(f^{\prime}(\bar{u})w)^{\prime}]-[z^{\prime\prime}+z^{\prime}f^{\prime}(\bar{u})]w\\
 &=& z[w^{\prime\prime}-(f^{\prime}(\bar{u})w)^{\prime}-\lambda w]-[z^{\prime\prime}+z^{\prime}f^{\prime}(\bar{u})-\lambda z]w\\
 &=& [z^{\prime\prime}+z^{\prime}f^{\prime}(\bar{u})-\lambda z]w
 \end{eqnarray*}
since $w$ is a solution of \eqref{2.1}, thus $z$ is a solution of \eqref{5.16} if and only if $(Z\mathcal{S}W)^{\prime}\equiv 0$.
\end{proof}

Using Lemma \ref{Lem5.2}, we can immediately define dual bases $\tilde{W}_1^{\pm}$ and $\tilde{W}_2^{\pm}$ of solutions to \eqref{5.17} by the relation
\begin{eqnarray} \label{5.19}
\tilde{W}_j^{\pm}\mathcal{S}W_k^{\pm}=\delta_{jk},
\end{eqnarray}
where
$
\delta_{jk}=
 \left\{
  \begin{array}{l l}
    1 & \quad \text{if $j=k$}\\
    0 & \quad \text{otherwise}
  \end{array}\right.
$,
denotes the Kronecker delta function. Note that we have defined the $\tilde{W}_j^{\pm}$ as row vectors, and $\tilde{W}_j^{\pm}$
has the form $\tilde{W}_j^{\pm}(x;\lambda)=\tilde{V}_j^{\pm}(x;\lambda)e^{\tilde{\mu}_j^{\pm}x}$.

According to \eqref{5.19}, we have $\tilde{\mu}_j^{\pm}=-\mu_j^{\pm}$ and $\tilde{V}_j^{\pm}\mathcal{S}V_k^{\pm}=\delta_{jk}$,
for all $\lambda$. In accordance with \eqref{5.10}-\eqref{5.13}, we define the dual subspace spanned by
\begin{eqnarray*}
\tilde{\phi}^+(x) &=& \tilde{W}_1^+(x;\lambda) =\tilde{V}_1^+(x;\lambda) e^{\tilde{\mu}_1^+ x}
=\tilde{V}_1^+(x;\lambda) e^{-\mu_1^+ x},\text{ for } x>0 \\
\tilde{\phi}^-(x) &=& \tilde{W}_2^-(x;\lambda) =\tilde{V}_2^-(x;\lambda) e^{\tilde{\mu}_2^- x}
=\tilde{V}_2^-(x;\lambda) e^{-\mu_2^- x},\text{ for } x<0
\end{eqnarray*}
as the \textit{growing subspace} and the dual subspace spanned by
\begin{eqnarray*}
\tilde{\psi}^+(x) &=& \tilde{W}_2^+(x;\lambda) =\tilde{V}_2^+(x;\lambda) e^{\tilde{\mu}_2^+ x}
=\tilde{V}_2^+(x;\lambda) e^{-\mu_2^+ x},\text{ for } x>0 \\
\tilde{\psi}^-(x) &=& \tilde{W}_1^-(x;\lambda) =\tilde{V}_1^-(x;\lambda) e^{\tilde{\mu}_1^- x}
=\tilde{V}_1^-(x;\lambda) e^{-\mu_1^- x},\text{ for } x<0
\end{eqnarray*}
as the \textit{decaying subspace}. We may define dual exponentially decaying and growing solutions $\tilde{\phi}^{\pm}$ and
$\tilde{\psi}^{\pm}$ via
\begin{eqnarray} \label{5.25}
 \begin{array}{l l}
 \tilde{\phi}^{\pm}\mathcal{S}\phi^{\pm}=1; & \tilde{\phi}^{\pm}\mathcal{S}\psi^{\pm}=0,\\
 \tilde{\psi}^{\pm}\mathcal{S}\phi^{\pm}=0; & \tilde{\psi}^{\pm}\mathcal{S}\psi^{\pm}=1.
 \end{array}
\end{eqnarray}
or written as matrix form
$
 \begin{pmatrix}
 \tilde{\phi}^{\pm}\\
 \tilde{\psi}^{\pm}
 \end{pmatrix}\mathcal{S}
 \left( \phi^{\pm}, \psi^{\pm} \right)
=I.
$

From $(L-\lambda I)G_\lambda(x,y)=\delta_y(x)I$, $(L-\lambda I)^\ast H_\lambda(x,y)=\delta_y(x)I$ and Lemma \ref{Lem5.1} we know that
$\begin{pmatrix}
G_\lambda(x,y)\\
G_{\lambda,x}(x,y)
\end{pmatrix}$
viewed as a function of $x$ satisfies \eqref{2.2}(differentiating with respect to $x$), while $(G_\lambda(x,y),G_{\lambda,y}(x,y))$ viewed as function of $y$ satisfies \eqref{5.17}(differentiating with respect to $y$). Furthermore, note that both $G_\lambda(x,\cdot)$ and $G_\lambda(\cdot,y)$ decay at $\pm\infty$ for $\lambda$ on the resolvent set, since $|(L-\lambda I)^{-1}|<\infty$ and $|(L-\lambda I)^{\ast -1}|<\infty$ imply $\|G_\lambda(\cdot,y)\|_{L^{1}(x)}<\infty$ and $\|G_\lambda(x,\cdot)\|_{L^{1}(y)}<\infty$ respectively. Combining, we have the representation
\begin{eqnarray} \label{5.26}
 \begin{pmatrix}
 G_\lambda & G_{\lambda,y}\\
 G_{\lambda,x} & G_{\lambda,xy}
 \end{pmatrix}
=\left\{
 \begin{array}{l l}
 \phi^+(x;\lambda)m^+(\lambda)\tilde{\psi}^-(y;\lambda) \quad \text{for $x>y$;}\\
 -\phi^-(x;\lambda)m^-(\lambda)\tilde{\psi}^+(y;\lambda) \quad \text{for $x<y$,}
 \end{array}
 \right.
\end{eqnarray}
where numbers $m^{\pm}(\lambda)$ are to be determined.

\begin{lem} \label{Lem5.3}
 \begin{eqnarray*}
  \begin{bmatrix}
   G_\lambda & G_{\lambda,y}\\
   G_{\lambda,x} & G_{\lambda,xy}
  \end{bmatrix}_{(y)}
 =\begin{pmatrix}
   0 & -1\\
   1 & -f^{\prime}(\bar{u})
  \end{pmatrix}
 =\mathcal{S}^{-1},
 \end{eqnarray*}
where $[f(x)]_{(y)}$ denotes the jump in $f(x)$ at $x=y$, and $\mathcal{S}$ is as in $\mathrm{Lemma}$ $\ref{Lem5.2}$.
\end{lem}

\begin{proof}
Expanding $\delta_y(x)=(L-\lambda I)G_\lambda=G_{\lambda,xx}-(f^{\prime}(\bar{u})G_\lambda)_x-\lambda G_\lambda$, and comparing orders of singularity, we find that
$(f^{\prime}(\bar{u})G_\lambda)_x+\lambda G_\lambda=0 \quad \text{and} \quad G_{\lambda,xx}=\delta_{y}(x)$,
giving, respectively,
$[G_{\lambda}]_{(y)}=0 \quad \text{and} \quad [G_{\lambda,x}]_{(y)}=1$.
Note further that we can expand $[G_{\lambda}]_{(y)}$ as
\begin{eqnarray} \label{5.27}
[G_\lambda]_{(y)}=[G_\lambda(\cdot,y)]_{(y)}=G_{\lambda}^{x>y}(y,y)-G_{\lambda}^{x<y}(y,y),
\end{eqnarray}
where $G_{\lambda}^{x>y}$ and $G_{\lambda}^{x<y}$ are the smooth functions denoting the value of $G_{\lambda}$ on the regions $x>y$ and $x<y$, respectively, i.e.
$$G_{\lambda}(x,y)=\left\{
\begin{array}{l l}
G_{\lambda}^{x>y}(x,y) \quad \text{for $x>y$;}\\
G_{\lambda}^{x<y}(x,y) \quad \text{for $x<y$,}
\end{array}\right.
$$
Differentiating \eqref{5.27} in $y$, we obtain
\begin{eqnarray*}
0=\frac{d}{dy}[G_{\lambda}]_{(y)}
&=& G_{\lambda,x}^{x>y}(y,y)+G_{\lambda,y}^{x>y}(y,y)
   -G_{\lambda,x}^{x<y}(y,y)-G_{\lambda,y}^{x<y}(y,y)\\
&=&[G_{\lambda,x}]_{(y)}+[G_{\lambda,y}]_{(y)}
\end{eqnarray*}
hence
$[G_{\lambda,y}]_{(y)}=-1$.

Differentiating $0=[G_{\lambda,x}]_{(y)}+[G_{\lambda,y}]_{(y)}$ a second time, we then find
\begin{eqnarray*}
0&=&[G_{\lambda,xx}]_{(y)}+[G_{\lambda,xy}]_{(y)}+[G_{\lambda,yx}]_{(y)}+[G_{\lambda,yy}]_{(y)}\\
 &=&[G_{\lambda,xx}]_{(y)}+[G_{\lambda,yy}]_{(y)}+2[G_{\lambda,xy}]_{(y)}
\end{eqnarray*}
Solve for $[G_{\lambda,xy}]_{(y)}$ to find that
\begin{eqnarray} \label{5.28}
[G_{\lambda,xy}]_{(y)}=-\frac{1}{2}\left([G_{\lambda,xx}]_{(y)}+[G_{\lambda,yy}]_{(y)}\right).
\end{eqnarray}
Finally, we can determine $[G_{\lambda,xx}]_{(y)}$ and $[G_{\lambda,yy}]_{(y)}$ by solving the ODE \eqref{2.1} and \eqref{5.16} to express
\begin{eqnarray*}
G_{\lambda,xx}&=&f^{\prime}(\bar{u})G_{\lambda,x}+(f^{\prime}(\bar{u}))_x G_\lambda;\\
G_{\lambda,yy}&=&-f^{\prime}(\bar{u})G_{\lambda,y}+\lambda G_{\lambda}.
\end{eqnarray*}
and then
\begin{eqnarray*}
\left[G_{\lambda,xx}\right]_{(y)}&=&f^{\prime}(\bar{u}(y))\left[G_{\lambda,x}\right]_{(y)}
        +f^{\prime\prime}(\bar{u}(y))\bar{u}^{\prime}(y) \left[G_\lambda\right]_{(y)}=f^{\prime}(\bar{u}(y));\\
\left[G_{\lambda,yy}\right]_{(y)}&=&-f^{\prime}(\bar{u}(y))\left[G_{\lambda,y}\right]_{(y)}
        +\lambda \left[G_{\lambda}\right]_{(y)}=f^{\prime}(\bar{u}(y)).
\end{eqnarray*}
so finally \eqref{5.28} gives
$\left[G_{\lambda,xy}\right]_{(y)}=-f^{\prime}(\bar{u}(y))$.

Then
 \begin{eqnarray*}
  \begin{bmatrix}
   G_\lambda & G_{\lambda,y}\\
   G_{\lambda,x} & G_{\lambda,xy}
  \end{bmatrix}_{(y)}
 =\begin{pmatrix}
   0 & -1\\
   1 & -f^{\prime}(\bar{u})
  \end{pmatrix}
 =\mathcal{S}^{-1},
 \end{eqnarray*}
as claimed.
\end{proof}

Combining Lemma \ref{Lem5.3} with \eqref{5.26}, we have
\begin{eqnarray*}
\left(\phi^+(y),\phi^-(y)\right)
\begin{pmatrix}
m^+(\lambda) &             0\\
           0 & m^-(\lambda)
\end{pmatrix}
\begin{pmatrix}
\tilde{\psi}^-(y)\\
\tilde{\psi}^+(y)
\end{pmatrix}
=\mathcal{S}^{-1}, \quad \text{or}
\end{eqnarray*}
\begin{eqnarray*}
 \begin{pmatrix}
  m^+(\lambda) &             0\\
             0 & m^-(\lambda)
 \end{pmatrix}
&=& \left(\phi^+(y),\phi^-(y)\right)^{-1}\mathcal{S}^{-1}
  \begin{pmatrix}
  \tilde{\psi}^-(y)\\
  \tilde{\psi}^+(y)
  \end{pmatrix}^{-1}\\
&=& \left(
  \begin{pmatrix}
  \tilde{\psi}^-(y)\\
  \tilde{\psi}^+(y)
  \end{pmatrix}
  \mathcal{S}
  \left(\phi^+(y),\phi^-(y)\right)
    \right)^{-1}\\
&=& \begin{pmatrix}
    \tilde{\psi}^-\mathcal{S}\phi^+ & \tilde{\psi}^-\mathcal{S}\phi^-\\
    \tilde{\psi}^+\mathcal{S}\phi^+ & \tilde{\psi}^+\mathcal{S}\phi^-
    \end{pmatrix}^{-1}(y)\\
&=& \begin{pmatrix}
    \tilde{\psi}^-\mathcal{S}\phi^+ &                               0\\
                                  0 & \tilde{\psi}^+\mathcal{S}\phi^-
    \end{pmatrix}^{-1}(y)\\
&=& \begin{pmatrix}
    \frac{1}{\tilde{\psi}^-\mathcal{S}\phi^+} &                                         0\\
                                            0 & \frac{1}{\tilde{\psi}^+\mathcal{S}\phi^-}
    \end{pmatrix},
\end{eqnarray*}
so
\begin{eqnarray*}
m^+(\lambda)&=& \frac{1}{\tilde{\psi}^-\mathcal{S}\phi^+}
             =  \frac{1}{\tilde{V}_1^-(x;\lambda)\mathcal{S}V_1^+(x;\lambda)e^{(\mu_1^{+}-\mu_1^{-})x}},\\
m^-(\lambda)&=&\frac{1}{\tilde{\psi}^+\mathcal{S}\phi^-}
             = \frac{1}{\tilde{V}_2^+(x;\lambda)\mathcal{S}V_2^-(x;\lambda)e^{(\mu_2^{-}-\mu_2^{+})x}}.
\end{eqnarray*}
We now introduce the notations,
\begin{equation*}
\begin{aligned}
\Phi:=\left(\phi^+,\phi^-\right), \quad
\Psi:=\left(\psi^-,\psi^+\right), \quad
\tilde{\Psi}:=\begin{pmatrix}
                 \tilde{\psi}^-\\
                 \tilde{\psi}^+
                \end{pmatrix}, \quad
\tilde{\Phi}:=\begin{pmatrix}
                 \tilde{\phi}^+\\
                 \tilde{\phi}^-
                 \end{pmatrix}.
\end{aligned}
\end{equation*}

\begin{pro} \label{Pro5.4}
The resolvent kernel may be expressed as
\begin{eqnarray*}
 \begin{pmatrix}
 G_\lambda & G_{\lambda,y}\\
 G_{\lambda,x} & G_{\lambda,xy}
 \end{pmatrix}
=\left\{
 \begin{array}{l l}
 \phi^+(x;\lambda)m^+(\lambda)\tilde{\psi}^-(y;\lambda) \quad \mathrm{for} \quad x>y;\\
 -\phi^-(x;\lambda)m^-(\lambda)\tilde{\psi}^+(y;\lambda) \quad \mathrm{for} \quad x<y,
 \end{array}
 \right.
\end{eqnarray*}
where
$M(\lambda):=\mathrm{diag}(m^+(\lambda),m^-(\lambda))=\Phi^{-1}(z;\lambda)\mathcal{S}^{-1}(z)\tilde{\Psi}^{-1}(z;\lambda)$.
\end{pro}

From Proposition \ref{Pro5.4}, we obtain the following scattering decomposition,

\begin{cor} \label{Cor5.5}
On $\Lambda\cap\rho(L)$,
\begin{eqnarray}
 \begin{pmatrix}
 G_\lambda & G_{\lambda,y}\\
 G_{\lambda,x} & G_{\lambda,xy}
 \end{pmatrix}
=m^+(\lambda)\phi^+(x;\lambda)\tilde{\psi}^-(y;\lambda)
\end{eqnarray}
for $y \leq 0 \leq x$,
\begin{eqnarray} \label{5.38}
 \begin{pmatrix}
 G_\lambda & G_{\lambda,y}\\
 G_{\lambda,x} & G_{\lambda,xy}
 \end{pmatrix}
=d^+(\lambda)\phi^-(x;\lambda)\tilde{\psi}^-(y;\lambda)+\psi^-(x;\lambda)\tilde{\psi}^-(y;\lambda)
\end{eqnarray}
for $y \leq x \leq 0$, where
\begin{equation} \label{5.40}
m^+=(1,0)(\phi^+,\phi^-)^{-1}\psi^-, \quad d^+=-(0,1)(\phi^+,\phi^-)^{-1}\psi^{-}
\end{equation}
\begin{eqnarray}
 \begin{pmatrix}
 G_\lambda & G_{\lambda,y}\\
 G_{\lambda,x} & G_{\lambda,xy}
 \end{pmatrix}
=-m^-(\lambda)\phi^-(x;\lambda)\tilde{\psi}^+(y;\lambda)
\end{eqnarray}
for $x \leq 0 \leq y$,
\begin{eqnarray}
 \begin{pmatrix}
 G_\lambda & G_{\lambda,y}\\
 G_{\lambda,x} & G_{\lambda,xy}
 \end{pmatrix}
=d^-(\lambda)\phi^-(x;\lambda)\tilde{\psi}^-(y;\lambda)-\phi^-(x;\lambda)\tilde{\phi}^-(y;\lambda)
\end{eqnarray}
for $x \leq y \leq 0$, where
\begin{equation*}
m^-=\tilde{\phi}^- \begin{pmatrix}\tilde{\psi}^- \\ \tilde{\psi}^+ \end{pmatrix}^{-1}
\begin{pmatrix}0 \\ 1\end{pmatrix}
, \quad
d^-(\lambda)=\tilde{\phi}^{-}\begin{pmatrix}\tilde{\psi}^- \\ \tilde{\psi}^+ \end{pmatrix}^{-1}
             \begin{pmatrix}1 \\ 0 \end{pmatrix}
\end{equation*}
\end{cor}

\begin{proof}
We may express $m^+$ using the duality relation \eqref{5.25} as
\begin{equation*}
\begin{aligned}
m^+&=(1,0)(\phi^+,\phi^-)^{-1}\mathcal{S}^{-1}
 \begin{pmatrix}
 \tilde{\psi}^- \\
 \tilde{\phi}^-
 \end{pmatrix}^{-1}
 \begin{pmatrix}
 1 \\
 0
 \end{pmatrix}
 =(1,0)(\phi^+,\phi^-)^{-1}\left(
 \begin{pmatrix}
 \tilde{\psi}^- \\
 \tilde{\phi}^-
 \end{pmatrix}
 \mathcal{S}
 \right)^{-1}
 \begin{pmatrix}
 1 \\
 0
 \end{pmatrix}\\
&=(1,0)(\phi^+,\phi^-)^{-1}(\psi^-,\phi^-)
 \begin{pmatrix}
 1 \\
 0
 \end{pmatrix}
 =(1,0)(\phi^+,\phi^-)^{-1}\psi^-
\end{aligned}
\end{equation*}

Next, expressing $\phi^+(x;\lambda)$ as a linear combination of basis elements at $-\infty$,
$
\phi^+(x;\lambda)=a^+(\lambda)\phi^-(x;\lambda)+b^+(\lambda)\psi^-(x;\lambda)
=(\phi^-,\psi^-)
 \begin{pmatrix}
 a^+ \\
 b^+
 \end{pmatrix}
$,
so we get
$
 \begin{pmatrix}
 a^+ \\
 b^+
 \end{pmatrix}
=(\phi^-,\psi^-)^{-1}\phi^{+}(x;\lambda)
=\begin{pmatrix}
 \tilde{\phi}^- \\
 \tilde{\psi}^-
 \end{pmatrix} \mathcal{S}\phi^+
$,
then we can represent $ \begin{pmatrix}
 G_\lambda & G_{\lambda,y}\\
 G_{\lambda,x} & G_{\lambda,xy}
 \end{pmatrix}$ as
\begin{align*}
 \begin{pmatrix}
 G_\lambda & G_{\lambda,y}\\
 G_{\lambda,x} & G_{\lambda,xy}
 \end{pmatrix}
&= m^+(\lambda)\phi^+(x;\lambda)\tilde{\psi}^-(y;\lambda)\\
&= m^+(\lambda)[a^+(\lambda)\phi^-(x;\lambda)+b^+(\lambda)\psi^-(x;\lambda)]\tilde{\psi}^-(y;\lambda)\\
&= a^+(\lambda)m^+(\lambda)\phi^-(x;\lambda)\tilde{\psi}^-(y;\lambda)
  +b^+(\lambda)m^+(\lambda)\psi^-(x;\lambda)\tilde{\psi}^-(y;\lambda)\\
&= d^+(\lambda)\phi^-(x;\lambda)\tilde{\psi}^-(y;\lambda)
  +e^+(\lambda)\psi^-(x;\lambda)\tilde{\psi}^-(y;\lambda)
\end{align*}
where $d^+=a^+ m^+$ and $e^+=b^+ m^+$, and can be computed as follows,
\begin{align*}
 \begin{pmatrix}
 d^+ \\
 e^+
 \end{pmatrix}
&=
 \begin{pmatrix}
 a^+ \\
 b^+
 \end{pmatrix}m^+
 =
 \begin{pmatrix}
 a^+ \\
 b^+
 \end{pmatrix}(1,0)(\phi^+,\phi^-)^{-1}\psi^-\\
&= \begin{pmatrix}
 \tilde{\phi}^- \\
 \tilde{\psi}^-
 \end{pmatrix} \mathcal{S}\phi^+ (1,0)(\phi^+,\phi^-)^{-1}\psi^-
 =\begin{pmatrix}
 \tilde{\phi}^- \\
 \tilde{\psi}^-
 \end{pmatrix} \mathcal{S} (\phi^+,0)(\phi^+,\phi^-)^{-1}\psi^-\\
&= (\phi^-,\psi^-)^{-1} (\phi^+,0)(\phi^+,\phi^-)^{-1}\psi^-\\
&= (\phi^-,\psi^-)^{-1}\left( I_2-(0,\phi^-)(\phi^+,\phi^-)^{-1}\right)\psi^-\\
&= (\phi^-,\psi^-)^{-1}\psi^{-} - (\phi^-,\psi^-)^{-1}(0,\phi^-)(\phi^+,\phi^-)^{-1}\psi^-\\
&= (\phi^-,\psi^-)^{-1}(\phi^-,\psi^-)\begin{pmatrix} 0 \\ 1 \end{pmatrix}
   -(\phi^-,\psi^-)^{-1}(\phi^-,\psi^-)\begin{pmatrix}0 & 1 \\ 0 & 0 \end{pmatrix}(\phi^+,\phi^-)^{-1}\psi^-\\
&= \begin{pmatrix} 0 \\ 1 \end{pmatrix} - \begin{pmatrix}0 & 1 \\ 0 & 0 \end{pmatrix}(\phi^+,\phi^-)^{-1}\psi^-
\end{align*}
yielding \eqref{5.38} and \eqref{5.40}.
\end{proof}

\section{Low-frequency expansions}

\begin{lem}
For $\lambda \in \Lambda$, the matrix $\mathbb{A}_{\pm}(\lambda)$ in \eqref{2.5} has eigenvalues
$\mu_1^{\pm}(\lambda)<0<\mu_2^{\pm}(\lambda)$ (with ordering referring to real parts) such that the eigenspaces
$S^{\pm}(\lambda)$ and $U^{\pm}(\lambda)$ associated to $\mu_1^{\pm}(\lambda)$ and $\mu_2^{\pm}(\lambda)$ respectively depend
analytically on $\lambda$. Furthermore, for each $j=1,2$, there is an analytic extension of $\mu_j^{\pm}(\lambda)$ to some
neighborhood $N$ of $\lambda=0$. For $\lambda \in N$ there also exists an analytic choice of an individual eigenvector
$V_j^{\pm}=\begin{pmatrix} v_j^{\pm}\\ \mu_j^{\pm}v_j^{\pm} \end{pmatrix}$ corresponding to each eigenvalue $\mu_j^{\pm}(\lambda)$, and they satisfy the following asymptotic descriptions:
\begin{eqnarray*}
\mu_1^{\pm}(\lambda)
=\left\{
 \begin{array}{l l}
 -\frac{\lambda}{f^{\prime}(\bar{u}_{\pm})}+\frac{\lambda ^2}{(f^{\prime}(\bar{u}_{\pm}))^3}+\mathbf{O}(\lambda^3),& \quad \mathrm{if} \quad f^{\prime}(\bar{u}_{\pm})>0;\\
  & \\
 f^{\prime}(\bar{u}_{\pm})+\mathbf{O}(\lambda),& \quad \mathrm{if} \quad f^{\prime}(\bar{u}_{\pm})<0,
 \end{array}
 \right. \\
\mu_2^{\pm}(\lambda)
=\left\{
 \begin{array}{l l}
 -\frac{\lambda}{f^{\prime}(\bar{u}_{\pm})}+\frac{\lambda ^2}{(f^{\prime}(\bar{u}_{\pm}))^3}+\mathbf{O}(\lambda^3),& \quad \mathrm{if} \quad f^{\prime}(\bar{u}_{\pm})<0;\\
  & \\
 f^{\prime}(\bar{u}_{\pm})+\mathbf{O}(\lambda),& \quad \mathrm{if} \quad f^{\prime}(\bar{u}_{\pm})>0,
 \end{array}
 \right.
\end{eqnarray*}
\begin{eqnarray*}
V_1^{\pm}(\lambda)
=\left\{
 \begin{array}{l l}
 \begin{pmatrix} 1+\mathbf{O}(\lambda) \\ -\frac{\lambda}{f^{\prime}(\bar{u}_{\pm})}+\mathbf{O}(\lambda^2) \end{pmatrix},& \quad \mathrm{if} \quad f^{\prime}(\bar{u}_{\pm})>0;\\
  & \\
 \begin{pmatrix} 1 \\ f^{\prime}(\bar{u}_{\pm}) \end{pmatrix}+\mathbf{O}(\lambda),& \quad \mathrm{if} \quad f^{\prime}(\bar{u}_{\pm})<0,
 \end{array}
 \right. \\
V_2^{\pm}(\lambda)
=\left\{
 \begin{array}{l l}
 \begin{pmatrix} 1+\mathbf{O}(\lambda) \\ -\frac{\lambda}{f^{\prime}(\bar{u}_{\pm})}+\mathbf{O}(\lambda^2) \end{pmatrix},& \quad \mathrm{if} \quad f^{\prime}(\bar{u}_{\pm})<0;\\
  & \\
 \begin{pmatrix} 1 \\ f^{\prime}(\bar{u}_{\pm}) \end{pmatrix}+\mathbf{O}(\lambda),& \quad \mathrm{if} \quad f^{\prime}(\bar{u}_{\pm})>0,
 \end{array}
 \right.
\end{eqnarray*}

The spectral projection operators $P_{S^{\pm}(\lambda)}$ and $P_{U^{\pm}(\lambda)}$ associated to the subspaces
$S^{\pm}(\lambda)$ and $U^{\pm}(\lambda)$ have analytic extensions to the neighborhood
$\Omega=\{\lambda: \mathrm{Re} \lambda >0 \}\cup N$.
\end{lem}

Since we have assumed the Lax condition, $f^{\prime}(\bar{u}_+)<0<f^{\prime}(\bar{u}_-)$,
so the results in above lemma simplify to read
\begin{align*}
\mu_1^{+}(\lambda)
=f^{\prime}(\bar{u}_{+})+\mathbf{O}(\lambda),& \quad
\mu_2^{+}(\lambda)
=-\frac{\lambda}{f^{\prime}(\bar{u}_{+})}+\frac{\lambda ^2}{(f^{\prime}(\bar{u}_{+}))^3}+\mathbf{O}(\lambda^3), \\
\mu_1^{-}(\lambda)
=-\frac{\lambda}{f^{\prime}(\bar{u}_{-})}+\frac{\lambda ^2}{(f^{\prime}(\bar{u}_{-}))^3}+\mathbf{O}(\lambda^3),& \quad
\mu_2^{-}(\lambda)
=f^{\prime}(\bar{u}_{-})+\mathbf{O}(\lambda),\\
V_1^{+}(\lambda)
=\begin{pmatrix} 1 \\ f^{\prime}(\bar{u}_{+}) \end{pmatrix}+\mathbf{O}(\lambda),& \quad
V_2^{+}(\lambda)
=\begin{pmatrix} 1+\mathbf{O}(\lambda) \\ -\frac{\lambda}{f^{\prime}(\bar{u}_{+})}+\mathbf{O}(\lambda^2) \end{pmatrix},\\
V_1^{-}(\lambda)
=\begin{pmatrix} 1+\mathbf{O}(\lambda) \\ -\frac{\lambda}{f^{\prime}(\bar{u}_{-})}+\mathbf{O}(\lambda^2) \end{pmatrix},& \quad V_2^{-}(\lambda)
=\begin{pmatrix} 1 \\ f^{\prime}(\bar{u}_{-}) \end{pmatrix}+\mathbf{O}(\lambda),
\end{align*}

\begin{lem} \label{Lem6.2}
For $\lambda \in \Omega \cap \{\lambda: |\lambda|<\delta\}$ and $\delta$ sufficiently small, there exist solutions
$W_j^{\pm}(x;\lambda)$ of \eqref{2.2}, $(j=1,2)$, $C^1$ in $x$ and analytic in $\lambda$, satisfying
\begin{align*}
W_j^{\pm}(x;\lambda)&=V_j^{\pm}(x;\lambda)e^{\mu_j^{\pm}(\lambda)x}, \\
\left(\frac{\partial}{\partial \lambda}\right)^k V_j^{\pm}(x;\lambda)&=\left(\frac{\partial}{\partial \lambda}\right)^k V_j^{\pm}(\lambda)
+\mathbf{O}\left(e^{-\tilde{\alpha}|x|}\left|\left(\frac{\partial}{\partial \lambda}\right)^k V_j^{\pm}(\lambda)\right|\right)
\end{align*}
for any $k\geq 0$ and $0<\tilde{\alpha}<\alpha$, where $\alpha$ is the rate of decay given in $\mathrm{Proposition}$ $\ref{Pro3.4}$,
$\mu_j^{\pm}(\lambda)$ and $V_j^{\pm}(\lambda)$ are as above, and $\mathbf{O}(\cdot)$ depends only on $k, \tilde{\alpha}$.
\end{lem}

\begin{proof}
This is a direct consequence of the Gap Lemma (Proposition \ref{Pro4.1}).
\end{proof}

Now we can classify our forward and dual modes in a greater detail:

The \textit{fast growing modes of the dual problem} \eqref{5.17} are
\begin{equation}
\tilde{\phi}^+(x) = \tilde{W}_1^+(x;\lambda) =\tilde{V}_1^+(x;\lambda) e^{-\mu_1^+(\lambda) x}, \quad
\tilde{\phi}^-(x) = \tilde{W}_2^-(x;\lambda) =\tilde{V}_2^-(x;\lambda) e^{-\mu_2^-(\lambda) x},
\end{equation}
the \textit{slow decaying modes of the dual problem} \eqref{5.17} are
\begin{equation} \label{6.17}
\tilde{\psi}^+(x) = \tilde{W}_2^+(x;\lambda) =\tilde{V}_2^+(x;\lambda) e^{-\mu_2^+(\lambda) x}, \quad
\tilde{\psi}^-(x) = \tilde{W}_1^-(x;\lambda) =\tilde{V}_1^-(x;\lambda) e^{-\mu_1^-(\lambda) x},
\end{equation}
the \textit{fast decaying modes of the forward problem} \eqref{2.2} are
\begin{equation}
\phi^+(x) = W_1^+(x;\lambda) =V_1^+(x;\lambda) e^{\mu_1^+(\lambda) x}, \quad
\phi^-(x) = W_2^-(x;\lambda) =V_2^-(x;\lambda) e^{\mu_2^-(\lambda) x},
\end{equation}
the \textit{slow growing modes of the forward problem} \eqref{2.2} are
\begin{equation}
\psi^+(x) = W_2^+(x;\lambda) =V_2^+(x;\lambda) e^{\mu_2^+(\lambda) x}, \quad
\psi^-(x) = W_1^-(x;\lambda) =V_1^-(x;\lambda) e^{\mu_1^-(\lambda) x}.
\end{equation}

Specifically, due to the special, conservative structure of the underlying evolution equations, the adjoint eigenvalue equation (dual problem \eqref{5.17}) at $\lambda=0$ can be written as
$\tilde{W}^{\prime}=\tilde{W}\begin{pmatrix} 0 & 0\\ 1 & -f^{\prime}(\bar{u}) \end{pmatrix}$,
so it admits a $1$-dimensional subspace of constant solutions $\tilde{W}\equiv \left( c,0 \right)$,
where $c$ is a constant. Thus, at $\lambda=0$, we may choose, by appropriate change of coordinates if necessary, to have slow decaying dual modes $\tilde{\psi}^{\pm}(x)$(\eqref{6.17}) identically constant. Because when we let $\lambda=0$ in \eqref{6.17}, we have
$\tilde{\psi}^+(x) = \tilde{W}_2^+(x;0) =\tilde{V}_2^+(x;0) e^{-\mu_2^+(0) x}=\tilde{V}_2^+(x;0)$, and
$\tilde{\psi}^-(x) = \tilde{W}_1^-(x;0) =\tilde{V}_1^-(x;0) e^{-\mu_1^-(0) x}=\tilde{V}_1^-(x;0)$,
and we can choose $\tilde{V}_2^+(x;0)\equiv \text{constant}$ and $\tilde{V}_1^-(x;0)\equiv \text{constant}$ according to the above observation.

\begin{lem}
With the above choice of bases at $\lambda=0$, and for $\lambda \in \Omega \cap \{\lambda: |\lambda|<\delta\}$ and $\delta$ sufficiently small, slow decaying dual modes $\tilde{W}_2^+(y;\lambda)$ and $\tilde{W}_1^-(y;\lambda)$ satisfy
 \begin{eqnarray}
 \tilde{W}_j^{\pm}(y;\lambda)=e^{-\mu_j^{\pm}(\lambda)y}\tilde{V}_j^{\pm}(0)+\lambda \tilde{\Theta}_j^{\pm}(y;\lambda),
 \end{eqnarray}
where
 \begin{equation} \label{6.25}
 \left| \tilde{\Theta}_j^{\pm} \right| \leq C\left| e^{-\mu_j^{\pm}(\lambda)y} \right|, \quad
 \left| \left(\frac{\partial}{\partial y}\right)\tilde{\Theta}_j^{\pm} \right|
 \leq C\left| e^{-\mu_j^{\pm}(\lambda)y} \right|\left( |\lambda|+ e^{-\alpha |y|} \right),
 \end{equation}
$\alpha >0$ is the rate of decay given in $\mathrm{Proposition}$ $\ref{Pro3.4}$, as $y \to \pm\infty$, and
$\tilde{V}_j^{\pm}(y;0)=\tilde{V}_j^{\pm}(0)\equiv \text{constant}$.
Similarly, fast decaying forward modes $W_1^+(x;\lambda)$ and $W_2^-(x;\lambda)$ satisfy
 \begin{eqnarray} \label{6.27}
 W_j^{\pm}(x;\lambda) = W_j^{\pm}(x;0) + \lambda \Theta_j^{\pm}(x;\lambda),
 \end{eqnarray}
where
 \begin{equation} \label{6.28}
 \left| \Theta_j^{\pm} \right| \leq C e^{-\alpha |x|}, \quad
 \left| \left(\frac{\partial}{\partial x}\right)\Theta_j^{\pm} \right| \leq C e^{-\alpha |x|},
 \end{equation}
as $x \to \pm\infty$.
\end{lem}

\begin{proof}
First, let us consider the augmented variables
\begin{equation*}
\begin{aligned}
\tilde{\mathbb{W}}_j^{\pm}(y;\lambda):
= \left(\tilde{W}_j^{\pm}, \tilde{W}_j^{\pm\prime} \right)(y;\lambda)
= e^{-\mu_j^{\pm}(\lambda)y}\tilde{\mathbb{V}}_j^{\pm}(y;\lambda)
= e^{-\mu_j^{\pm}(\lambda)y}\left( \tilde{V}_j^{\pm}, -\mu_j^{\pm}\tilde{V}_j^{\pm}+\tilde{V}_j^{\pm\prime} \right)(y;\lambda)
\end{aligned}
\end{equation*}
and
\begin{equation*}
\begin{aligned}
\mathbb{W}_j^{\pm}(x;\lambda):
= \begin{pmatrix}W_j^{\pm} \\ W_j^{\pm\prime} \end{pmatrix}(x;\lambda)
= e^{\mu_j^{\pm}(\lambda)x}\mathbb{V}_j^{\pm}(x;\lambda)
= e^{\mu_j^{\pm}(\lambda)x}\begin{pmatrix} V_j^{\pm} \\ \mu_j^{\pm}V_j^{\pm}+V_j^{\pm\prime}\end{pmatrix}(x;\lambda).
\end{aligned}
\end{equation*}
Note that since $W_j^{\pm}(x;\lambda)$ satisfies $W^{\prime}=\mathbb{A}(x;\lambda)W$, so
$W^{\prime\prime}=\mathbb{A}^{\prime}W+\mathbb{A}W^{\prime}$ and so
$\mathbb{W}(x;\lambda)=\begin{pmatrix} W \\ W^{\prime} \end{pmatrix}(x;\lambda)$
satisfies
$\mathbb{W}^{\prime}=\begin{pmatrix}\mathbb{A} & 0 \\ \mathbb{A}^{\prime} & \mathbb{A} \end{pmatrix}(x;\lambda) \mathbb{W}$.
Let $x \to \pm\infty$ in the coefficient matrix, we get the limiting equation
$\mathbb{W}^{\prime}=\begin{pmatrix}\mathbb{A}_{\pm} & 0 \\ 0 & \mathbb{A}_{\pm} \end{pmatrix}(\lambda) \mathbb{W}$.
If $\bar{W}_j^{\pm}(x;\lambda)=V_j^{\pm}(\lambda)e^{\mu_j^{\pm}(\lambda)x}$ is a solution of $W^{\prime}=\mathbb{A}_{\pm}W$,
then
$
 \bar{\mathbb{W}}_j^{\pm}(x;\lambda)
 =\begin{pmatrix} \bar{W}_j^{\pm}(x;\lambda) \\ \bar{W}_j^{\pm\prime}(x;\lambda) \end{pmatrix}
 =e^{\mu_j^{\pm}(\lambda)x}
  \begin{pmatrix} V_j^{\pm}(\lambda) \\ \mu_j^{\pm}(\lambda)V_j^{\pm}(\lambda)  \end{pmatrix}
$
is a solution of
$\mathbb{W}^{\prime}=\begin{pmatrix}\mathbb{A}_{\pm} & 0 \\ 0 & \mathbb{A}_{\pm} \end{pmatrix}(\lambda) \mathbb{W}$.

Now we apply the Gap Lemma to obtain bounds
\begin{equation}
\tilde{\mathbb{W}}_j^{\pm}(y;\lambda)
=\tilde{\mathbb{V}}_j^{\pm}(y;\lambda)e^{-\mu_j^{\pm}(\lambda)y},
\end{equation}
\begin{equation} \label{6.31}
\left(\frac{\partial}{\partial \lambda}\right)^{k}\tilde{\mathbb{V}}_j^{\pm}(y;\lambda)
=\left(\frac{\partial}{\partial \lambda}\right)^{k} \tilde{\mathbb{V}}_j^{\pm}(\lambda)
   +\mathbf{O} \left( e^{-\tilde{\alpha}|y|}\left|\tilde{\mathbb{V}}_j^{\pm}(\lambda) \right| \right), y \gtrless 0,
\end{equation}
and
\begin{equation}
\mathbb{W}_j^{\pm}(x;\lambda)
=\mathbb{V}_j^{\pm}(x;\lambda)e^{\mu_j^{\pm}(\lambda)x},
\end{equation}
\begin{equation}
\left(\frac{\partial}{\partial \lambda}\right)^{k}\mathbb{V}_j^{\pm}(x;\lambda)
=\left(\frac{\partial}{\partial \lambda}\right)^{k} \mathbb{V}_j^{\pm}(\lambda)
 +\mathbf{O} \left( e^{-\tilde{\alpha}|x|}\left|\mathbb{V}_j^{\pm}(\lambda) \right| \right), x \gtrless 0,
\end{equation}
$0<\tilde{\alpha}<\alpha$, analogous to Lemma \ref{Lem6.2}, valid for $\lambda \in \Omega \cap \{\lambda: |\lambda|<\delta\}$, where
$\tilde{\mathbb{V}}_j^{\pm}(\lambda)=\left(\tilde{V}_j^{\pm}(\lambda), -\mu_j^{\pm}(\lambda)\tilde{V}_j^{\pm}(\lambda) \right)$
and
$\mathbb{V}_j^{\pm}(\lambda)=\begin{pmatrix} V_j^{\pm}(\lambda) \\ \mu_j^{\pm}(\lambda)V_j^{\pm}(\lambda) \end{pmatrix}$.

By Taylor's Theorem with differential remainder, applied to $\tilde{\mathbb{V}}_j^{\pm}(y;\lambda)$ with respect to $\lambda$, we have:
\begin{equation} \label{6.34}
\tilde{\mathbb{W}}_j^{\pm}(y;\lambda)=e^{-\mu_j^{\pm}(\lambda)y}
\left( \tilde{\mathbb{V}}_j^{\pm}(y;0)+\lambda\left(\frac{\partial}{\partial\lambda}\right)\tilde{\mathbb{V}}_j^{\pm}(y;0)
+\frac{1}{2} \lambda^2 \left(\frac{\partial}{\partial\lambda}\right)^2 \tilde{\mathbb{V}}_j^{\pm}(y;\lambda_{\ast}) \right)
\end{equation}
for some $\lambda_{\ast}$ on the ray from $0$ to $\lambda$, where, recall,
$\left(\frac{\partial}{\partial\lambda}\right)\tilde{\mathbb{V}}_j^{\pm}(y;\cdot)$ and
$\left(\frac{\partial}{\partial\lambda}\right)^2 \tilde{\mathbb{V}}_j^{\pm}(y;\cdot)$ are uniformly bounded in $L^{\infty}[0,\pm\infty]$ for $\lambda \in \Omega \cap \{\lambda: |\lambda|<\delta\}$.
Together with the choice $\tilde{V}_j^{\pm}(y;0)\equiv \text{constant}$, this immediately gives the first bound in \eqref{6.25}.

Applying now the bound \eqref{6.31} with $k=1$, we may expand the second coordinate of \eqref{6.34} as
\begin{eqnarray*}
\left(\frac{\partial}{\partial y}\right)\tilde{W}_j^{\pm}(y;\lambda)
&=&e^{-\mu_j^{\pm}(\lambda)y}\left(-\mu_j^{\pm}(0)\tilde{V}_j^{\pm}(y;0)+ \tilde{V}_j^{\pm\prime}(y;0)\right.\\
& &\left. -\lambda \left(\left(\frac{\partial}{\partial\lambda}\right)(\mu_j^{\pm}\tilde{V}_j^{\pm})(0)+\mathbf{O}(e^{-\alpha|y|}) \right) +\mathbf{O}(\lambda^2)\right)\\
&=&e^{-\mu_j^{\pm}(\lambda)y} \left( -\lambda \left(\left(\frac{\partial}{\partial\lambda}\right)\mu_j^{\pm}(0)\tilde{V}_j^{\pm}(0)+\mathbf{O}(e^{-\alpha|y|}) \right) +\mathbf{O}(\lambda^2) \right),
\end{eqnarray*}
and subtracting off the corresponding Taylor expansion
\begin{align*}
\left(\frac{\partial}{\partial y}\right)\left(e^{-\mu_j^{\pm}(\lambda)y}\tilde{V}_j^{\pm}(y;0)\right)
&=-\mu_j^{\pm}(\lambda)e^{-\mu_j^{\pm}(\lambda)y}\tilde{V}_j^{\pm}(y;0)\\
&=e^{-\mu_j^{\pm}(\lambda)y}\left( -\mu_j^{\pm}(0)\tilde{V}_j^{\pm}(y;0)-\lambda\left(\frac{\partial}{\partial\lambda}\right)
  \mu_j^{\pm}(0)\tilde{V}_j^{\pm}(y;0)+\mathbf{O}(\lambda^2)  \right)\\
&=e^{-\mu_j^{\pm}(\lambda)y}\left( -\lambda\left(\frac{\partial}{\partial\lambda}\right)
  \mu_j^{\pm}(0)\tilde{V}_j^{\pm}(0)+\mathbf{O}(\lambda^2) \right)
\end{align*}
we obtain
\begin{eqnarray*}
\lambda\left(\frac{\partial}{\partial y}\right)\tilde{\Theta}_j^{\pm}(y;\lambda)
=e^{-\mu_j^{\pm}(\lambda)y}\left( \lambda\mathbf{O}(e^{-\alpha|y|})+\mathbf{O}(\lambda^2) \right)
\end{eqnarray*}
so
\begin{eqnarray*}
\left(\frac{\partial}{\partial y}\right)\tilde{\Theta}_j^{\pm}(y;\lambda)
=e^{-\mu_j^{\pm}(\lambda)y}\left( \mathbf{O}(e^{-\alpha|y|})+\mathbf{O}(\lambda) \right)
\leq C\left| e^{-\mu_j^{\pm}(\lambda)y} \right|\left( |\lambda|+ e^{-\alpha |y|} \right),
\end{eqnarray*}
as claimed.

To obtain the estimates \eqref{6.28} we use the Taylor's theorem on
$W_j^{\pm}(x;\lambda)=W_j^{\pm}(x;0)+\lambda \frac{\partial}{\partial\lambda}W_j^{\pm}(x;\lambda_{\ast})$,
where $\lambda_{\ast}$ is some number on the ray from $0$ to $\lambda$. The
$\frac{\partial}{\partial\lambda}W_j^{\pm}(x;\lambda_{\ast})$ term can be written as
$
\frac{\partial}{\partial\lambda}W_j^{\pm}(x;\lambda_{\ast})
=\left(\frac{d\mu_j^{\pm}(\lambda_{\ast})}{d\lambda}\right)xe^{\mu_j^{\pm}(\lambda_{\ast})x}V_j^{\pm}(x;\lambda_{\ast})
+e^{\mu_j^{\pm}(\lambda_{\ast})x}\left[\frac{\partial}{\partial \lambda}V_j^{\pm}(x;\lambda_{\ast})\right]
$
together with the observation that $|xe^{-(\alpha+\varepsilon) |x|}|\leq Ce^{-\alpha |x|}$ for some $\varepsilon>0$ small,
we can derive \eqref{6.28}.
\end{proof}

We now turn to the estimation of scattering coefficients $m^{\pm}$, $d^{\pm}$ in Corollary \ref{Cor5.5}.

\begin{lem}
For $|\lambda|$ sufficiently small, $|m^{\pm}|, |d^{\pm}| \leq C\lambda^{-1}$.
Moreover, $\mathrm{Res}_{\lambda=0}m^{\pm}=\mathrm{Res}_{\lambda=0}d^{\pm}$.
\end{lem}

\begin{proof}
Expanding $m^+=(1,0)(\phi^+,\phi^-)^{-1}\psi^-$ using Cramer's rule, and setting $x=y=0$ in $\phi^{\pm}$ and $\psi^-$, we obtain
 \begin{eqnarray*}
 m^+
 &=&(1,0)(\phi^+,\phi^-)^{-1}\psi^-
 =\frac{1}{D}(1,0)(\phi^+,\phi^-)^{\mathrm{adj}}\psi^-=\frac{1}{D}c^+\\
 &=&\frac{\det(\psi^-,\phi^-)}{\det(\phi^+,\phi^-)}=\frac{\det(\psi^-,\phi^-)}{D}
 \end{eqnarray*}
where $D=\det(\phi^+,\phi^-)$, $c^+=(1,0)(\phi^+,\phi^-)^{\mathrm{adj}}\psi^-$,
so
 \begin{eqnarray*}
 c^+
 &=&Dm^+=\det(\psi^-,\phi^-)\\
 &=&\det\left( V_1^{-}(x;\lambda)e^{\mu_1^{-}(\lambda)x},V_2^{-}(y;\lambda)e^{\mu_2^{-}(\lambda)y} \right)\\
 &=&\det\left( \left[V_1^{-}(\lambda)+\mathbf{O}(e^{-\tilde{\alpha}|x|})\right]e^{\left(-\frac{\lambda}{f^{\prime}(\bar{u}_{-})}+\frac{\lambda ^2}{(f^{\prime}(\bar{u}_{-}))^3}+\mathbf{O}(\lambda^3)\right)x},\right.\\
 & & \left.V_2^{-}(y,0)e^{f^{\prime}(\bar{u}_{-})y}
     +\lambda\mathbf{O}(e^{-\alpha|y|}) \right)\\
 &=&\det(V_1^{-}(\lambda)+\mathbf{O}(1),V_2^{-}(0)+\lambda\mathbf{O}(1))\\
 &=&\det\begin{pmatrix} 1+\mathbf{O}(\lambda)+\mathbf{O}(1) & 1+\mathbf{O}(\lambda) \\ -\frac{\lambda}{f^{\prime}(\bar{u}_{-})}+\mathbf{O}(\lambda^2)+\mathbf{O}(1) & f^{\prime}(\bar{u}_{-})+\mathbf{O}(\lambda) \end{pmatrix}\\
 &=&f^{\prime}(\bar{u}_{-})+\frac{\lambda}{f^{\prime}(\bar{u}_{-})}+\mathbf{O}(1)+\mathbf{O}(\lambda)+\mathbf{O}(\lambda^2)
    +\mathbf{O}(\lambda^3)\\
 &=&f^{\prime}(\bar{u}_{-})+\mathbf{O}(1)
 \end{eqnarray*}
and
$D=\det(\phi^+,\phi^-)=\det(V_1^{+}(0;\lambda),V_2^{-}(0;\lambda))=\mathbf{O}(\lambda)$
together we have
$m^+=\frac{c^+}{D}=\frac{f^{\prime}(\bar{u}_{-})+\mathbf{O}(1)}{\mathbf{O}(\lambda)}$
or equivalently $|m^+|\leq C/\lambda$.
\end{proof}

\begin{pro} \label{Pro6.5}
For $\lambda \in \Omega \cap \{\lambda: |\lambda|<\delta\}$ and $\delta$ sufficiently small, the resolvent kernel $G_\lambda$ has a meromorphic extension onto $\{\lambda: |\lambda|<\delta\}$, which may be decomposed as
 \begin{eqnarray}
 G_\lambda=E_\lambda+S_\lambda +R_\lambda=E_\lambda+S_\lambda +R_\lambda^E+R_\lambda^S
 \end{eqnarray}

For $y \leq 0 \leq x$:
 \begin{equation} \label{6.38}
 \begin{pmatrix}
 E_\lambda & E_{\lambda,y}\\
 E_{\lambda,x} & E_{\lambda,xy}
 \end{pmatrix}:=
 C_1\lambda^{-1} W_1^{+}(x;0)\tilde{V}_1^{-}(0)e^{\left(\frac{\lambda}{f^{\prime}(\bar{u}_{-})}-\frac{\lambda ^2}{(f^{\prime}(\bar{u}_{-}))^3}\right)y}
 \end{equation}
 \begin{eqnarray}
 S_\lambda=0
 \end{eqnarray}
and
 \begin{equation}
 \begin{pmatrix}
 R_\lambda & R_{\lambda,y}\\
 R_{\lambda,x} & R_{\lambda,xy}
 \end{pmatrix}:=
 e^{-\alpha|x|}e^{\left(\frac{\lambda}{f^{\prime}(\bar{u}_{-})}-\frac{\lambda ^2}{(f^{\prime}(\bar{u}_{-}))^3}
 \right)y}\left(\mathbf{O}(\lambda^{-1})\left(e^{\mathbf{O}(\lambda^3)y}-1\right)+\mathbf{O}(1)\right)
 \end{equation}

For $y \leq x \leq 0$:
 \begin{equation} \label{6.41}
 \begin{pmatrix}
 E_\lambda & E_{\lambda,y}\\
 E_{\lambda,x} & E_{\lambda,xy}
 \end{pmatrix}:=
 C_2\lambda^{-1} W_2^{-}(x;0)\tilde{V}_1^{-}(0)e^{\left(\frac{\lambda}{f^{\prime}(\bar{u}_{-})}-\frac{\lambda ^2}{(f^{\prime}(\bar{u}_{-}))^3}\right)y}
 \end{equation}
 \begin{equation} \label{6.42}
 \begin{pmatrix}
 S_\lambda & S_{\lambda,y}\\
 S_{\lambda,x} & S_{\lambda,xy}
 \end{pmatrix}:=
 V_1^{-}(0)\tilde{V}_1^{-}(0)e^{\left(-\frac{\lambda}{f^{\prime}(\bar{u}_-)}+\frac{\lambda^2}{(f^{\prime}(\bar{u}_-))^3}\right)(x-y)}
 \end{equation}
 \begin{equation} \label{6.43}
 \begin{pmatrix}
 R_\lambda^E & R_{\lambda,y}^E\\
 R_{\lambda,x}^E & R_{\lambda,xy}^E
 \end{pmatrix}:=
 e^{-\alpha|x|}e^{\left(\frac{\lambda}{f^{\prime}(\bar{u}_{-})}-\frac{\lambda ^2}{(f^{\prime}(\bar{u}_{-}))^3}
    \right)y}\left(\mathbf{O}(\lambda^{-1})\left(e^{\mathbf{O}(\lambda^3)y}-1\right)+\mathbf{O}(1)\right)
 \end{equation}
 \begin{equation} \label{6.44}
 \begin{pmatrix}
 R_\lambda^S & R_{\lambda,y}^S\\
 R_{\lambda,x}^S & R_{\lambda,xy}^S
 \end{pmatrix}:=
 e^{\left(-\frac{\lambda}{f^{\prime}(\bar{u}_-)}+\frac{\lambda^2}{(f^{\prime}(\bar{u}_-))^3}\right)(x-y)}
    \left[\mathbf{O}\left(e^{\mathbf{O}(\lambda^3)(x-y)}-1\right)+\mathbf{O}(\lambda)+\mathbf{O}\left(e^{-\tilde{\alpha}|x|}\right)\right]
 \end{equation}
In fact, the derivatives of $R_{\lambda}$ can have better bounds.

For $y \leq 0 \leq x$,
 \begin{eqnarray} \label{6.45}
 R_{\lambda,y}=e^{-\alpha|x|}e^{\left(\frac{\lambda}{f^{\prime}(\bar{u}_{-})}-\frac{\lambda ^2}{(f^{\prime}(\bar{u}_{-}))^3}
 \right)y}\left(\mathbf{O}\left(e^{\mathbf{O}(\lambda^3)y}-1\right)+\mathbf{O}(\lambda)\right)
 \end{eqnarray}

For $y \leq x \leq 0$,
 \begin{eqnarray} \label{6.46}
 R_{\lambda,y}^E=e^{-\alpha|x|}e^{\left(\frac{\lambda}{f^{\prime}(\bar{u}_{-})}-\frac{\lambda ^2}{(f^{\prime}(\bar{u}_{-}))^3}
    \right)y}\left(\mathbf{O}\left(e^{\mathbf{O}(\lambda^3)y}-1\right)+\mathbf{O}(\lambda)\right)
 \end{eqnarray}
 \begin{equation} \label{6.47}
 R_{\lambda,y}^S=e^{\left(-\frac{\lambda}{f^{\prime}(\bar{u}_-)}+\frac{\lambda^2}{(f^{\prime}(\bar{u}_-))^3}\right)(x-y)}
    \left[\lambda\mathbf{O}\left(e^{\mathbf{O}(\lambda^3)(x-y)}-1\right)
    +\mathbf{O}(\lambda)+\lambda\mathbf{O}\left(e^{-\tilde{\alpha}|x|}\right)\right]
 \end{equation}

For the case when $x\leq y$, there hold some similar estimates.
\end{pro}

\begin{proof}
For $y \leq 0 \leq x$,
 \begin{eqnarray*}
 & &
  \begin{pmatrix}
  G_\lambda & G_{\lambda,y}\\
  G_{\lambda,x} & G_{\lambda,xy}
  \end{pmatrix}\\
 &=&m^+(\lambda)\phi^+(x;\lambda)\tilde{\psi}^-(y;\lambda)
 =m^+(\lambda)W_1^{+}(x;\lambda)\tilde{W}_1^{-}(y;\lambda)\\
 &=&C\lambda^{-1}\left(W_1^{+}(x;0)+\lambda\mathbf{O}\left(e^{-\alpha|x|}\right)\right)
    \left(e^{\left(\frac{\lambda}{f^{\prime}(\bar{u}_{-})}-\frac{\lambda ^2}{(f^{\prime}(\bar{u}_{-}))^3}+\mathbf{O}(\lambda^3)\right)y}\tilde{V}_1^{-}(0)\right.\\
 & &\left.+\lambda\mathbf{O}\left(e^{\left(\frac{\lambda}{f^{\prime}(\bar{u}_{-})}-\frac{\lambda ^2}{(f^{\prime}(\bar{u}_{-}))^3}+\mathbf{O}(\lambda^3)\right)y}\right) \right)\\
 &=&C\lambda^{-1}\left(W_1^{+}(x;0)\tilde{V}_1^{-}(0)e^{\left(\frac{\lambda}{f^{\prime}(\bar{u}_{-})}-\frac{\lambda ^2}{(f^{\prime}(\bar{u}_{-}))^3}+\mathbf{O}(\lambda^3)\right)y}\right.\\
 & &\left.+\lambda \mathbf{O}\left(e^{-\alpha|x|}e^{\left(\frac{\lambda}{f^{\prime}(\bar{u}_{-})}-\frac{\lambda ^2}{(f^{\prime}(\bar{u}_{-}))^3}+\mathbf{O}(\lambda^3)\right)y}\right)\right.\\
 & &\left.+\lambda W_1^{+}(x;0)\mathbf{O}\left(e^{\left(\frac{\lambda}{f^{\prime}(\bar{u}_{-})}-\frac{\lambda ^2}{(f^{\prime}(\bar{u}_{-}))^3}+\mathbf{O}(\lambda^3)\right)y} \right)\right)\\
 &=&C\lambda^{-1}\left(W_1^{+}(x;0)\tilde{V}_1^{-}(0)e^{\left(\frac{\lambda}{f^{\prime}(\bar{u}_{-})}-\frac{\lambda ^2}{(f^{\prime}(\bar{u}_{-}))^3}\right)y}\right.\\
 & &\left.+W_1^{+}(x;0)\tilde{V}_1^{-}(0)e^{\left(\frac{\lambda}{f^{\prime}(\bar{u}_{-})}-\frac{\lambda ^2}{(f^{\prime}(\bar{u}_{-}))^3}+\mathbf{O}(\lambda^3)\right)y}\right.\\
 & &\left.-W_1^{+}(x;0)\tilde{V}_1^{-}(0)e^{\left(\frac{\lambda}{f^{\prime}(\bar{u}_{-})}-\frac{\lambda ^2}{(f^{\prime}(\bar{u}_{-}))^3}\right)y}\right.\\
 & &\left. +\lambda \mathbf{O}\left(e^{-\alpha|x|}e^{\left(\frac{\lambda}{f^{\prime}(\bar{u}_{-})}-\frac{\lambda ^2}{(f^{\prime}(\bar{u}_{-}))^3}+\mathbf{O}(\lambda^3)\right)y}\right) \right)\\
 &=&C_1\lambda^{-1}W_1^{+}(x;0)\tilde{V}_1^{-}(0)e^{\left(\frac{\lambda}{f^{\prime}(\bar{u}_{-})}-\frac{\lambda ^2}{(f^{\prime}(\bar{u}_{-}))^3}\right)y}\\
 & &+e^{-\alpha|x|}e^{\left(\frac{\lambda}{f^{\prime}(\bar{u}_{-})}-\frac{\lambda ^2}{(f^{\prime}(\bar{u}_{-}))^3}
 \right)y}\left(\mathbf{O}(\lambda^{-1})\left(e^{\mathbf{O}(\lambda^3)y}-1\right)+\mathbf{O}(1)\right)\\
 &=&\begin{pmatrix} E_\lambda & E_{\lambda,y}\\ E_{\lambda,x} & E_{\lambda,xy} \end{pmatrix}
    +  \begin{pmatrix} R_\lambda & R_{\lambda,y}\\ R_{\lambda,x} & R_{\lambda,xy} \end{pmatrix}
 \end{eqnarray*}
Thus we have those representations as claimed.

For $y \leq x \leq 0$,
 \begin{eqnarray*}
 & &
  \begin{pmatrix}
  G_\lambda & G_{\lambda,y}\\
  G_{\lambda,x} & G_{\lambda,xy}
  \end{pmatrix}\\
 &=&d^+(\lambda)\phi^-(x;\lambda)\tilde{\psi}^-(y;\lambda)+\psi^-(x;\lambda)\tilde{\psi}^-(y;\lambda)\\
 &=&d^+(\lambda)W_2^{-}(x;\lambda)\tilde{W}_1^-(y;\lambda)+W_1^-(x;\lambda)\tilde{W}_1^-(y;\lambda)\\
 &=&C\lambda^{-1}\left(W_2^{-}(x;0)+\lambda\mathbf{O}\left(e^{-\alpha|x|}\right)\right)
    \left(e^{\left(\frac{\lambda}{f^{\prime}(\bar{u}_{-})}-\frac{\lambda ^2}{(f^{\prime}(\bar{u}_{-}))^3}+\mathbf{O}(\lambda^3)\right)y}\tilde{V}_1^{-}(0)\right.\\
 & &\left.+\lambda\mathbf{O}\left(e^{\left(\frac{\lambda}{f^{\prime}(\bar{u}_{-})}-\frac{\lambda ^2}{(f^{\prime}(\bar{u}_{-}))^3}+\mathbf{O}(\lambda^3)\right)y}\right) \right)\\
 & &+\left(V_1^-(\lambda)+\mathbf{O}\left(e^{-\tilde{\alpha}|x|}\right)\right)
    e^{\left(-\frac{\lambda}{f^{\prime}(\bar{u}_-)}+\frac{\lambda^2}{(f^{\prime}(\bar{u}_-))^3}+\mathbf{O}(\lambda^3)\right)x}\\
 & &\left[e^{\left(\frac{\lambda}{f^{\prime}(\bar{u}_-)}-\frac{\lambda^2}{(f^{\prime}(\bar{u}_-))^3}+\mathbf{O}(\lambda^3)\right)y}\tilde{V}_1^{-}(0)
    +\lambda \mathbf{O}\left(e^{\left(\frac{\lambda}{f^{\prime}(\bar{u}_-)}-\frac{\lambda^2}{(f^{\prime}(\bar{u}_-))^3}+\mathbf{O}(\lambda^3)\right)y}\right)\right]\\
 &=&C_2\lambda^{-1}W_2^{-}(x;0)\tilde{V}_1^{-}(0)e^{\left(\frac{\lambda}{f^{\prime}(\bar{u}_{-})}-\frac{\lambda ^2}{(f^{\prime}(\bar{u}_{-}))^3}\right)y}\\
 & &+e^{-\alpha|x|}e^{\left(\frac{\lambda}{f^{\prime}(\bar{u}_{-})}-\frac{\lambda ^2}{(f^{\prime}(\bar{u}_{-}))^3}
    \right)y}\left(\mathbf{O}(\lambda^{-1})\left(e^{\mathbf{O}(\lambda^3)y}-1\right)+\mathbf{O}(1)\right)\\
 & &+\left(V_1^-(0)+\mathbf{O}(\lambda)+\mathbf{O}\left(e^{-\tilde{\alpha}|x|}\right)\right)
    e^{\left(-\frac{\lambda}{f^{\prime}(\bar{u}_-)}+\frac{\lambda^2}{(f^{\prime}(\bar{u}_-))^3}+\mathbf{O}(\lambda^3)\right)x}\\
 & &\left[e^{\left(\frac{\lambda}{f^{\prime}(\bar{u}_-)}-\frac{\lambda^2}{(f^{\prime}(\bar{u}_-))^3}+\mathbf{O}(\lambda^3)\right)y}\tilde{V}_1^{-}(0)
    +\lambda \mathbf{O}\left(e^{\left(\frac{\lambda}{f^{\prime}(\bar{u}_-)}-\frac{\lambda^2}{(f^{\prime}(\bar{u}_-))^3}+\mathbf{O}(\lambda^3)\right)y}\right)\right]\\
 &=&C_2\lambda^{-1}W_2^{-}(x;0)\tilde{V}_1^{-}(0)e^{\left(\frac{\lambda}{f^{\prime}(\bar{u}_{-})}-\frac{\lambda ^2}{(f^{\prime}(\bar{u}_{-}))^3}\right)y}\\
 & &+e^{-\alpha|x|}e^{\left(\frac{\lambda}{f^{\prime}(\bar{u}_{-})}-\frac{\lambda ^2}{(f^{\prime}(\bar{u}_{-}))^3}
    \right)y}\left(\mathbf{O}(\lambda^{-1})\left(e^{\mathbf{O}(\lambda^3)y}-1\right)+\mathbf{O}(1)\right)\\
 & &+V_1^{-}(0)\tilde{V}_1^{-}(0)e^{\left(-\frac{\lambda}{f^{\prime}(\bar{u}_-)}+\frac{\lambda^2}{(f^{\prime}(\bar{u}_-))^3}\right)(x-y)}\\
 & &+e^{\left(-\frac{\lambda}{f^{\prime}(\bar{u}_-)}+\frac{\lambda^2}{(f^{\prime}(\bar{u}_-))^3}\right)(x-y)}
    \left[\mathbf{O}\left(e^{\mathbf{O}(\lambda^3)(x-y)}-1\right)+\mathbf{O}(\lambda)+\mathbf{O}\left(e^{-\tilde{\alpha}|x|}\right)\right]\\
 &=&C_2\lambda^{-1}W_2^{-}(x;0)\tilde{V}_1^{-}(0)e^{\left(\frac{\lambda}{f^{\prime}(\bar{u}_{-})}-\frac{\lambda ^2}{(f^{\prime}(\bar{u}_{-}))^3}\right)y}\\
 & &+V_1^{-}(0)\tilde{V}_1^{-}(0)e^{\left(-\frac{\lambda}{f^{\prime}(\bar{u}_-)}+\frac{\lambda^2}{(f^{\prime}(\bar{u}_-))^3}\right)(x-y)}\\
 & &+e^{-\alpha|x|}e^{\left(\frac{\lambda}{f^{\prime}(\bar{u}_{-})}-\frac{\lambda ^2}{(f^{\prime}(\bar{u}_{-}))^3}
    \right)y}\left(\mathbf{O}(\lambda^{-1})\left(e^{\mathbf{O}(\lambda^3)y}-1\right)+\mathbf{O}(1)\right)\\
 & &+e^{\left(-\frac{\lambda}{f^{\prime}(\bar{u}_-)}+\frac{\lambda^2}{(f^{\prime}(\bar{u}_-))^3}\right)(x-y)}
    \left[\mathbf{O}\left(e^{\mathbf{O}(\lambda^3)(x-y)}-1\right)+\mathbf{O}(\lambda)
    +\mathbf{O}\left(e^{-\tilde{\alpha}|x|}\right)\right]\\
 &=& \begin{pmatrix} E_\lambda & E_{\lambda,y}\\ E_{\lambda,x} & E_{\lambda,xy} \end{pmatrix}
    +\begin{pmatrix} S_\lambda & S_{\lambda,y}\\ S_{\lambda,x} & S_{\lambda,xy} \end{pmatrix}
    +\begin{pmatrix} R_\lambda & R_{\lambda,y}\\ R_{\lambda,x} & R_{\lambda,xy} \end{pmatrix}\\
 &=& \begin{pmatrix} E_\lambda & E_{\lambda,y}\\ E_{\lambda,x} & E_{\lambda,xy} \end{pmatrix}
    +\begin{pmatrix} S_\lambda & S_{\lambda,y}\\ S_{\lambda,x} & S_{\lambda,xy} \end{pmatrix}
    +\begin{pmatrix} R_\lambda^E & R_{\lambda,y}^E\\ R_{\lambda,x}^E & R_{\lambda,xy}^E \end{pmatrix}
    +\begin{pmatrix} R_\lambda^S & R_{\lambda,y}^S\\ R_{\lambda,x}^S & R_{\lambda,xy}^S \end{pmatrix}
 \end{eqnarray*}
Then this gives \eqref{6.41}-\eqref{6.44}.

Next we derive the derivative bounds \eqref{6.45}-\eqref{6.47}. We utilize the estimates \eqref{6.26}.

For $y \leq 0 \leq x$,
 \begin{eqnarray*}
 & &
  G_{\lambda,y}(x,y)\\
 &=&\begin{pmatrix}1 & 0\end{pmatrix}m^+(\lambda)\phi^+(x;\lambda)
    \left(\frac{\partial}{\partial y}\right)\tilde{\psi}^-(y;\lambda)\begin{pmatrix}0 \\ 1\end{pmatrix}\\
 &=&\begin{pmatrix}1 & 0\end{pmatrix}m^+(\lambda)W_1^{+}(x;\lambda)
    \left(\frac{\partial}{\partial y}\right)\tilde{W}_1^{-}(y;\lambda)\begin{pmatrix}0 \\ 1\end{pmatrix}\\
 &=&C\lambda^{-1}\begin{pmatrix}1 & 0\end{pmatrix}\left(W_1^{+}(x;0)+\lambda\mathbf{O}\left(e^{-\alpha|x|}\right)\right)\\
 & &\left(e^{\left(\frac{\lambda}{f^{\prime}(\bar{u}_{-})}-\frac{\lambda ^2}{(f^{\prime}(\bar{u}_{-}))^3}
    +\mathbf{O}(\lambda^3)\right)y}\tilde{V}_1^{-}(0)
    \left(\frac{\lambda}{f^{\prime}(\bar{u}_{-})}-\frac{\lambda ^2}{(f^{\prime}(\bar{u}_{-}))^3}+\mathbf{O}(\lambda^3)\right)
    \right.\\
 & &\left.+C\lambda e^{\left(\frac{\lambda}{f^{\prime}(\bar{u}_{-})}
    -\frac{\lambda ^2}{(f^{\prime}(\bar{u}_{-}))^3}+\mathbf{O}(\lambda^3)\right)y}(|\lambda|+e^{-\alpha|y|})
    \right)\begin{pmatrix}0 \\ 1\end{pmatrix}
 \end{eqnarray*}
and the remainder term $R_{\lambda,y}(x,y)$ should be
$$R_{\lambda,y}=e^{-\alpha|x|}e^{\left(\frac{\lambda}{f^{\prime}(\bar{u}_{-})}-\frac{\lambda ^2}{(f^{\prime}(\bar{u}_{-}))^3}
\right)y}\left(\mathbf{O}\left(e^{\mathbf{O}(\lambda^3)y}-1\right)+\mathbf{O}(\lambda)\right).$$ The $R_{\lambda,y}(x,y)$ estimates for the case $y \leq x \leq 0$ is similarly derived.
\end{proof}

\begin{rem}
{\rm In $\mathrm{Proposition}$ $\ref{Pro6.5}$, in fact we can take $W_1^{+}(x;0)=W_2^{-}(x;0)
=\begin{pmatrix}\bar{u}^{\prime}(x) \\ \bar{u}^{\prime\prime}(x)\end{pmatrix}.$}
\end{rem}
\begin{rem}
{\rm The The derivative bounds \eqref{6.45}-\eqref{6.47} is valid only for Lax and over compressive case, it does not hold in under compressive case. See \cite{MaZ}.}
\end{rem}

\section{High-frequency bounds}

Now we derive the bounds for large $|\lambda|$, on any sector contained in the resolvent set.
\begin{pro} \label{Pro7.1}
Assuming the Lax condition, it follows that for some $C,\beta,R>0$, and $\theta_1,\theta_2 >0$ sufficiently small,
 \begin{equation*}
 \left|G_{\lambda}(x,y)\right| \leq C|\lambda|^{-\frac{1}{2}}e^{-\beta^{-\frac{1}{2}}|\lambda|^{\frac{1}{2}}|x-y|},
 \end{equation*}
 \begin{equation*}
 \left|G_{\lambda,x}(x,y)\right| \leq Ce^{-\beta^{-\frac{1}{2}}|\lambda|^{\frac{1}{2}}|x-y|}, \quad
 \left|G_{\lambda,y}(x,y)\right| \leq Ce^{-\beta^{-\frac{1}{2}}|\lambda|^{\frac{1}{2}}|x-y|},
 \end{equation*}
for all $\lambda \in \Omega_{\theta} \setminus B(0,R)$.\\
(Here, we may choose any
$\beta^{-\frac{1}{2}} < \min_{\lambda \in \Omega_{\theta}\cap\{\lambda:|\lambda| \geq R\}}\mathrm{Re}(\sqrt{\frac{\lambda}{|\lambda|}})$.)
\end{pro}

\begin{proof}
Setting $\bar{x}=|\lambda|^{\frac{1}{2}}x, \bar{\lambda}=\frac{\lambda}{|\lambda|},\bar{w}(\bar{x})=w(\frac{\bar{x}}{|\lambda|^{\frac{1}{2}}})=w(x)$, we obtain
$\bar{w}^{\prime\prime}=\bar{\lambda}\bar{w}+\mathbf{O}(|\lambda|^{-\frac{1}{2}})(\bar{w}+\bar{w}^{\prime})$,
or
\begin{equation} \label{7.5}
\overline{W}^{\prime}=\mathbb{\bar{B}}\overline{W}+\mathbf{O}(|\lambda|^{-\frac{1}{2}})\overline{W},
\end{equation}
where $\overline{W}=(\bar{w},\bar{w}^{\prime})^{T}$, and
$\mathbb{\bar{B}}:=\begin{pmatrix}0 & 1 \\ \bar{\lambda} & 0\end{pmatrix}, \mathbb{\bar{B}}^{\prime}=\mathbf{O}(|\lambda|^{-\frac{1}{2}}), |\bar{\lambda}|=1.$
It is easily computed that the eigenvalues of $\mathbb{\bar{B}}$ are $\bar{\mu}=\mp\sqrt{\bar{\lambda}}$,
We know that there exists some $\beta>0$, such that
\begin{equation} \label{7.8}
\left|\mathrm{Re} \sqrt{\bar{\lambda}}\right| > \beta^{-\frac{1}{2}}
\end{equation}
for all $\lambda \in \Omega_{\theta}$, hence the stable and unstable subspaces of each $\mathbb{\bar{B}}{\bar{x}}$ are both of dimension $n$, and separated by a spectral gap of more than $2\beta^{-\frac{1}{2}}$. Since $\mathbb{\bar{B}}(\lambda,\bar{x})$ varies within a compact set, it follows that there are continuous eigenprojections $P_{\pm}(\mathbb{\bar{B}})$ taking $\overline{W}$ onto the stable and unstable subspaces, respectively, of $\mathbb{\bar{B}}$, with
$|P_{\pm}^{\prime}|=\mathbf{O}(|\lambda|^{-\frac{1}{2}})$.\\
Introducing new coordinates $z_{\pm}=P_{\pm}\bar{w}$, we thus obtain a diagonal system
\begin{equation} \label{7.9}
 \begin{pmatrix}z_{+} \\ z_{-} \end{pmatrix}^{\prime}=\begin{pmatrix}f^{\prime}(\bar{u}_{+}) & 0\\0 & f^{\prime}(\bar{u}_{-}) \end{pmatrix}
 \begin{pmatrix} z_{+} \\ z_{-} \end{pmatrix}+\mathbf{O}\left(|\lambda|^{-\frac{1}{2}}\right)\begin{pmatrix}z_{+} \\ z_{-} \end{pmatrix},
\end{equation}
We choose $\beta$ large enough such that
$\mathrm{Re}f^{\prime}(\bar{u}_{\pm}) \lessgtr \mp \beta^{-\frac{1}{2}}$.
and hence
\begin{equation} \label{7.11}
\frac{|\bar{w}|}{C} \leq |z| \leq C|\bar{w}|.
\end{equation}
From \eqref{7.9}, we obtain the "energy estimates"
\begin{equation} \label{7.12}
\begin{aligned}
\langle z_{\pm},z_{\pm} \rangle^{\prime}
       &=\langle z_{\pm}, 2\mathrm{Re}f^{\prime}(\bar{u}_{\pm})z_{\pm} \rangle + \mathrm{O}\left(|\lambda|^{-\frac{1}{2}}\right)
          (\langle z_{+},z_{+} \rangle+\langle z_{-},z_{-}\rangle)\\
&\lessgtr\mp \beta^{-\frac{1}{2}}\langle z_{\pm},z_{\pm} \rangle+\mathrm{O}\left(|\lambda|^{-\frac{1}{2}}\right)
          (\langle z_{+},z_{+} \rangle+\langle z_{-},z_{-}\rangle)
\end{aligned}
\end{equation}
In consequence, the ratios $r_{+}:=\frac{\langle z_{-},z_{-}\rangle}{\langle z_{+},z_{+}\rangle}$ and $r_{-}:=\frac{\langle z_{+},z_{+}\rangle}{\langle z_{-},z_{-}\rangle}$
satisfy
\begin{equation} \label{7.14}
r_{\pm}^{\prime} \gtrless \pm 4\beta^{-\frac{1}{2}}r_{\pm} \mp C|\lambda|^{-\frac{1}{2}}(1+r_{\pm}+r_{\pm}^{2})
\end{equation}
for some $C>0$.
From \eqref{7.14} it follows easily that the cones $\mathbb{K}_{\mp}:=\{0<r_{\mp}<\frac{\beta^{-\frac{1}{2}}}{C}|\lambda|^{\frac{1}{2}}\}$ are invariant under forward and backward flow, respectively, of \eqref{7.9}, provided that
$C|\lambda|^{-\frac{1}{2}}\beta^{\frac{1}{2}}<\frac{4}{3}$.
Since the stable/unstable subspaces of $\begin{pmatrix}f^{\prime}(\bar{u}_{+}) & 0\\0 & f^{\prime}(\bar{u}_{-}) \end{pmatrix}$ at $x=\pm \infty$ are precisely
$\{z_{\pm}=0\}$, we have that the stable/unstable subspaces of
$\begin{pmatrix}f^{\prime}(\bar{u}_{+}) & 0\\0 & f^{\prime}(\bar{u}_{-}) \end{pmatrix}+\mathbf{O}\left(|\lambda|^{-\frac{1}{2}}\right)$ at $x=\pm\infty$ lie within the
respective cones $\mathbb{K}_{\pm}$, provided $|\lambda|$ is sufficiently large. It follows that the stable/unstable manifolds of solutions of \eqref{7.9} lie within
$\mathbb{K}_{\pm}$ for all $x$. Plugging this information back into \eqref{7.12}, we find that
$(|z_{\pm}|^{2})^{\prime} \lessgtr \mp 2\tilde{\beta}^{-\frac{1}{2}}|z_{\pm}|^{2}$
for any solution $(z_{+},z_{-})^{T}$ decaying at $x=\pm\infty$, hence
$\frac{|z_{+}(x)|}{|z_{-}(y)|} \leq e^{-\tilde{\beta}^{-\frac{1}{2}}|x-y|},$
where $0<\tilde{\beta}<\beta$, and thus
$\frac{|z(x)|}{|z(y)|} \leq C_{1} e^{-\tilde{\beta}^{-\frac{1}{2}}|x-y|},$
for any $x \lessgtr y$, provided $|\lambda|$ is sufficiently large. This gives
\begin{equation} \label{7.18}
\frac{\overline{W}(x)}{\overline{W}(y)} \leq C_{1}C^{2}e^{-\tilde{\beta}^{-\frac{1}{2}}|x-y|},
\end{equation}
where $C$ is as in \eqref{7.11}. Further, untangling intermediate coordinate changes, we find that
\begin{pro}$(\mathbb{K})$
The stable/unstable manifolds of solutions of \eqref{7.5} lie within angle $\mathbf{O}\left(|\lambda|^{-\frac{1}{2}}\right)$ of the stable/unstable subspaces of
$\mathbb{\bar{B}}(x)$.
\end{pro}
Now, recall the coordinate-free representation of the Green function as
$\begin{pmatrix}G_{\lambda} \\ G_{\lambda,x} \end{pmatrix}=\mathcal{F}^{y \to x}\Pi_{+}(y)
 \begin{pmatrix} 0 \\  B^{-1}(y)\end{pmatrix}.$
Translating the bound \eqref{7.18} back to the original system \eqref{2.1}, we obtain
\begin{equation} \label{7.19}
|\mathcal{F}^{y \to x}| \leq C_{1}C^{2}e^{-\tilde{\beta}|\lambda|^{\frac{1}{2}}|x-y|},
\end{equation}
Likewise, the projection $\Pi_{+}$ can be related to its counterparts $\bar{\Pi}_{+}$ for the rescaled system by the factorization
$
\Pi_{+}=\begin{pmatrix}1 & 0 \\ 0 & |\lambda|^{\frac{1}{2}} \end{pmatrix}\bar{\Pi}_{+} \begin{pmatrix}1 & 0 \\ 0 & |\lambda|^{-\frac{1}{2}} \end{pmatrix},
$
and similarly for $\tilde{\Pi}_{-}$. Since the stable/unstable manifolds stay separated, by Proposition $(\mathbb{K})$, and $\bar{\lambda}$ varies within a compact set, the projections $\bar{\Pi}_{+}$ and $\tilde{\bar{\Pi}}_{-}$ are uniformly bounded. Thus, we have
$
\Pi_{+}(y)\begin{pmatrix}0 \\ B^{-1}(y)\end{pmatrix}
=\begin{pmatrix}1 & 0 \\ 0 & |\lambda|^{\frac{1}{2}}\end{pmatrix} \mathbf{O}(1) \begin{pmatrix}1 & 0 \\ 0 & |\lambda|^{-\frac{1}{2}}\end{pmatrix}
   \begin{pmatrix}0 \\ B^{-1}(y)\end{pmatrix}
=\begin{pmatrix}\mathbf{O}\left(|\lambda|^{-\frac{1}{2}}\right) \\ \mathbf{O}(1)\end{pmatrix}
$.
Combining with \eqref{7.19}, and recalling that $0<\tilde{\beta}<\beta$ was arbitrary in \eqref{7.8}, we obtain the claimed bounds on $|G_{\lambda}|$ and $|G_{\lambda,x}|$.
The bound on $|G_{\lambda,y}|$ follows by symmetric argument applied to the adjoint operator $L^{\ast}$, or, equivalently, using the symmetric representation
$
\begin{pmatrix} G_{\lambda} & G_{\lambda,y} \end{pmatrix} = \begin{pmatrix} 0 & B^{-1}(y) \end{pmatrix}
\tilde{\Pi}_{-}(x)\mathcal{\tilde{F}}^{x \to y},
$
where $\mathcal{\tilde{F}}^{x \to y}$ denotes the flow of the adjoint eigenvalue equation.
\end{proof}

\section{Pointwise Green function bounds} \label{sectiongreenfunctionbounds}

In this section, let us recall the representation in \cite{MaZ} and \cite{ZH},
\begin{equation} \label{8.1}
G(x,t;y)=\frac{1}{2\pi i}\mathrm{P.V.}\int_{\eta-i\infty}^{\eta+i\infty}e^{\lambda t}G_{\lambda}(x,y)d\lambda
        =\frac{1}{2\pi i}\lim_{T \to \infty}\int_{\eta-iT}^{\eta+iT}e^{\lambda t}G_{\lambda}(x,y)d\lambda
\end{equation}
which is valid for $\eta$ sufficiently large.
We will spend this entire section proving Theorem \ref{greenfunctionbounds}.

\begin{proof}[Proof of Theorem \ref{greenfunctionbounds}]
{\bf Case I. $\frac{|x-y|}{t}$ large.}
We first treat the trivial case that $\frac{|x-y|}{t} \geq S$, $S$ sufficiently large, the regime in which standard short-time parabolic theory applies. Set
\begin{equation} \label{8.7}
\bar{\alpha}:=\frac{|x-y|}{2\beta t}, \quad R:= \beta \bar{\alpha}^2,
\end{equation}
where $\beta$ is as in Proposition \ref{Pro7.1}, and consider again the representation of $G$:
$
G(x,t;y)=\frac{1}{2\pi i}\int_{\Gamma_1\cup \Gamma_2}e^{\lambda t} G_{\lambda}(x,y) d\lambda,
$
where $\Gamma_1:= \partial B(0,R)\cap \bar{\Omega}_\theta$ and $\Gamma_2:= \partial \Omega_\theta \setminus B(0,R)$.
Note that the intersection of $\Gamma$ with the real axis is $\lambda_{\mathrm{min}}=R=\beta \bar{\alpha}^2$.
By the large $|\lambda|$ estimates of Proposition \ref{Pro7.1}, we have for all $\lambda \in \Gamma_1\cup \Gamma_2$ that
\begin{equation} \label{8.9}
|G_{\lambda}(x,y)|\leq C |\lambda|^{-1/2} e^{-\beta^{-\frac{1}{2}}|\lambda|^{\frac{1}{2}}|x-y|}.
\end{equation}
Further, we have
\begin{eqnarray} \label{8.10}
\mathrm{Re} \lambda &\leq&  R(1- \eta\omega^2), \quad \lambda\in \Gamma_1,\\
\mathrm{Re} \lambda &\leq& \mathrm{Re}\lambda_0 - \eta (|\mathrm{Im} \lambda| - |\mathrm{Im} \lambda_0|), \quad \lambda \in \Gamma_2,
\end{eqnarray}
for $R$ sufficiently large, where $\omega$ is the argument of $\lambda$ and $\lambda_0$ and $\lambda_0^*$ are the two points of intersection of $\Gamma_1$ and $\Gamma_2$,
for some $\eta>0$ independent of $\bar{\alpha}$.

Combining \eqref{8.9},\eqref{8.10} and \eqref{8.7}, we obtain
\begin{align*}
\left|\int_{\Gamma_{1}} e^{\lambda t} G_{\lambda}(x,y)  d\lambda\right|
&\leq \int_{\Gamma_{1}}C |\lambda|^{-\frac{1}{2}} e^{(\mathrm{Re}\lambda) t -\beta^{-\frac{1}{2}} |\lambda^{\frac{1}{2}}||x-y| } d\lambda \\
&\leq C e^{-\beta \bar{\alpha}^{2}t} \int_{-M}^{+M} R^{-\frac{1}{2}}e^{-\beta R \eta \omega^2 t} R d\omega
\leq C t^{-\frac{1}{2}} e^{-\beta \bar{\alpha}^{2}t}.
\end{align*}
Likewise,
\begin{align*}
\left|\int_{\Gamma_{2}} e^{\lambda t} G_{\lambda}(x,y )d\lambda \right|
&\leq \int_{\Gamma_{2}} C |\lambda|^{-\frac{1}{2}} C e^{(\mathrm{Re}\lambda)t-\beta^{-\frac{1}{2}}|\lambda|^{\frac{1}{2}}||x-y|} d\lambda \\
&\leq C e^{(\mathrm{Re}\lambda_{0})t-\beta^{-\frac{1}{2}}|\lambda_{0}|^{\frac{1}{2}}|x-y|}
       \int_{\Gamma_{2}} |\lambda|^{-\frac{1}{2}}e^{(\mathrm{Re}\lambda-\mathrm{Re}\lambda_{0})t}|d\lambda| \\
&\leq C e^{-\beta \bar{\alpha}^{2}t}\int_{\Gamma_2}|\mathrm{Im}\lambda|^{-\frac{1}{2}}
       e^{-\eta(|\mathrm{Im}\lambda|-|\mathrm{Im}\lambda_{0}|)t}|d\mathrm{Im}\lambda| \\
&\leq Ct^{-\frac{1}{2}} e^{-\beta \bar{\alpha}^{2}t}.
\end{align*}

Combining these last two estimates, and recalling \eqref{8.7}, we have
$$
|G(x,t;y)|
\leq Ct^{-\frac{1}{2}} e^{-\frac{\beta \bar{\alpha}^{2}t}{2}} e^{-\frac{(x-y)^{2}}{8\beta t}}
\leq Ct^{-\frac{1}{2}} e^{-\eta t} e^{-\frac{(x-y)^{2}}{8\beta t}},
$$
for $\eta>0$ independent of $\bar{\alpha}$. Observing that
$
\frac{|x-y-at|}{2t} \leq \frac{|x-y|}{t} \le \frac{2|x-y-at|}{t}
$
for any bounded $a$, for $|x-y|/t$ sufficiently large, we find that $|G|$ can be absorbed in the residual term $\mathbf{O}\left(e^{-\eta t}e^{-\frac{|x-y|^2}{Mt}}\right)$ for $t\geq \epsilon$, any $\epsilon>0$, and in the residual term\\
$\mathbf{O}\left((t+1)^{-\frac{1}{2}}e^{-\eta x^+}t^{-\frac{1}{2}}e^{-\frac{(x-y-f^{\prime}(\bar{u}_{-})t)^2}{Mt}}\right)$ for $t$ small.

{\bf Case II. $\frac{|x-y|}{t}$ bounded.}
We now turn to the critical case where $\frac{|x-y|}{t} \leq S$ for some fixed $S$. In this regime, note that any contribution of order $e^{\theta t},\theta >0$, may be absorbed in the residual term $R$; we shall use this observation repeatedly. We begin by converting contour integral \eqref{8.1} into a more convenient form decomposing high, intermediate, and low frequency contributions.
\begin{lem} \label{Lem8.2}
If the Lax condition holds, we can use the following decomposition:
 \begin{equation*}
 G(x,t;y)
 = \mathbf{I} + \mathbf{II}
 = \frac{1}{2\pi i}\int_{\Gamma_{2}} e^{\lambda t} G_{\lambda}(x,y) d\lambda+\frac{1}{2\pi i}\int_{\Gamma^{\prime}} e^{\lambda t} G_{\lambda}(x,y) d\lambda
 \end{equation*}
where $\Gamma^{\prime}:=[-\eta_{1}-iR,\eta-iR]\cup[\eta-iR,\eta+iR]\cup[\eta+iR,-\eta_{1}+iR]$, and
$\Gamma_{2}:=\partial \Omega_{\theta}\setminus\Omega$
with $\Omega_{\theta}$ as defined in section \ref{sectioncon}, for any $\eta>0$ such that \eqref{8.1} holds, $R$ sufficiently large, and $\eta_{1}>0$ sufficiently small such that $\Omega \setminus B(0,r)$ is compactly contained in the set of consistent splitting $\Lambda$ for some small $r>0$ to be chosen later, where $\Omega:=\{\lambda:-\eta_{1}\leq \mathrm{Re} \lambda \}$.
\end{lem}

\begin{lem}
The term $\mathbf{II}$ in Lemma \ref{Lem8.2} may be further decomposed as
 \begin{equation*}
 \mathbf{II}
 = \mathbf{\tilde{II}} + \mathbf{III}
 = \frac{1}{2\pi i}\left(\int_{-\eta_{1}-iR}^{-\eta_{1}-i\frac{r}{2}}+\int_{-\eta_{1}+i\frac{r}{2}}^{-\eta_{1}+iR}\right)e^{\lambda t}G_{\lambda}(x,y)d\lambda
 +\frac{1}{2\pi i}\int_{\tilde{\Gamma}}e^{\lambda t}G_{\lambda}(x,y)d\lambda
 \end{equation*}
and
$
\tilde{\Gamma}:=[-\eta_{1}-i\frac{r}{2},\eta-i\frac{r}{2}]\cup [\eta-i\frac{r}{2},\eta+i\frac{r}{2}]\cup [\eta+i\frac{r}{2},-\eta_{1}+i\frac{r}{2}],
$
for any $\eta,r>0$, and $\eta_{1}$ sufficiently small with respect to $r$.
\end{lem}
The proofs of these two lemmas are trivial. We are going to estimate terms $\mathbf{I},\mathbf{\tilde{II}}$ and $\mathbf{III}$ respectively.

The term $\mathbf{I}$ may be estimated exactly as was term $\int_{\Gamma_{2}} e^{\lambda t} G_{\lambda}(x,y )d\lambda$ in the large $\frac{|x-y|}{t}$ case (Case I.), to
obtain contribution
$\mathbf{O}\left(t^{-\frac{1}{2}}e^{-\eta_{1} t}\right)$ absorbable again in the residual term
$\mathbf{O}\left(e^{-\eta t}e^{-\frac{|x-y|^2}{Mt}}\right)$ for $t \geq \epsilon$, any $\epsilon >0$, and in the residual term

$\mathbf{O}\left((t+1)^{-\frac{1}{2}}e^{-\eta x^+}t^{-\frac{1}{2}}e^{-\frac{(x-y-f^{\prime}(\bar{u}_{-})t)^2}{Mt}}\right)$ for $t$ small.
To estimate the term $\mathbf{\tilde{II}}$, we use the fact that $|G_{\lambda}(x,y)| \leq C e^{-\eta |x-y|}$ for $\lambda$ on any compact subset $K$ of $\rho(L)\cap \Lambda$, where $C>0$ and $\eta>0$ depend only on $K,L$.
\begin{eqnarray*}
\left|\mathbf{\tilde{II}}\right|
&\leq& \frac{1}{2\pi}\left(\left|\int_{-R}^{-\frac{r}{2}}e^{(-\eta_{1}+i\xi)t}e^{-\eta |x-y|}d\xi
       +\int_{\frac{r}{2}}^{R}e^{(-\eta_{1}+i\xi)t}e^{-\eta |x-y|}d\xi\right|\right)\\
&  = & \frac{1}{2\pi}e^{-\eta_{1}t}e^{-\eta |x-y|}\left|\int_{-R}^{-\frac{r}{2}}e^{i\xi t}d\xi+\int_{\frac{r}{2}}^{R}e^{i\xi t}d\xi\right|\\
&  = & \frac{1}{2\pi}e^{-\eta_{1}t}e^{-\eta |x-y|}t^{-1}\left|e^{-i\frac{rt}{2}}-e^{-iRt}+e^{iRt}-e^{i\frac{rt}{2}}\right|\\
&\leq& \frac{2}{\pi}t^{-1}e^{-\eta_{1}t}e^{-\eta |x-y|}
\end{eqnarray*}
Thus $\mathbf{\tilde{II}}$ can be absorbed in the residual term $R$.

It remains to estimate the low frequency term $\mathbf{III}=\frac{1}{2\pi i}\int_{\tilde{\Gamma}}e^{\lambda t}G_{\lambda}(x,y)d\lambda$.

{\bf Case. $t \leq 1$.}
First observe that estimates in the short-time regime $t \leq 1$ are trivial, since then $|e^{\lambda t}G_{\lambda}(x,y)|$ is uniformly bounded on the compact set $\tilde{\Gamma}$, and we have $|G(x,t;y)|\leq C\leq e^{-\theta t}$ for $\theta >0$ sufficiently small. But, likewise, $E$ and $S$ are uniformly bounded in this regime, hence time-exponentially decaying. As observed previously, all such terms are negligible, begin absorbable in the error term $R$. Thus, we may add $E+S$ and subtract $G$ to obtain the result.

{\bf Case. $t \geq 1$.}
Next, consider the critical (long-time) regime $t \geq 1$. For definiteness, take $y \leq x \leq 0$; the other two cases are similar. Decomposing,
\begin{equation*}
    \frac{1}{2\pi i}\int_{\tilde{\Gamma}}e^{\lambda t}G_{\lambda}(x,y)d\lambda
=   \frac{1}{2\pi i}\int_{\tilde{\Gamma}}e^{\lambda t}E_{\lambda}(x,y)d\lambda
   +\frac{1}{2\pi i}\int_{\tilde{\Gamma}}e^{\lambda t}S_{\lambda}(x,y)d\lambda
   +\frac{1}{2\pi i}\int_{\tilde{\Gamma}}e^{\lambda t}R_{\lambda}(x,y)d\lambda
\end{equation*}
with $E_{\lambda}$, $S_{\lambda}$ and $R_{\lambda}$ as defined in Proposition \ref{Pro6.5}, we consider in turn each of the three terms on the right-hand side.

\textit{The $E_{\lambda}$ term.} Let us first consider the dominant term
$
\frac{1}{2\pi i}\int_{\tilde{\Gamma}}e^{\lambda t}E_{\lambda}(x,y)d\lambda
$
which by (\ref{6.38}) is given by
$
C_1\bar{u}^{\prime}(x)\Xi(x,t;y),
$
where
$
\Xi(x,t;y):=\frac{1}{2\pi i}\int_{\tilde{\Gamma}}\lambda^{-1}e^{\lambda t}
e^{\left(\frac{\lambda}{f^{\prime}(\bar{u}_{-})}-\frac{\lambda ^2}{(f^{\prime}(\bar{u}_{-}))^3}\right)y}d\lambda.
$
Using Cauchy's theorem,
\begin{equation*}
\frac{1}{2\pi i}\left(\int_{\tilde{\Gamma}}+\int_{-\eta_{1}+i\frac{r}{2}}^{-\eta_{1}-i\frac{r}{2}}\right)
\lambda^{-1}e^{\lambda t}
e^{\left(\frac{\lambda}{f^{\prime}(\bar{u}_{-})}-\frac{\lambda ^2}{(f^{\prime}(\bar{u}_{-}))^3}\right)y}d\lambda
=\mathrm{Res}_{\lambda=0}\lambda^{-1}e^{\lambda t}
e^{\left(\frac{\lambda}{f^{\prime}(\bar{u}_{-})}-\frac{\lambda ^2}{(f^{\prime}(\bar{u}_{-}))^3}\right)y},
\end{equation*}
thus we may move the contour $\tilde{\Gamma}$ to obtain
\begin{align*}
\Xi(x,t;y)
&=\frac{1}{2\pi i}\int_{\tilde{\Gamma}}\lambda^{-1}e^{\lambda t}
   e^{\left(\frac{\lambda}{f^{\prime}(\bar{u}_{-})}-\frac{\lambda ^2}{(f^{\prime}(\bar{u}_{-}))^3}\right)y}d\lambda\\
&=-\frac{1}{2\pi i}\int_{-\eta_{1}+i\frac{r}{2}}^{-\eta_{1}-i\frac{r}{2}}
   \lambda^{-1}e^{\lambda t}
   e^{\left(\frac{\lambda}{f^{\prime}(\bar{u}_{-})}-\frac{\lambda ^2}{(f^{\prime}(\bar{u}_{-}))^3}\right)y}d\lambda
   +\mathrm{Res}_{\lambda=0}\lambda^{-1}e^{\lambda t}
   e^{\left(\frac{\lambda}{f^{\prime}(\bar{u}_{-})}-\frac{\lambda ^2}{(f^{\prime}(\bar{u}_{-}))^3}\right)y}\\
&=\frac{1}{2\pi i}\left(\int_{-\eta_{1}-i\frac{r}{2}}^{-i\frac{r}{2}}+\int_{i\frac{r}{2}}^{-\eta_{1}+i\frac{r}{2}}\right)
   \lambda^{-1}e^{\lambda t}e^{\left(\frac{\lambda}{f^{\prime}(\bar{u}_{-})}-\frac{\lambda^2}
   {(f^{\prime}(\bar{u}_{-}))^3}\right)y}d\lambda\\
&\quad +\frac{1}{2\pi i}\left(\int_{-i\frac{r}{2}}^{-i\delta}+\int_{i\delta}^{i\frac{r}{2}}\right)
   \lambda^{-1}e^{\lambda t}
   e^{\left(\frac{\lambda}{f^{\prime}(\bar{u}_{-})}-\frac{\lambda ^2}{(f^{\prime}(\bar{u}_{-}))^3}\right)y}d\lambda\\
&\quad +\frac{1}{2\pi i}\int_{\gamma}\lambda^{-1}e^{\lambda t}
   e^{\left(\frac{\lambda}{f^{\prime}(\bar{u}_{-})}-\frac{\lambda ^2}{(f^{\prime}(\bar{u}_{-}))^3}\right)y}d\lambda
   +\mathrm{Res}_{\lambda=0}\lambda^{-1}e^{\lambda t}
   e^{\left(\frac{\lambda}{f^{\prime}(\bar{u}_{-})}-\frac{\lambda ^2}{(f^{\prime}(\bar{u}_{-}))^3}\right)y}
\end{align*}
where $\gamma$ is the left half circle $\gamma:=\{\delta e^{i\theta}:\frac{\pi}{2}\leq \theta\leq\frac{3\pi}{2}\}$, for some $\delta>0$. Notice that
$
\mathrm{Res}_{\lambda=0}\lambda^{-1}e^{\lambda t}
   e^{\left(\frac{\lambda}{f^{\prime}(\bar{u}_{-})}-\frac{\lambda ^2}{(f^{\prime}(\bar{u}_{-}))^3}\right)y}
=\lim_{\lambda \to 0}e^{\lambda t}
   e^{\left(\frac{\lambda}{f^{\prime}(\bar{u}_{-})}-\frac{\lambda ^2}{(f^{\prime}(\bar{u}_{-}))^3}\right)y}
=1
$\\
and
$
\lim_{\delta \to 0}\frac{1}{2\pi i}\int_{\gamma}\lambda^{-1}e^{\lambda t}
   e^{\left(\frac{\lambda}{f^{\prime}(\bar{u}_{-})}-\frac{\lambda ^2}{(f^{\prime}(\bar{u}_{-}))^3}\right)y}d\lambda
=-\frac{1}{2},
$
thus sending $\delta \to 0$ in the above calculations gives us
\begin{eqnarray*}
\Xi(x,t;y)
&=&\frac{1}{2\pi}\mathrm{P.V.}\int_{-\frac{r}{2}}^{\frac{r}{2}}
   (i\xi)^{-1}e^{i\xi t}
   e^{\left(\frac{i\xi}{f^{\prime}(\bar{u}_{-})}+\frac{\xi ^2}{(f^{\prime}(\bar{u}_{-}))^3}\right)y}d\xi\\
& &+\frac{1}{2\pi i}\left(\int_{-\eta_{1}-i\frac{r}{2}}^{-i\frac{r}{2}}+\int_{i\frac{r}{2}}^{-\eta_{1}+i\frac{r}{2}}\right)
   \lambda^{-1}e^{\lambda t}e^{\left(\frac{\lambda}{f^{\prime}(\bar{u}_{-})}-\frac{\lambda^2}
   {(f^{\prime}(\bar{u}_{-}))^3}\right)y}d\lambda\\
& &+\frac{1}{2}\mathrm{Res}_{\lambda=0}\lambda^{-1}e^{\lambda t}
   e^{\left(\frac{\lambda}{f^{\prime}(\bar{u}_{-})}-\frac{\lambda ^2}{(f^{\prime}(\bar{u}_{-}))^3}\right)y}\\
&=&\left(\frac{1}{2\pi}\mathrm{P.V.}\int_{-\infty}^{+\infty}(i\xi)^{-1}e^{i\xi\left(t+\frac{y}{f^{\prime}(\bar{u}_{-})}\right)}
   e^{\xi^2\frac{y}{(f^{\prime}(\bar{u}_{-}))^3}}d\xi+\frac{1}{2}\right)\\
& &-\frac{1}{2\pi}\left(\int_{-\infty}^{-\frac{r}{2}}+\int_{\frac{r}{2}}^{+\infty}\right)
   (i\xi)^{-1}e^{i\xi\left(t+\frac{y}{f^{\prime}(\bar{u}_{-})}\right)}e^{\xi^2\frac{y}{(f^{\prime}(\bar{u}_{-}))^3}}d\xi\\
& &+\frac{1}{2\pi i}\left(\int_{-\eta_{1}-i\frac{r}{2}}^{-i\frac{r}{2}}+\int_{i\frac{r}{2}}^{-\eta_{1}+i\frac{r}{2}}\right)
   \lambda^{-1}e^{\lambda t}e^{\left(\frac{\lambda}{f^{\prime}(\bar{u}_{-})}-\frac{\lambda^2}
   {(f^{\prime}(\bar{u}_{-}))^3}\right)y}d\lambda
\end{eqnarray*}
The first term in the above equality may be explicitly evaluated to give
\begin{equation} \label{8.24}
\mathrm{errfn}\left(\frac{y+f^{\prime}(\bar{u}_{-})t}{\sqrt{4\left|\frac{y}{f^{\prime}(\bar{u}_{-})}\right|}}\right),
\end{equation}
where
$
\mathrm{errfn}(z):=\frac{1}{\sqrt{\pi}}\int_{-\infty}^{z}e^{-y^2}dy,
$
whereas the second and third terms are clearly time-exponentially small for $t \leq C|y|$ and $\eta_{1}$ sufficiently small relative to $r$. In the trivial case $t \geq C|y|, C>0$ sufficiently large, we can simply move the contour to
$[-\eta_{1}-i\frac{r}{2},-\eta_{1}+i\frac{r}{2}]$ to obtain a complete residue of $1$ plus a time-exponentially small error corresponding to the shifted contour integral, which result again agrees with \eqref{8.24} up to a time-exponentially small error.

Expression \eqref{8.24} may be rewritten as
\begin{equation} \label{8.26}
\mathrm{errfn}\left(\frac{y+f^{\prime}(\bar{u}_{-})t}{\sqrt{4t}}\right),
\end{equation}
plus error
\begin{equation} \label{8.27}
\begin{aligned}
& \mathrm{errfn}\left(\frac{y+f^{\prime}(\bar{u}_{-})t}{\sqrt{4\left|\frac{y}{f^{\prime}(\bar{u}_{-})}\right|}}\right)
   -\mathrm{errfn}\left(\frac{y+f^{\prime}(\bar{u}_{-})t}{\sqrt{4t}}\right)\\
&=\mathrm{errfn}^{\prime}\left(\frac{y+f^{\prime}(\bar{u}_{-})t}{\sqrt{4t}}\right)
   \left(-2(y+f^{\prime}(\bar{u}_{-})t)^{2}(4t)^{-\frac{3}{2}} \right)
=\mathbf{O}(t^{-1}e^{\frac{(y+f^{\prime}(\bar{u}_{-})t)^2}{Mt}}),
\end{aligned}
\end{equation}
for $M>0$ sufficiently large, and similarly for $x$- and $y$-derivatives. Multiplying by
$C_1\bar{u}^{\prime}(x)=\mathbf{O}\left(e^{-\alpha |x|}\right)$
we find that term \eqref{8.26} gives contribution
\begin{equation} \label{8.28}
C_1\bar{u}^{\prime}(x)\mathrm{errfn}\left(\frac{y+f^{\prime}(\bar{u}_{-})t}{\sqrt{4t}}\right)
\end{equation}
whereas term \eqref{8.27} gives a contribution absorbable in $R$.

Finally, observing that
\begin{equation}
C_1\bar{u}^{\prime}(x)\mathrm{errfn}\left(\frac{y-f^{\prime}(\bar{u}_{-})t}{\sqrt{4t}}\right)
\end{equation}
is time-exponentially small for $t \geq 1$, since $f^{\prime}(\bar{u}_{-})>0, y<0$, and
$\left|\bar{u}^{\prime}(x)\right|\leq Ce^{-\alpha |x|}$, we may subtract and add this term to \eqref{8.28} to obtain a total of $$E(x,t;y)=C\bar{u}^{\prime}(x)\left(\mathrm{errfn}\left(\frac{y+f^{\prime}(\bar{u}_{-})t}{\sqrt{4t}}\right)
 -\mathrm{errfn}\left(\frac{y-f^{\prime}(\bar{u}_{-})t}{\sqrt{4t}}\right)\right),$$ plus terms absorbable in $R$.

\textit{The $S_{\lambda}$ term.} Next, we consider the second-order term
$
\frac{1}{2\pi i}\int_{\tilde{\Gamma}}e^{\lambda t}S_{\lambda}(x,y)d\lambda
$
which by \eqref{6.42}, is given by
$C_2\Xi^{\prime}(x,t;y)$
where
$
\Xi^{\prime}:=\frac{1}{2\pi i}
\int_{\tilde{\Gamma}}e^{\lambda t}
e^{\left(-\frac{\lambda}{f^{\prime}(\bar{u}_-)}+\frac{\lambda^2}{(f^{\prime}(\bar{u}_-))^3}\right)(x-y)}d\lambda.
$
Similarly as in the treatment of the $E_{\lambda}$ term, just above, by deforming the contour $\tilde{\Gamma}$ to
$
\Gamma^{\prime\prime}:=[-\eta_{1}-i\frac{r}{2},-i\frac{r}{2}]
\cup[-i\frac{r}{2},+i\frac{r}{2}]\cup[+i\frac{r}{2},-\eta_{1}+i\frac{r}{2}],
$
these may be transformed, neglecting time-exponentially decaying terms, to the elementary Fourier integrals
$$
\frac{1}{2\pi}\mathrm{P.V.}\int_{-\infty}^{+\infty}e^{i\xi\left(t-\frac{x-y}{f^{\prime}(\bar{u}_{-})}\right)}
e^{\xi^{2}\left(-\frac{1}{\left(f^{\prime}(\bar{u}_{-})\right)^3}\right)(x-y)}d\xi
=(4\pi t)^{-\frac{1}{2}}e^{-\frac{(x-y-f^{\prime}(\bar{u}_{-})t)^2}{4t}}.
$$
Noting that for $t \geq 1, y \leq x \leq 0$, there is
$$
\left|(4\pi t)^{-\frac{1}{2}}e^{-\frac{(x-y-f^{\prime}(\bar{u}_{-})t)^2}{4t}}\left(1-\frac{e^{-x}}{e^x+e^{-x}}\right)\right|
\leq (4\pi t)^{-\frac{1}{2}}e^{-\frac{(x-y-f^{\prime}(\bar{u}_{-})t)^2}{4t}}e^{-\alpha |x|}
$$
for some $\alpha>0$, so is absorbable in error term $R$, we find that the total contribution of this term, neglecting terms absorbable in $R$, is
$$
S(x,t;y)=\chi_{\{t \geq 1\}}(4\pi t)^{-\frac{1}{2}}e^{-\frac{\left(x-y-f^{\prime}(\bar{u}_{-})t\right)^2}{4t}}\left(\frac{e^{-x}}{e^x+e^{-x}}\right).
$$

\textit{The $R_{\lambda}$ term.} Finally, we briefly discuss the estimation of error term
$\frac{1}{2\pi i}\int_{\tilde{\Gamma}}e^{\lambda t}R_{\lambda}(x,y)d\lambda$. We can decompose the above integral into sum of integrals involving various terms of $R_{\lambda}^{E}$ and $R_{\lambda}^{S}$ given in \eqref{6.43} and \eqref{6.44}. By expanding the term $\mathbf{O}\left(e^{\mathbf{O}(\lambda^3)(x-y)}-1\right)$, we get contour integrals of the form
\begin{equation} \label{8.35}
\frac{1}{2\pi i}\int_{\tilde{\Gamma}}e^{\lambda t}\lambda^q
e^{\left(-\frac{\lambda}{f^{\prime}(\bar{u}_-)}+\frac{\lambda^2}{(f^{\prime}(\bar{u}_-))^3}\right)(x-y)}d\lambda
\end{equation}
It may be deformed to contour
$
\Gamma^{\prime\prime\prime}:=[-\eta_1-i\frac{r}{2},\eta_{\ast}-i\frac{r}{2}]\cup[\eta_{\ast}-i\frac{r}{2},\eta_{\ast}+i\frac{r}{2}]
\cup[\eta_{\ast}+i\frac{r}{2},-\eta_1+i\frac{r}{2}],
$
where the saddle-point $\eta_{\ast}$ is defined as
\begin{equation*}
\eta_{\ast}(x,y,t)
:=\left\{
 \begin{array}{l l}
 \frac{\bar{\alpha}}{p},& \quad \mathrm{if} \quad \left|\frac{\bar{\alpha}}{p}\right|\leq \varepsilon;\\
  & \\
 \pm \varepsilon,& \quad \mathrm{if} \quad \frac{\bar{\alpha}}{p}\gtrless\pm\varepsilon,
 \end{array}
 \right. \\
\end{equation*}
with
$\bar{\alpha}:=\frac{x-y-f^{\prime}(\bar{u}_-)t}{2t}$, $p:=\frac{x-y}{(f^{\prime}(\bar{u}_-))^2t}>0$,
so the integral \eqref{8.35} may be rewritten as
\begin{equation*}
\frac{1}{2\pi i}\left(\int_{-\eta_1-i\frac{r}{2}}^{\eta_{\ast}-i\frac{r}{2}}
+\int_{\eta_{\ast}-i\frac{r}{2}}^{\eta_{\ast}+i\frac{r}{2}}
+\int_{\eta_{\ast}+i\frac{r}{2}}^{-\eta_1+i\frac{r}{2}}\right)e^{\lambda t}\lambda^q
e^{\left(-\frac{\lambda}{f^{\prime}(\bar{u}_-)}+\frac{\lambda^2}{(f^{\prime}(\bar{u}_-))^3}\right)(x-y)}d\lambda.
\end{equation*}
With a bit of computation we can show that the main contribution lies along the central vertical portion
$[\eta_{\ast}-i\frac{r}{2},\eta_{\ast}+i\frac{r}{2}]$ of the contour $\Gamma^{\prime\prime\prime}$:
\begin{equation} \label{8.39}
\frac{1}{2\pi i}\int_{\eta_{\ast}-i\frac{r}{2}}^{\eta_{\ast}+i\frac{r}{2}}e^{\lambda t}\lambda^q
   e^{\left(-\frac{\lambda}{f^{\prime}(\bar{u}_-)}+\frac{\lambda^2}{(f^{\prime}(\bar{u}_-))^3}\right)(x-y)}d\lambda.
\end{equation}
Now we estimate \eqref{8.39}, set $\lambda=\eta_{\ast}+i\xi$ where $-\frac{r}{2}\leq\xi\leq\frac{r}{2}$. Because
\begin{eqnarray*}
& &\mathrm{Re}\left(\lambda t
   +\left(-\frac{\lambda}{f^{\prime}(\bar{u}_-)}+\frac{\lambda^2}{(f^{\prime}(\bar{u}_-))^3}\right)(x-y)\right)\\
&=&\mathrm{Re}\left(-\lambda\frac{2t}{f^{\prime}(\bar{u}_-)}
   \left(\frac{x-y-f^{\prime}(\bar{u}_-)t}{2t}\right)
   +\frac{\lambda^2 t}{f^{\prime}(\bar{u}_-)}\frac{(x-y)}{(f^{\prime}(\bar{u}_-))^2t}\right)\\
&=&\mathrm{Re}\left(-\frac{2\lambda t}{f^{\prime}(\bar{u}_-)}\bar{\alpha}+\frac{\lambda^2t}{f^{\prime}(\bar{u}_-)}p\right)
 = \mathrm{Re}\left(-\frac{t}{f^{\prime}(\bar{u}_-)}\left(2\bar{\alpha}\lambda-p\lambda^2\right)\right)\\
&=&-\frac{t}{f^{\prime}(\bar{u}_-)}\left(2\bar{\alpha}\mathrm{Re}(\lambda)-p\mathrm{Re}(\lambda^2)\right)
 = -\frac{t}{f^{\prime}(\bar{u}_-)}\left(2\bar{\alpha}\mathrm{Re}(\eta_{\ast}+i\xi)-p\mathrm{Re}(\eta_{\ast}+i\xi)^2\right)\\
&=&-\frac{t}{f^{\prime}(\bar{u}_-)}\left(2\bar{\alpha}\eta_{\ast}-p\eta_{\ast}^2+p\xi^2\right)
 = -\frac{t}{f^{\prime}(\bar{u}_-)}\frac{\bar{\alpha}^{2}}{p}-\frac{tp}{f^{\prime}(\bar{u}_-)}\xi^2
\end{eqnarray*}
and $|\lambda|^q =|\eta_{\ast}+i\xi|^q \leq \mathbf{O}(|\eta_{\ast}|^q + |\xi|^q)$, we have
\begin{eqnarray*}
&    &\left|\frac{1}{2\pi i}\int_{\eta_{\ast}-i\frac{r}{2}}^{\eta_{\ast}+i\frac{r}{2}}e^{\lambda t}\lambda^q
   e^{\left(-\frac{\lambda}{f^{\prime}(\bar{u}_-)}+\frac{\lambda^2}{(f^{\prime}(\bar{u}_-))^3}\right)(x-y)}d\lambda\right|\\
&\leq&e^{-\frac{t}{f^{\prime}(\bar{u}_-)}\frac{\bar{\alpha}^{2}}{p}}\int_{-\frac{r}{2}}^{\frac{r}{2}}
      \mathbf{O}(|\eta_{\ast}|^q + |\xi|^q)e^{-\frac{tp}{f^{\prime}(\bar{u}_-)}\xi^2}d\xi\\
&\leq&e^{-\frac{(f^{\prime}(\bar{u}_-))^2}{x-y}\frac{(x-y-f^{\prime}(\bar{u}_-)t)^2}{4t}}\int_{-\infty}^{\infty}
      \mathbf{O}(|\eta_{\ast}|^q + |\xi|^q)e^{-\frac{p}{f^{\prime}(\bar{u}_-)}\xi^2t}d\xi\\
&\leq&\mathbf{O}\left(t^{-\frac{q+1}{2}}e^{-\frac{(x-y-f^{\prime}(\bar{u}_-)t)^2}{Mt}}\right)
\end{eqnarray*}
if $\left|\frac{\bar{\alpha}}{p}\right|\leq \varepsilon$, and
\begin{eqnarray*}
     \left|\frac{1}{2\pi i}\int_{\eta_{\ast}-i\frac{r}{2}}^{\eta_{\ast}+i\frac{r}{2}}e^{\lambda t}\lambda^q
     e^{\left(-\frac{\lambda}{f^{\prime}(\bar{u}_-)}+\frac{\lambda^2}{(f^{\prime}(\bar{u}_-))^3}\right)(x-y)}d\lambda\right|
&\leq& e^{-\frac{\varepsilon t}{M}}\int_{-\infty}^{\infty}
     \mathbf{O}(|\eta_{\ast}|^q + |\xi|^q)e^{-\frac{p}{f^{\prime}(\bar{u}_-)}\xi^2t}d\xi\\
&\leq& \mathbf{O}\left(t^{-\frac{q+1}{2}}e^{-\eta t}\right)
\end{eqnarray*}
if $\left|\frac{\bar{\alpha}}{p}\right|\geq \varepsilon$. Combining these estimates, we get the bound in \eqref{8.6}.
\end{proof}

\begin{rem}
{\rm The derivation of \eqref{8.24}. \\
To evaluate the integral,
$
\frac{1}{2\pi}\mathrm{P.V.}\int_{-\infty}^{+\infty}(i\xi)^{-1}e^{i\xi\left(t+\frac{y}{f^{\prime}(\bar{u}_{-})}\right)}
   e^{\xi^2\frac{y}{(f^{\prime}(\bar{u}_{-}))^3}}d\xi,
$
we make a change of variable $\zeta=\sqrt{\frac{-y}{(f^{\prime}(\bar{u}_{-}))^3}}\xi$, then
$\frac{y}{(f^{\prime}(\bar{u}_{-}))^3}\xi^2=-\zeta^2$ and
$d\xi=\frac{1}{\sqrt{\frac{-y}{(f^{\prime}(\bar{u}_{-}))^3}}}d\zeta$.
The integral becomes
\begin{align*}
& \frac{1}{2\pi}\int_{-\infty}^{+\infty}\frac{\sqrt{\frac{-y}{(f^{\prime}(\bar{u}_{-}))^3}}}{i\zeta}
   e^{i\left(t+\frac{y}{f^{\prime}(\bar{u}_{-})}\right)\frac{1}{\sqrt{\frac{-y}{(f^{\prime}(\bar{u}_{-}))^3}}}\zeta}
   e^{-\zeta^2}\frac{1}{\sqrt{\frac{-y}{(f^{\prime}(\bar{u}_{-}))^3}}}d\zeta\\
&=\frac{1}{2\pi}\int_{-\infty}^{+\infty}\frac{1}{i\zeta}
   e^{i\left(\frac{y+f^{\prime}(\bar{u}_{-})t}{\sqrt{-\frac{y}{f^{\prime}(\bar{u}_{-})}}}\right)\zeta}e^{-\zeta^2}d\zeta
 =\frac{1}{\sqrt{2\pi}}\frac{1}{\sqrt{2\pi}}\int_{-\infty}^{+\infty}\frac{e^{-\zeta^2}}{i\zeta}
   e^{i\left(\frac{y+f^{\prime}(\bar{u}_{-})t}{\sqrt{-\frac{y}{f^{\prime}(\bar{u}_{-})}}}\right)\zeta}d\zeta\\
&=\frac{1}{\sqrt{2\pi}}\left(\mathcal{F}_{\zeta}^{-1}f(\zeta)\right)
   \left(\frac{y+f^{\prime}(\bar{u}_{-})t}{\sqrt{-\frac{y}{f^{\prime}(\bar{u}_{-})}}}\right)
 =\frac{1}{\sqrt{2\pi}}\left(\mathcal{F}_{\zeta}^{-1}f(\zeta)\right)(\tau):=\frac{1}{\sqrt{2\pi}}g(\tau)
\end{align*}
where $f(\zeta)=\frac{e^{-\zeta^2}}{i\zeta}$,
$\tau=\frac{y+f^{\prime}(\bar{u}_{-})t}{\sqrt{-\frac{y}{f^{\prime}(\bar{u}_{-})}}}$ and
$g(\tau)=\left(\mathcal{F}_{\zeta}^{-1}f(\zeta)\right)(\tau)$.
By the inverse Fourier transform formulae, we have
$f(\zeta)=\left(\mathcal{F}_{\tau}g(\tau)\right)(\zeta)$
and
$i\zeta f(\zeta)=\left(\mathcal{F}_{\tau}g^{\prime}(\tau)\right)(\zeta)=e^{-\zeta^2}$
so
\begin{equation*}
g^{\prime}(\tau)
=\left(\mathcal{F}_{\zeta}^{-1}e^{-\zeta^2}\right)(\tau)
=\frac{1}{\sqrt{2\pi}}\int_{-\infty}^{+\infty}e^{-\zeta^2}e^{i\tau\zeta}d\zeta
=\frac{1}{\sqrt{2}}e^{-\frac{\tau^2}{4}}
=\frac{\sqrt{2\pi}}{2}\mathrm{errfn}^{\prime}\left(\frac{\tau}{2}\right),
\end{equation*}
then
$\frac{1}{\sqrt{2\pi}}g^{\prime}(\tau)=\frac{1}{2}\mathrm{errfn}^{\prime}\left(\frac{\tau}{2}\right)$,
integrate this equation to get
$\frac{1}{\sqrt{2\pi}}g(\tau)=\mathrm{errfn}\left(\frac{\tau}{2}\right)-\frac{1}{2}$
because $g(0)=0$ and $\mathrm{errfn}(0)=\frac{1}{2}$. This completes the proof of \eqref{8.24}.
}
\end{rem}

\begin{rem}
{\rm The reason that we made the excited term $E(x,t;y)$ look like \eqref{8.4} is that we would like to have the Green function decompositon look similar to the scalar Burger's equation case in \cite{Z1}, doing so also makes $E(x,t;y)$ vanishes at $t=0$.}
\end{rem}

\begin{rem}
{\rm The $\frac{e^{-x}}{e^x+e^{-x}}$ term in the scattering term serves as a smooth cutoff function, which smoothly interpolate between different cases of solutions. For $x>0$ and $|x|$ large, $\frac{e^{-x}}{e^x+e^{-x}}$ decays to $0$, for $x<0$ and $|x|$ large, $\frac{e^{-x}}{e^x+e^{-x}}$ is almost $1$.}
\end{rem}

Now it is time to give some $L^p$ estimates on Green function convolved with some function $f$ in $L^{p}(\mathbb{R})$
($1\leq p\leq \infty$).

\begin{pro} \label{Pro8.7}
The Green function $G$ decomposes as $G=E+S+R=E+\tilde{G}$, where $\tilde{G}=S+R$,
$E(x,t;y)=C\bar{u}^{\prime}(x)e(y,t)$ and
$$e(y,t):=\mathrm{errfn}\left(\frac{y+f^{\prime}(\bar{u}_{-})t}{\sqrt{4t}}\right)
 -\mathrm{errfn}\left(\frac{y-f^{\prime}(\bar{u}_{-})t}{\sqrt{4t}}\right)$$
then for some $C>0$ and all $t>0$,
\begin{equation} \label{8.41}
\left|\int_{-\infty}^{+\infty}\tilde{G}(x,t;y)h(y)dy\right|_{L^{p}(x)}\leq Ct^{-\frac{1}{2}(1-\frac{1}{p})}|h|_{L^1},
\end{equation}
\begin{equation} \label{8.42}
\left|\int_{-\infty}^{+\infty}\tilde{G}_{y}(x,t;y)h(y)dy\right|_{L^{p}(x)}\leq Ct^{-\frac{1}{2}(1-\frac{1}{p})-\frac{1}{2}}|h|_{L^1},
\end{equation}
\begin{equation} \label{8.43}
\left|\int_{-\infty}^{+\infty}\tilde{G}(x,t;y)h(y)dy\right|_{L^{p}(x)}\leq C|h|_{L^p},
\end{equation}
\begin{equation} \label{8.44}
\left|\int_{-\infty}^{+\infty}\tilde{G}_{y}(x,t;y)h(y)dy\right|_{L^{p}(x)}\leq Ct^{-\frac{1}{2}}|h|_{L^p}.
\end{equation}
and
\begin{equation} \label{8.45}
\left|\int_{-\infty}^{+\infty}e(y,t)h(y)dy\right|\leq C|h|_{L^1},
\left|\int_{-\infty}^{+\infty}e_{y}(y,t)h(y)dy\right|\leq Ct^{-\frac{1}{2}}|h|_{L^1},
\end{equation}
\begin{equation} \label{8.46}
\left|\int_{-\infty}^{+\infty}e_{t}(y,t)h(y)dy\right|\leq Ct^{-\frac{1}{2}}|h|_{L^1},
\left|\int_{-\infty}^{+\infty}e_{ty}(y,t)h(y)dy\right|\leq Ct^{-1}|h|_{L^1},
\end{equation}
\begin{equation} \label{8.47}
\left|\int_{-\infty}^{+\infty}e_{t}(y,t)h(y)dy\right|\leq C|h|_{L^{\infty}},
\left|\int_{-\infty}^{+\infty}e_{yt}(y,t)h(y)dy\right|\leq Ct^{-\frac{1}{2}}|h|_{L^{\infty}}.
\end{equation}
\end{pro}

\begin{proof}
We prove \eqref{8.41} first. Write $\tilde{G}$ as $\tilde{G}(x,t;y)=S(x,t;y)+R(x,t;y)$,
recall from Theorem \ref{greenfunctionbounds} that
$S(x,t;y)=\chi_{\{t \geq 1\}}(4\pi t)^{-\frac{1}{2}}e^{-\frac{\left(x-y-f^{\prime}(\bar{u}_{-})t\right)^2}{4t}}\left(\frac{e^{-x}}{e^x+e^{-x}}\right)$
so
\begin{equation} \label{8.48}
\left|S(x,t;y)\right|\leq (4\pi t)^{-\frac{1}{2}}e^{-\frac{\left(x-y-f^{\prime}(\bar{u}_{-})t\right)^2}{4t}}
\end{equation}
and
$R(x,t;y)=\mathbf{O}\left(e^{-\eta t}e^{-\frac{|x-y|^2}{Mt}}\right)
 +\mathbf{O}\left((t+1)^{-\frac{1}{2}}e^{-\eta x^+}+e^{-\eta |x|}\right)t^{-\frac{1}{2}}e^{-\frac{(x-y-f^{\prime}(\bar{u}_{-})t)^2}{Mt}}$
so
\begin{equation} \label{8.49}
\left|R(x,t;y)\right|\leq Ce^{-\eta |x|}t^{-\frac{1}{2}}e^{-\frac{(x-y-f^{\prime}(\bar{u}_{-})t)^2}{Mt}}
+Ce^{-\eta^{\prime}(|x-y|+t)}
\end{equation}
for some $0<\eta^{\prime}<\eta$.
By Minkowski's inequality,
$$\left|\tilde{G}(x,t;y)\right|_{L^{p}(x)}\leq \left|S(x,t;y)\right|_{L^{p}(x)}+\left|R(x,t;y)\right|_{L^{p}(x)}$$
We estimate $\left|S(x,t;y)\right|_{L^{p}(x)}$ first,
\begin{eqnarray*}
\left|S(x,t;y)\right|_{L^{p}(x)}
&\leq&\left(\int_{-\infty}^{+\infty}\left[(4\pi t)^{-\frac{1}{2}}e^{-\frac{\left(x-y-f^{\prime}(\bar{u}_{-})t\right)^2}{4t}}\right]^{p}dx\right)^{\frac{1}{p}}\\
&\leq&(4\pi t)^{-\frac{1}{2}}\left(\int_{-\infty}^{+\infty}
e^{-\frac{p}{4t}\left(x-y-f^{\prime}(\bar{u}_{-})t\right)^2}dx\right)^{\frac{1}{p}}\\
&=&(4\pi t)^{-\frac{1}{2}}\left(\sqrt{\frac{4t}{p}}\int_{-\infty}^{+\infty}e^{-z^2}dz\right)^{\frac{1}{p}}\\
&=&(4\pi t)^{-\frac{1}{2}}\left(\sqrt{\frac{4t}{p}}\sqrt{\pi}\right)^{\frac{1}{p}}
=C_1 t^{-\frac{1}{2}(1-\frac{1}{p})}
\end{eqnarray*}
Then we estimate $\left|R(x,t;y)\right|_{L^{p}(x)}$,
\begin{eqnarray*}
\left|R(x,t;y)\right|_{L^{p}(x)}
\leq\left|Ce^{-\eta |x|}t^{-\frac{1}{2}}e^{-\frac{(x-y-f^{\prime}(\bar{u}_{-})t)^2}{Mt}}\right|_{L^{p}(x)}
+\left|Ce^{-\eta^{\prime}(|x-y|+t)}\right|_{L^{p}(x)}
:=R_A+R_B.
\end{eqnarray*}
In fact $R_A$ is similar to $\left|S(x,t;y)\right|_{L^{p}(x)}$, so we have
$R_A\leq C_2 t^{-\frac{1}{2}(1-\frac{1}{p})}$. Estimate $R_B$ as,
\begin{eqnarray*}
R_B^p
&=&C^p\int_{-\infty}^{+\infty}\left(e^{-\eta^{\prime}(|x-y|+t)}\right)^p dx
=C^p\int_{-\infty}^{+\infty}e^{-p\eta^{\prime}(|x-y|+t)} dx\\
&=&C^p\int_{-\infty}^{y}e^{-p\eta^{\prime}(y-x+t)} dx+C^p\int_{y}^{+\infty}e^{-p\eta^{\prime}(x-y+t)} dx\\
&=&C^p e^{-p\eta^{\prime}t}\left(\int_{-\infty}^{y}e^{-p\eta^{\prime}(y-x)} dx
+\int_{y}^{+\infty}e^{-p\eta^{\prime}(x-y)} dx\right)\\
&=&2C^p e^{-p\eta^{\prime}t}\int_{0}^{+\infty}e^{-p\eta^{\prime}x} dx
=\frac{2C^p e^{-p\eta^{\prime}t}}{p\eta^{\prime}}
\end{eqnarray*}
so
$R_B=C\left(\frac{2}{p\eta^{\prime}}\right)^{\frac{1}{p}}e^{-\eta^{\prime}t}$.

Finally, we use the above estimates to derive,
\begin{eqnarray*}
&    &\left|\int_{-\infty}^{+\infty}\tilde{G}(x,t;y)h(y)dy\right|_{L^{p}(x)}\\
&\leq&\int_{-\infty}^{+\infty}\left|\tilde{G}(x,t;y)\right|_{L^{p}(x)}\left|h(y)\right|dy\\
&\leq&\int_{-\infty}^{+\infty}\left(\left|S(x,t;y)\right|_{L^{p}(x)}+\left|R(x,t;y)\right|_{L^{p}(x)}\right)\left|h(y)\right|dy\\
&\leq&\int_{-\infty}^{+\infty}\left(C_1t^{-\frac{1}{2}(1-\frac{1}{p})}+C_2t^{-\frac{1}{2}(1-\frac{1}{p})}
      +C\left(\frac{2}{p\eta^{\prime}}\right)^{\frac{1}{p}}e^{-\eta^{\prime}t}\right)\left|h(y)\right|dy\\
&\leq&\int_{-\infty}^{+\infty}(C_1+C_2+C_3)t^{-\frac{1}{2}(1-\frac{1}{p})}\left|h(y)\right|dy
 =  Ct^{-\frac{1}{2}(1-\frac{1}{p})}|h|_{L^1}.
\end{eqnarray*}
For $y$-derivative bounds of $\tilde{G}$, $\tilde{G}_{y}(x,t;y)$, we just need to notice that we have the following estimates
\begin{equation} \label{8.50}
\left|S_{y}(x,t;y)\right|\leq C t^{-1}e^{-\frac{\left(x-y-f^{\prime}(\bar{u}_{-})t\right)^2}{4t}}
\end{equation}
and
$R_{y}(x,t;y)
=\mathbf{O}\left(e^{-\eta t}e^{-\frac{|x-y|^2}{Mt}}\right)
+\mathbf{O}\left((t+1)^{-\frac{1}{2}}e^{-\eta x^+}+e^{-\eta |x|}\right)t^{-1}e^{-\frac{(x-y-f^{\prime}(\bar{u}_{-})t)^2}{Mt}}$,
so
\begin{equation} \label{8.51}
\left|R_{y}(x,t;y)\right|\leq Ce^{-\eta |x|}t^{-1}e^{-\frac{(x-y-f^{\prime}(\bar{u}_{-})t)^2}{Mt}}
+Ce^{-\eta^{\prime}(|x-y|+t)}
\end{equation}
for some $0<\eta^{\prime}<\eta$. With some similar computation as the for $\tilde{G}(x,t;y)$ we get
\begin{equation*}
\left|\int_{-\infty}^{+\infty}\tilde{G}_{y}(x,t;y)h(y)dy\right|_{L^{p}(x)}\leq Ct^{-\frac{1}{2}(1-\frac{1}{p})-\frac{1}{2}}|h|_{L^1},
\end{equation*}
To prove \eqref{8.43} and \eqref{8.44}, notice that \eqref{8.48} implies
\begin{equation} \label{8.52}
\left|S(x,t;y)\right|_{L^{1}(x)}\leq \left|(4\pi t)^{-\frac{1}{2}}e^{-\frac{\left(x-y-f^{\prime}(\bar{u}_{-})t\right)^2}{4t}}\right|_{L^{1}(x)}\leq C,
\end{equation}
$(\ref{8.49})$ implies
\begin{equation} \label{8.53}
\left|R(x,t;y)\right|_{L^{1}(x)}\leq C\left|e^{-\eta |x|}t^{-\frac{1}{2}}e^{-\frac{(x-y-f^{\prime}(\bar{u}_{-})t)^2}{Mt}}\right|_{L^{1}(x)}
+C\left|e^{-\eta^{\prime}(|x-y|+t)}\right|_{L^{1}(x)}\leq C,
\end{equation}
\eqref{8.50} implies
\begin{equation} \label{8.54}
\left|S_{y}(x,t;y)\right|_{L^{1}(x)}\leq C \left|t^{-1}e^{-\frac{\left(x-y-f^{\prime}(\bar{u}_{-})t\right)^2}{4t}}\right|_{L^{1}(x)}\leq Ct^{-\frac{1}{2}}
\end{equation}
\eqref{8.51} implies
\begin{equation} \label{8.55}
\left|R_{y}(x,t;y)\right|_{L^{1}(x)}\leq C\left|e^{-\eta |x|}t^{-1}e^{-\frac{(x-y-f^{\prime}(\bar{u}_{-})t)^2}{Mt}}\right|_{L^{1}(x)}
+C\left|e^{-\eta^{\prime}(|x-y|+t)}\right|_{L^{1}(x)}\leq Ct^{-\frac{1}{2}}.
\end{equation}
Use estimates \eqref{8.52}-\eqref{8.55} and the inequality $|f\ast g|_{L^{p}}\leq |f|_{L^{1}}|g|_{L^{p}}$ we can derive the estimates \eqref{8.43} and \eqref{8.44}.
To prove bounds \eqref{8.45}-\eqref{8.47} we need the following lemma.
\end{proof}

\begin{lem} \label{Lem8.8}
For some $C>0$ and all $t>0$,
 \begin{eqnarray}
 |e(\cdot,t)|_{L^{\infty}}&\leq& C,\label{8.56}\\
 |e_y(\cdot,t)|_{L^{p}}&\leq& Ct^{-\frac{1}{2}(1-\frac{1}{p})},\label{8.57}\\
 |e_t(\cdot,t)|_{L^{p}}&\leq& Ct^{-\frac{1}{2}(1-\frac{1}{p})},\label{8.58}\\
 |e_{ty}(\cdot,t)|_{L^{p}}&\leq& Ct^{-\frac{1}{2}(1-\frac{1}{p})-\frac{1}{2}},\label{8.59}\\
 |e_y(y,t)|&\leq& Ct^{-\frac{1}{2}}\left(e^{-\frac{(-y-t)^2}{Ct}}+e^{-\frac{(-y+t)^2}{Ct}}\right),\label{8.60}\\
 |e_t(y,t)|&\leq& Ct^{-\frac{1}{2}}\left(e^{-\frac{(-y-t)^2}{Ct}}+e^{-\frac{(-y+t)^2}{Ct}}\right),\label{8.61}\\
 |e_{ty}(y,t)|&\leq& Ct^{-1}\left(e^{-\frac{(-y-t)^2}{Ct}}+e^{-\frac{(-y+t)^2}{Ct}}\right).\label{8.62}
 \end{eqnarray}
\end{lem}

\begin{proof}
Estimates \eqref{8.56} follows directly from the definition of $e(y,t)$:
\begin{eqnarray*}
e(y,t)
&:=&\mathrm{errfn}\left(\frac{y+f^{\prime}(\bar{u}_{-})t}{\sqrt{4t}}\right)
    -\mathrm{errfn}\left(\frac{y-f^{\prime}(\bar{u}_{-})t}{\sqrt{4t}}\right)\\
& =&\frac{1}{\sqrt{\pi}}\int_{-\infty}^{\frac{y+f^{\prime}(\bar{u}_{-})t}{\sqrt{4t}}}e^{-z^2}dz
    -\frac{1}{\sqrt{\pi}}\int_{-\infty}^{\frac{y-f^{\prime}(\bar{u}_{-})t}{\sqrt{4t}}}e^{-z^2}dz
 =\frac{1}{\sqrt{\pi}}\int_{\frac{y-f^{\prime}(\bar{u}_{-})t}{\sqrt{4t}}}
    ^{\frac{y+f^{\prime}(\bar{u}_{-})t}{\sqrt{4t}}}e^{-z^2}dz
\end{eqnarray*}
Differentiating the above equation with respect to $y$ we get,
\begin{eqnarray*}
e_y(y,t)
=\frac{1}{\sqrt{\pi}}\left(\frac{e^{-\frac{(y+f^{\prime}(\bar{u}_{-})t)^2}{4t}}}{\sqrt{4t}}
 -\frac{e^{-\frac{(y-f^{\prime}(\bar{u}_{-})t)^2}{4t}}}{\sqrt{4t}}\right)
\end{eqnarray*}
yielding \eqref{8.60}.

Differentiating with respect to $t$ we get,
\begin{eqnarray*}
e_t(y,t)
&=&f^{\prime}(\bar{u}_{-})\left(\frac{e^{-\frac{(y+f^{\prime}(\bar{u}_{-})t)^2}{4t}}}{\sqrt{4\pi t}}
   +\frac{e^{-\frac{(y-f^{\prime}(\bar{u}_{-})t)^2}{4t}}}{\sqrt{4\pi t}}\right)\\
& &+\frac{1}{2\sqrt{t}}
   \left(\frac{(y-f^{\prime}(\bar{u}_{-})t)}{\sqrt{t}}\frac{e^{-\frac{(y-f^{\prime}(\bar{u}_{-})t)^2}{4t}}}{\sqrt{4\pi t}}
   -\frac{(y+f^{\prime}(\bar{u}_{-})t)}{\sqrt{t}}\frac{e^{-\frac{(y+f^{\prime}(\bar{u}_{-})t)^2}{4t}}}{\sqrt{4\pi t}}\right)
\end{eqnarray*}
then we get \eqref{8.61} immediately if $t\geq 1$. For the case when $0<t<1$, we use the mean value theorem,
\begin{eqnarray*}
& &\left|\frac{(y-f^{\prime}(\bar{u}_{-})t)}{\sqrt{t}}\frac{e^{-\frac{(y-f^{\prime}(\bar{u}_{-})t)^2}{4t}}}{\sqrt{4\pi t}}
   -\frac{(y+f^{\prime}(\bar{u}_{-})t)}{\sqrt{t}}\frac{e^{-\frac{(y+f^{\prime}(\bar{u}_{-})t)^2}{4t}}}{\sqrt{4\pi t}}\right|\\
&=&t\left|\int_{-f^{\prime}(\bar{u}_{-})}^{f^{\prime}(\bar{u}_{-})}
   \frac{\partial}{\partial z}\left(\frac{z}{\sqrt{t}}\frac{e^{-\frac{z^2}{4t}}}{\sqrt{4\pi t}}\right)
   |_{z=-y+\theta t}d\theta\right|
\leq 2C_0 t\left|\frac{\partial}{\partial z}\left(\frac{z}{\sqrt{t}}\frac{e^{-\frac{z^2}{4t}}}{\sqrt{4\pi t}}\right)
      |_{z=-y}\right|\\
&=&2C_0 t\frac{e^{-\frac{y^2}{4t}}|2t-y^2|}{4\sqrt{\pi}t^2}
\leq C\left(e^{-\frac{(-y-t)^2}{Ct}}+e^{-\frac{(-y+t)^2}{Ct}}\right),
\end{eqnarray*}
we get \eqref{8.61} again. Estimate \eqref{8.62} can be derived similarly.\\
Now the estimates \eqref{8.57}-\eqref{8.59} follow as in the heat kernel estimates.
\end{proof}

Once we have the above lemma, \eqref{8.45}-\eqref{8.47} can be derived similarly as previous estimates on $\tilde{G}$.

\section{Nonlinear stability} \label{sectionstability}

In this section we will establish the stability results of the viscous shock solutions of \eqref{0.1}, proving
Theorem \ref{mainthm}. Let $\tilde{u}$ be a second
solution of \eqref{0.1}, define the perturbation
\begin{equation}
u(x,t):=\tilde{u}(x+\alpha(t),t)-\bar{u}(x) \label{per}
\end{equation}
as the difference between a translate of $\tilde{u}$ and the background wave $\bar{u}$, where the translation $\alpha(t)$ is to
be determined later. After a bit of computation, we derive the \textit{perturbation equation}
\begin{equation}
u_t-Lu=N(u)_x+\dot{\alpha}(t)(u_x+\bar{u}_x), \label{pereq}
\end{equation}
where $Lu=u_{xx} - (f^{\prime}(\bar{u})u)_x$ is the linearized differential operator as before, and
$N(u):=f(\bar{u})+f^{\prime}(\bar{u})u-f(u+\bar{u})$.

\begin{rem}
{\rm From direct calculation $N(u)$ ought to be
$N(u):=f^{\prime}(\bar{u})u-f(u+\bar{u})+\bar{u}_x,$
but notice that $\bar{u}_{xx}=f(\bar{u})_{x}$, thus in \eqref{pereq} if we substitute $N(u)$ with the one defined above will
make no difference.}
\end{rem}

Starting from the perturbation equation, we apply Duhamel's principle to \eqref{pereq},
\begin{eqnarray*}
u(x,t)
&=&e^{Lt}u_0+\int_{0}^{t}e^{L(t-s)}\left(N(u)_x+\dot{\alpha}(s)(\bar{u}_x+u_x)\right)ds\\
&=&\int_{-\infty}^{+\infty}G(x,t;y)u_{0}(y)dy\\
& &+\int_{0}^{t}\int_{-\infty}^{+\infty}G(x,t-s;y)\left[(N(u))_x(y,s)+\dot{\alpha}(s)u_x(y,s)\right]dyds\\
& &+\int_{0}^{t}\int_{-\infty}^{+\infty}G(x,t-s;y)\dot{\alpha}(s)\bar{u}_x(y)dyds\\
&=&\int_{-\infty}^{+\infty}G(x,t;y)u_{0}(y)dy\\
& &-\int_{0}^{t}\int_{-\infty}^{+\infty}G_y(x,t-s;y)\left[(N(u))(y,s)+\dot{\alpha}(s)u(y,s)\right]dyds\\
& &+\int_{0}^{t}G(x,t-s;y)\left[(N(u))(y,s)+\dot{\alpha}(s)u(y,s)\right]_{-\infty}^{+\infty}ds\\
& &+\left(\int_{0}^{t}\dot{\alpha}(s)ds\right)\left(\int_{-\infty}^{+\infty}G(x,t-s;y)\bar{u}_x(y)dy\right)\\
&=&\int_{-\infty}^{+\infty}G(x,t;y)u_{0}(y)dy\\
& &-\int_{0}^{t}\int_{-\infty}^{+\infty}G_y(x,t-s;y)\left[(N(u))(y,s)+\dot{\alpha}(s)u(y,s)\right]dyds
   +\alpha(t)\bar{u}_x(x).
\end{eqnarray*}

To determine $\alpha$, recall the decomposition $G=C\bar{u}_{x}e(y,t)+\tilde{G}$ from Proposition \ref{Pro8.7}, substitute this
decomposition into formulae of $u(x,t)$ derived above,
\begin{eqnarray*}
& &u(x,t)
   =\int_{-\infty}^{+\infty}\tilde{G}(x,t;y)u_{0}(y)dy
   -\int_{0}^{t}\int_{-\infty}^{+\infty}\tilde{G}_y(x,t-s;y)\left[N(u)+\dot{\alpha}u\right]dyds\\
& &+C\bar{u}_x(x)\left\{\alpha(t)+\int_{-\infty}^{+\infty}e(y,t)u_0(y)dy
   -\int_{0}^{t}\int_{-\infty}^{+\infty}e_y(y,t-s)[N(u)+\dot{\alpha}u]dyds\right\}.
\end{eqnarray*}
So to make the term $\bar{u}_x$ disappear, we have to define $\alpha$ as
\begin{equation*}
\begin{aligned}
\alpha(t)=-\int_{-\infty}^{+\infty}e(y,t)u_0(y)dy
\quad +\int_{0}^{t}\int_{-\infty}^{+\infty}e_y(y,t-s)(N(u)(y,s)+\dot{\alpha}(s)u(y,s))dyds.
\end{aligned}
\end{equation*}
We obtain the \textit{integral representation},
\begin{equation*}
\begin{aligned}
u(x,t)
=\int_{-\infty}^{+\infty}\tilde{G}(x,t;y)u_{0}(y)dy
 -\int_{0}^{t}\int_{-\infty}^{+\infty}\tilde{G}_y(x,t-s;y)\left(N(u)(y,s)+\dot{\alpha}(s)u(y,s)\right)dyds
\end{aligned}
\end{equation*}
differentiating $\alpha(t)$ with respect to $t$,
\begin{equation*}
\begin{aligned}
\dot{\alpha}(t)
=-\int_{-\infty}^{+\infty}e_t(y,t)u_0(y)dy
 +\int_{0}^{t}\int_{-\infty}^{+\infty}e_{yt}(y,t-s)(N(u)(y,s)+\dot{\alpha}(s)u(y,s))dyds.
\end{aligned}
\end{equation*}

\begin{lem}
Define
 \begin{equation}
 \zeta(t):=\sup_{0\leq s\leq t, 1\leq p\leq\infty}\left(|u|_{L^p}(s)(1+s)^{\frac{1}{2}(1-\frac{1}{p})}
 +|\dot{\alpha}(s)|(1+s)^{\frac{1}{2}}\right),
 \end{equation}
then for all $t\geq 0$ for which $\zeta(t)$ is finite, some $C>0$, and $E_0:=|u_0|_{L^1\cap L_{\infty}}$,
 \begin{equation}
 \zeta(t)\leq C(E_0+\zeta(t)^2). \label{zeta_est}
 \end{equation}
\end{lem}

\begin{proof}
Use Taylor expansion of $f(u+\bar{u})$ in $N(u):=f(\bar{u})+f^{\prime}(\bar{u})u-f(u+\bar{u})$ to derive
$N(u)=-\frac{1}{2}f^{\prime\prime}(\bar{u})u^2+\mathbf{O}(u^3)=\mathbf{O}(u^2)$.

We have then the following estimates for $|N(u)+\dot{\alpha}u|_{L^1}(s)$ and $|N(u)+\dot{\alpha}u|_{L^p}(s)$,
 \begin{align*}
 &\quad|N(u)+\dot{\alpha}u|_{L^1}(s)\\
 &\leq C|u|_{L^1}(s)(|u|_{L^{\infty}}(s)+|\dot{\alpha}(s)|)\\
 &\leq C\left(|u|_{L^1}(s)+|\dot{\alpha}(s)|(1+s)^{\frac{1}{2}}\right)(1+s)^{-\frac{1}{2}}
        \left(|u|_{L^{\infty}}(s)(1+s)^{\frac{1}{2}}+|\dot{\alpha}(s)|(1+s)^{\frac{1}{2}}\right)\\
 &\leq C\zeta(s)^2 (1+s)^{-\frac{1}{2}}
 \end{align*}
and
 \begin{align*}
       |N(u)+\dot{\alpha}u|_{L^p}(s)
 &\leq C|u|_{L^p}(s)(|u|_{L^{\infty}}(s)+|\dot{\alpha}(s)|)\\
 &\leq C(1+s)^{-\frac{1}{2}(1-\frac{1}{p})}\left(|u|_{L^p}(s)(1+s)^{\frac{1}{2}(1-\frac{1}{p})}
       +|\dot{\alpha}(s)|(1+s)^{\frac{1}{2}}\right)\\
 &\quad(1+s)^{-\frac{1}{2}}
       \left(|u|_{L^{\infty}}(s)(1+s)^{\frac{1}{2}}+|\dot{\alpha}(s)|(1+s)^{\frac{1}{2}}\right)\\
 &\leq C\zeta(s)^2 (1+s)^{-\frac{1}{2}(1-\frac{1}{p})-\frac{1}{2}}.
 \end{align*}

Since we already have the bounds of $\tilde{G}$ and $e$, we can now estimate $|u(\cdot,t)|_{L^p(x)}$. Utilize the
representation of $u(x,t)$,
 \begin{align*}
       |u(\cdot,t)|_{L^p(x)}
 &\leq C(1+t)^{-\frac{1}{2}(1-\frac{1}{p})}E_0
       +C\int_{0}^{\frac{t}{2}}(t-s)^{-\frac{1}{2}(1-\frac{1}{p})-\frac{1}{2}}|N(u)+\dot{\alpha}u|_{L^1}(s)ds\\
 &\quad+C\int_{\frac{t}{2}}^{t}(t-s)^{-\frac{1}{2}}|N(u)+\dot{\alpha}u|_{L^p}(s)ds\\
 &\leq C(1+t)^{-\frac{1}{2}(1-\frac{1}{p})}E_0
       +C\int_{0}^{\frac{t}{2}}(t-s)^{-\frac{1}{2}(1-\frac{1}{p})-\frac{1}{2}}\zeta(s)^2 (1+s)^{-\frac{1}{2}}ds\\
 &\quad+C\int_{\frac{t}{2}}^{t}(t-s)^{-\frac{1}{2}}\zeta(s)^2 (1+s)^{-\frac{1}{2}(1-\frac{1}{p})-\frac{1}{2}}ds\\
 &\leq C(E_0+\zeta(t)^2)(1+t)^{-\frac{1}{2}(1-\frac{1}{p})}
 \end{align*}

Similarly, for $|\dot{\alpha}(t)|$, we have
 \begin{align*}
       |\dot{\alpha}(t)|
 &\leq C(1+t)^{-\frac{1}{2}}E_0
       +C\int_{0}^{\frac{t}{2}}(t-s)^{-1}|N(u)+\dot{\alpha}u|_{L^1}(s)ds\\
 &\quad+C\int_{\frac{t}{2}}^{t}(t-s)^{-\frac{1}{2}}|N(u)+\dot{\alpha}u|_{L^{\infty}}(s)ds\\
 &\leq C(1+t)^{-\frac{1}{2}}E_0
       +C\int_{0}^{\frac{t}{2}}(t-s)^{-1}\zeta(s)^2 (1+s)^{-\frac{1}{2}}ds\\
 &\quad+C\int_{\frac{t}{2}}^{t}(t-s)^{-\frac{1}{2}}\zeta(s)^2 (1+s)^{-1}ds\\
 &\leq C(E_0+\zeta(t)^2)(1+t)^{-\frac{1}{2}}
 \end{align*}

Rearranging the above two estimates together we obtain \eqref{zeta_est}.
\end{proof}

Now we can prove our main result, Theorem \ref{mainthm}.
\begin{proof}[Proof of Thoerem \ref{mainthm}]
The first two bounds are proved by continuous induction. The third follows using \eqref{8.45} and \eqref{8.46}.
To show the last inequality, notice that
$$\tilde{u}(x,t)-\bar{u}(x)=u(x-\alpha(t),t)-(\bar{u}(x)-\bar{u}(x-\alpha(t))),$$
so that $|\tilde{u}(\cdot,t)-\bar{u}|$ is controlled by the sum of $|u|$ and
$\bar{u}-\bar{u}(x-\alpha(t))=\alpha(t)|\bar{u}^{\prime}(x)|$, hence remains $\leq CE_0$ for all $t\geq 0$, for $E_0$
sufficiently small.
\end{proof}

This completes the proof of Theorem \ref{mainthm}, giving the nonlinear stability.

\section{Small-amplitude Green function bounds}

In this section we consider the small-amplitude stationary
profiles $\bar u^\epsilon(x)$ of one-dimensional strictly parabolic
viscous conservation law \eqref{generalscalar} in a neighborhood of a particular state $u_0$.
\begin{equation}
u_t + f(u)_x = u_{xx}, \label{generalscalar}
\end{equation}
where $u \in C^2(\mathbb{R}^2 \to \mathbb{R}), f \in C^2(\mathbb{R} \to \mathbb{R})$, $u=u(x,t)$,
$u=\bar{u}^{\epsilon}(x-st)$, $\lim_{z \to \pm \infty} \bar{u}^{\epsilon}(z) = u^{\epsilon}_{\pm}$.
After linearizing \eqref{generalscalar} about $\bar{u}^{\epsilon}$, we have
\begin{eqnarray}
v_t=Lv:=v_{xx}-(av)_x=v_{xx}-a_x v-av_x, \label{generalscalarlinearized}
\end{eqnarray}
where $a=a(x)=f^{\prime}(\bar{u}^{\epsilon}(x))$.
Here, $\epsilon>0$ denotes the shock strength $\epsilon:=|u^\epsilon_+-u^\epsilon_-|$, and profiles $\bar u^\epsilon(\cdot)$
are assumed to converge as $\epsilon\to 0$ to $u_0$.  Under these
assumptions, the center-manifold argument of Majda and Pego yields
convergence after rescaling of $\bar u^\epsilon$ to the standard
Burgers profile \eqref{burgersprofile} below, which we shall describe here.

We compare \eqref{generalscalar} with the \textit{Burgers equation},
\begin{equation}
u_t + \left(\frac{u^2}{2}\right)_x= u_{xx}, \label{burgers}
\end{equation}
which approximately describes small-amplitude viscous behavior in the principal characteristic mode to our general scalar
conservation law \eqref{generalscalar}.
In particular, the family of exact solutions
\begin{equation}
\tilde{\bar{u}}^{\epsilon}(x):= -\epsilon \tanh\left(\frac{\epsilon x}{2}\right), \label{burgersprofile}
\end{equation}
give an asymptotic description of the structure of weak viscous shock profiles in the principal direction, in the limit as amplitude $\epsilon$ goes to zero.

\medskip
{\bf The rescaled profile equations.}
Let $\tilde{\eta}:=\bar{u}^{\epsilon}-\frac{u_{+}+u_{-}}{2}$, $\epsilon:=|u_{+}-u_{-}|$,
we have the reduced flow $\tilde{\eta}$ satisfying,
\begin{equation}
\beta \tilde{\eta}_{\tilde{x}}=\frac{\Lambda}{2}(\tilde{\eta}^2-\epsilon^2)+\mathbf{O}(|\tilde{\eta},\epsilon|^3)
\label{generalreduced}
\end{equation}
where $\beta$ and $\Lambda$ are positive constants, after rescaling
$\tilde{\eta} \to \eta=\tilde{\eta}/\epsilon, \tilde{x} \to x=\Lambda \epsilon \tilde{x}/\beta$
we get
\begin{equation}
\eta^{\prime}=\frac{1}{2}(\eta^2-1)+\epsilon Q(\eta,\epsilon)
\label{generalreducedrescale}
\end{equation}
where $Q(\eta,\epsilon)\in C^1(\mathbb{R}^2)$ and $^{\prime}$ denotes $\frac{d}{dx}$.
Compare this to the exact profile equation of Burgers
\begin{equation}
\bar{\eta}^{\prime}=\frac{1}{2}(\bar{\eta}^2-1) \label{burgerprofileequation}.
\end{equation}
Standard Implicit Function Theorem and stable/unstable manifolds estimates give us,
\begin{pro} \label{ProProfile}
There holds
 \begin{align}
 |\eta_{\pm}-\bar{\eta}_{\pm}| &\leq C\epsilon, \label{endpointsestimate} \\
 |(\eta-\eta_{\pm})-(\bar{\eta}-\bar{\eta}_{\pm})| &\leq C\epsilon e^{-\theta|x|} \label{reducedprofileestimate}
 \end{align}
for $x\gtrless 0$ and for any fixed $0<\theta<1$, and some $C=C(\theta)>0$, where $\eta_{\pm}$ and $\bar{\eta}_{\pm}=\mp 1$ denote the rest points of \eqref{generalreducedrescale} and \eqref{burgerprofileequation}, respectively.
\end{pro}
This suggests that we can expect similar stability results of profiles $\bar{u}^{\epsilon}(x)$ as that of
$\tilde{\bar{u}}^{\epsilon}(x)$.

Using our former results on Green function bounds of the general scalar conservation law, we can derive the Green bounds under the rescalled coordinates. First, let's recall Theorem \ref{greenfunctionbounds}, we rescale it using
$\tilde{x}=\Lambda\epsilon x, \tilde{t}=\Lambda^2\epsilon^2 t,
\tilde{u}=\frac{u}{\Lambda\epsilon}, \tilde{a}=\frac{1}{\Lambda\epsilon}a$.
Note that from now on the coordinate system $(x,t;y)$-$u$ denotes the coordinate system of the small amplitude problem
($\epsilon$-dependent case), and the coordinate system $(\tilde{x},\tilde{t};\tilde{y})$-$\tilde{u}$ denotes the rescaled small
amplitude problem(back to normal amplitude case or $\epsilon$-independent case).
We define
$\tilde{u}(\tilde{x},\tilde{t})$ and the flux function $\tilde{f}$ as
$
\tilde{u}(\tilde{x},\tilde{t})=\frac{1}{\Lambda\epsilon}u(x,t), \quad
\tilde{f}(\cdot)=\frac{1}{\Lambda^2\epsilon^2}f(\Lambda\epsilon(\cdot)).
$
If we substitute $u$ by $v$, $a^{\epsilon}$, we get
$
\tilde{v}(\tilde{x},\tilde{t})=\frac{1}{\Lambda\epsilon}v(x,t), \quad
\tilde{a}(\tilde{x})=\frac{1}{\Lambda\epsilon}a^{\epsilon}(x)
$.
Now we can compute the derivatives
\begin{eqnarray*}
\tilde{u}_{\tilde{x}}(\tilde{x},\tilde{t})&=&\frac{\partial}{\partial\tilde{x}}\left[\tilde{u}(\tilde{x},\tilde{t})\right]
=\frac{1}{\Lambda\epsilon}\frac{\partial}{\partial x}\left[\frac{1}{\Lambda\epsilon}u(x,t)\right]=\frac{1}{\Lambda^2\epsilon^2}u_x(x,t),\\
\tilde{u}_{\tilde{x}\tilde{x}}(\tilde{x},\tilde{t})
&=&\frac{\partial}{\partial\tilde{x}}\left[\frac{1}{\Lambda^2\epsilon^2}u_x(x,t)\right]
=\frac{1}{\Lambda\epsilon}\frac{\partial}{\partial x}\left[\frac{1}{\Lambda^2\epsilon^2}u_{x}(x,t)\right]
=\frac{1}{\Lambda^3\epsilon^3}u_{xx}(x,t),\\
\tilde{u}_{\tilde{t}}(\tilde{x},\tilde{t})&=&\frac{\partial}{\partial\tilde{t}}\left[\tilde{u}(\tilde{x},\tilde{t})\right]
=\frac{1}{\Lambda^2\epsilon^2}\frac{\partial}{\partial t}\left[\frac{1}{\Lambda\epsilon}u(x,t)\right]=\frac{1}{\Lambda^3\epsilon^3}u_t(x,t).
\end{eqnarray*}
Thus
$
\tilde{u}_{\tilde{t}}(\tilde{x},\tilde{t})+\tilde{u}(\tilde{x},\tilde{t})\tilde{u}_{\tilde{x}}(\tilde{x},\tilde{t})
-\tilde{u}_{\tilde{x}\tilde{x}}(\tilde{x},\tilde{t})
=\frac{1}{\Lambda^3\epsilon^3}\left[u_t(x,t)+u(x,t)u_x(x,t)-u_{xx}(x,t)\right]=0
$,
and notice that
\begin{eqnarray*}
&\tilde{v}_{\tilde{t}}(\tilde{x},\tilde{t})=\frac{1}{\Lambda^3\epsilon^3}v_t(x,t), \quad
\tilde{v}_{\tilde{x}}(\tilde{x},\tilde{t})=\frac{1}{\Lambda^2\epsilon^2}v_x(x,t), \quad
\tilde{v}_{\tilde{x}\tilde{x}}(\tilde{x},\tilde{t})=\frac{1}{\Lambda^3\epsilon^3}v_{xx}(x,t), \\
&\left[\tilde{a}(\tilde{x})\tilde{v}(\tilde{x},\tilde{t})\right]_{\tilde{x}}
=\frac{1}{\Lambda\epsilon}\left[\frac{1}{\Lambda\epsilon}a^{\epsilon}(x)\frac{1}{\Lambda\epsilon}v(x,t)\right]_x
=\frac{1}{\Lambda^3\epsilon^3}\left[a^{\epsilon}(x)v(x,t)\right]_x,
\end{eqnarray*}
so $\tilde{v}(\tilde{x},\tilde{t})$ satisfies
$
\tilde{v}_{\tilde{t}}(\tilde{x},\tilde{t})+\left[\tilde{a}(\tilde{x})\tilde{v}(\tilde{x},\tilde{t})\right]_{\tilde{x}}
-\tilde{v}_{\tilde{x}\tilde{x}}(\tilde{x},\tilde{t})
=\frac{1}{\Lambda^3\epsilon^3}\left\{v_t(x,t)+\left[a^{\epsilon}(x)v(x,t)\right]_x-v_{xx}(x,t)\right\}=0
$.
The flux function $\tilde{f}(\tilde{u})=\frac{1}{\Lambda^2\epsilon^2}f(u)=\frac{1}{\Lambda^2\epsilon^2}f(\Lambda\epsilon\tilde{u})$, so
$\tilde{f}^{\prime}(\tilde{u})=\frac{1}{\Lambda\epsilon}f^{\prime}(\Lambda\epsilon\tilde{u})
=\frac{1}{\Lambda\epsilon}f^{\prime}(u).$

By the $\epsilon$-independent case $(3.5)$ in section 3 of \cite{Z1},
\begin{equation}
\tilde{v}(\tilde{x},\tilde{t})=\int_{-\infty}^{+\infty}G(\tilde{x},\tilde{t};\tilde{y})\tilde{h}(\tilde{y})d\tilde{y}
\label{rescaledlinearsolver}
\end{equation}
where $\tilde{h}(\tilde{x})=\tilde{v}(\tilde{x},0)=\frac{1}{\Lambda\epsilon}v(x,0)=\frac{1}{\Lambda\epsilon}h(x)$,
and $\tilde{y}=\Lambda\epsilon y$.
From \eqref{rescaledlinearsolver} we have
$$
\frac{1}{\Lambda\epsilon}v(x,t)=\tilde{v}(\tilde{x},\tilde{t})
=\int_{-\infty}^{+\infty}G(\tilde{x},\tilde{t};\tilde{y})\frac{1}{\Lambda\epsilon}h(y)\Lambda\epsilon dy,
$$
so
$$
v(x,t)=\Lambda\epsilon\int_{-\infty}^{+\infty}G(\tilde{x},\tilde{t};\tilde{y})h(y)dy
=\int_{-\infty}^{+\infty}\Lambda\epsilon G(\tilde{x},\tilde{t};\tilde{y})h(y)dy
=\int_{-\infty}^{+\infty}G^{\epsilon}(x,t;y)h(y)dy,
$$
where $G^{\epsilon}(x,t;y)
=\Lambda\epsilon G(\tilde{x},\tilde{t};\tilde{y})
=\Lambda\epsilon G(\Lambda\epsilon x,\Lambda^2\epsilon^2 t;\Lambda\epsilon y)$,
and $E^{\epsilon}$ and $S^{\epsilon}$ can be carried out explicitly as,
\begin{eqnarray*}
E^{\epsilon}(x,t;y)
&=&\Lambda\epsilon E(\tilde{x},\tilde{t};\tilde{y})
   =\Lambda\epsilon E(\Lambda\epsilon x,\Lambda^2\epsilon^2 t;\Lambda\epsilon y)\\
&=&C\Lambda\epsilon \tilde{\bar{u}}_{\tilde{x}}(\tilde{x})
   \left(\mathrm{errfn}\left(\frac{\tilde{y}+\tilde{f}^{\prime}(\tilde{\bar{u}}_{-})\tilde{t}}{\sqrt{4\tilde{t}}}\right)
   -\mathrm{errfn}\left(\frac{\tilde{y}-\tilde{f}^{\prime}(\tilde{\bar{u}}_{-})\tilde{t}}{\sqrt{4\tilde{t}}}\right)\right)\\
&=&C\Lambda\epsilon \frac{1}{\Lambda^2\epsilon^2}\bar{u}^{\epsilon}_{x}(x)
   \left(\mathrm{errfn}\left(\frac{\Lambda\epsilon y+\frac{1}{\Lambda\epsilon}f^{\prime}(\bar{u}_{-})\Lambda^2\epsilon^2 t}{\sqrt{4\Lambda^2\epsilon^2t}}\right)\right.\\
& &\left.-\mathrm{errfn}\left(\frac{\Lambda\epsilon y-\frac{1}{\Lambda\epsilon}f^{\prime}(\bar{u}_{-})\Lambda^2\epsilon^2
   t}{\sqrt{4\Lambda^2\epsilon^2 t}}\right)\right)\\
&=&C\frac{1}{\Lambda\epsilon}\bar{u}^{\epsilon}_{x}(x)
   \left(\mathrm{errfn}\left(\frac{y+f^{\prime}(\bar{u}_{-})t}{\sqrt{4t}}\right)
   -\mathrm{errfn}\left(\frac{y-f^{\prime}(\bar{u}_{-})t}{\sqrt{4t}}\right)\right)
\end{eqnarray*}
and
\begin{eqnarray*}
S^{\epsilon}(x,t;y)
&=&\Lambda\epsilon S(\tilde{x},\tilde{t};\tilde{y})
   =\Lambda\epsilon S(\Lambda\epsilon x,\Lambda^2\epsilon^2 t;\Lambda\epsilon y)\\
&=&\Lambda\epsilon\left[(4\pi \tilde{t})^{-\frac{1}{2}}
   e^{-\frac{\left(\tilde{x}-\tilde{y}-\tilde{f}^{\prime}(\tilde{\bar{u}}_{-})\tilde{t}\right)^2}{4t}}
   \left(\frac{e^{-\tilde{x}}}{e^{\tilde{x}}+e^{-\tilde{x}}}\right)\right.\\
& &\left.+(4\pi \tilde{t})^{-\frac{1}{2}}
   e^{-\frac{\left(\tilde{x}-\tilde{y}+\tilde{f}^{\prime}(\tilde{\bar{u}}_{-})\tilde{t}\right)^2}{4t}}
   \left(\frac{e^{\tilde{x}}}{e^{\tilde{x}}+e^{-\tilde{x}}}\right)\right]\\
&=&\Lambda\epsilon\left[(4\pi \Lambda^2\epsilon^2 t)^{-\frac{1}{2}}
   e^{-\frac{\left(\Lambda\epsilon x-\Lambda\epsilon y-\frac{1}{\Lambda\epsilon}f^{\prime}(\bar{u}_{-})\Lambda^2\epsilon^2 t\right)^2}{4\Lambda^2\epsilon^2 t}}
   \left(\frac{e^{-\Lambda\epsilon x}}{e^{\Lambda\epsilon x}+e^{-\Lambda\epsilon x}}\right)\right.\\
& &\left.+(4\pi \Lambda^2\epsilon^2 t)^{-\frac{1}{2}}
   e^{-\frac{\left(\Lambda\epsilon x-\Lambda\epsilon y+\frac{1}{\Lambda\epsilon}f^{\prime}(\bar{u}_{-})\Lambda^2\epsilon^2 t\right)^2}{4\Lambda^2\epsilon^2 t}}
   \left(\frac{e^{\Lambda\epsilon x}}{e^{\Lambda\epsilon x}+e^{-\Lambda\epsilon x}}\right)\right]\\
&=&(4\pi t)^{-\frac{1}{2}}e^{-\frac{\left(x-y-f^{\prime}(\bar{u}_{-})t\right)^2}{4t}}
   \left(\frac{e^{-\Lambda \epsilon x}}{e^{\Lambda \epsilon x}+e^{-\Lambda \epsilon x}}\right)\\
& &+(4\pi t)^{-\frac{1}{2}}e^{-\frac{\left(x-y+f^{\prime}(\bar{u}_{-})t\right)^2}{4t}}
   \left(\frac{e^{\Lambda \epsilon x}}{e^{\Lambda \epsilon x}+e^{-\Lambda \epsilon x}}\right).
\end{eqnarray*}

Finally, $R^{\epsilon}$ and $R^{\epsilon}_{y}$ can be similarly derived.

Now we have the following estimates by rescaling Theorem \ref{greenfunctionbounds}.

\begin{pro}
The Green function $G(x,t;y)$ associated with the linearized evolution equation
 \eqref{generalscalarlinearized} may be decomposed as
 \begin{equation}
 G^{\epsilon}(x,t;y)=E^{\epsilon}(x,t;y)+S^{\epsilon}(x,t;y)+R^{\epsilon}(x,t;y)
 \end{equation}
where for $y \leq 0:$
 \begin{equation}
 E^{\epsilon}(x,t;y)=C\bar{u}^{\epsilon}_{x}(x)\frac{1}{\Lambda\epsilon}
 \left(\mathrm{errfn}\left(\frac{y+f^{\prime}(\bar{u}_{-})t}{\sqrt{4t}}\right)
 -\mathrm{errfn}\left(\frac{y-f^{\prime}(\bar{u}_{-})t}{\sqrt{4t}}\right)\right),
 \end{equation}
and
 \begin{equation}
 \begin{aligned}
 S^{\epsilon}(x,t;y)
 &     =(4\pi t)^{-\frac{1}{2}}e^{-\frac{\left(x-y-f^{\prime}(\bar{u}_{-})t\right)^2}{4t}}\left(\frac{e^{-\Lambda \epsilon x}}{e^{\Lambda \epsilon x}+e^{-\Lambda \epsilon x}}\right)\\
 &\quad +(4\pi t)^{-\frac{1}{2}}e^{-\frac{\left(x-y+f^{\prime}(\bar{u}_{-})t\right)^2}{4t}}\left(\frac{e^{\Lambda \epsilon x}}{e^{\Lambda \epsilon x}+e^{-\Lambda \epsilon x}}\right),
 \end{aligned}
 \end{equation}

 \begin{equation}
 \begin{aligned}
 R^{\epsilon}(x,t;y)
 &     =\mathbf{O}\left(\Lambda\epsilon e^{-\eta \Lambda\epsilon t}e^{-\frac{|x-y|^2}{Mt}}\right)\\
 &\quad +\mathbf{O}\left((t+1)^{-\frac{1}{2}}e^{-\eta x^+}+e^{-\eta |x|}\right)t^{-\frac{1}{2}}e^{-\frac{(x-y-f^{\prime}(\bar{u}_{-})t)^2}{Mt}}
 \end{aligned}
 \end{equation}

 \begin{equation}
 \begin{aligned}
 R^{\epsilon}_{y}(x,t;y)
 &     =\mathbf{O}\left(\Lambda\epsilon e^{-\eta \Lambda\epsilon t}e^{-\frac{|x-y|^2}{Mt}}\right)\\
 &\quad +\mathbf{O}\left((t+1)^{-\frac{1}{2}}e^{-\eta x^+}+e^{-\eta |x|}\right)t^{-1}e^{-\frac{(x-y-f^{\prime}(\bar{u}_{-})t)^2}{Mt}}
 \end{aligned}
 \end{equation}
\end{pro}

Again by rescaling, we can recover all the estimates in the $\epsilon$-independent case for the small amplitude case.
Here we give an example of how the rescaling argument works to recover the bounds in Proposition \ref{Pro8.7}.
Recall that we have the following estimate in the $\epsilon$-independent case,
\begin{equation*}
\left|\int_{-\infty}^{+\infty}\tilde{G}(\tilde{x},\tilde{t};\tilde{y})\tilde{h}(\tilde{y})d\tilde{y}\right|_{L^{p}(\tilde{x})}
\leq C\tilde{t}^{-\frac{1}{2}(1-\frac{1}{p})}|\tilde{h}|_{L^1},
\end{equation*}
now we substitute $\frac{1}{\Lambda\epsilon}\tilde{G}^{\epsilon}(x,t;y)$ for $\tilde{G}(\tilde{x},\tilde{t};\tilde{y})$,
$\frac{1}{\Lambda\epsilon}h(y)$ for $\tilde{h}(\tilde{y})$,
$\Lambda\epsilon y$ for $\tilde{y}$, $\Lambda^2\epsilon^2 t$ for $\tilde{t}$, we get
\begin{align*}
      \left|\int_{-\infty}^{+\infty}\frac{1}{\Lambda\epsilon}\tilde{G}^{\epsilon}(x,t;y)\frac{1}{\Lambda\epsilon}h(y)
      \Lambda\epsilon dy\right|_{L^{p}(\tilde{x})}
&\leq C(\Lambda^2\epsilon^2 t)^{-\frac{1}{2}(1-\frac{1}{p})}\int_{-\infty}^{+\infty}|\tilde{h}(\tilde{y})|d\tilde{y}\\
      \left(\int_{-\infty}^{+\infty}\left|\int_{-\infty}^{+\infty}
      \frac{1}{\Lambda\epsilon}\tilde{G}^{\epsilon}(x,t;y)\frac{1}{\Lambda\epsilon}h(y)
      \Lambda\epsilon dy\right|^{p}\Lambda\epsilon dx\right)^{\frac{1}{p}}
&\leq C(\Lambda^2\epsilon^2 t)^{-\frac{1}{2}(1-\frac{1}{p})}
      \int_{-\infty}^{+\infty}\left|\frac{1}{\Lambda\epsilon}h(y)\right|\Lambda\epsilon dy\\
      (\Lambda\epsilon)^{\frac{1}{p}-1}\left(\int_{-\infty}^{+\infty}\left|\int_{-\infty}^{+\infty}
      \tilde{G}^{\epsilon}(x,t;y)h(y)
      dy\right|^{p}dx\right)^{\frac{1}{p}}
&\leq C(\Lambda\epsilon)^{-(1-\frac{1}{p})}t^{-\frac{1}{2}(1-\frac{1}{p})}
      \int_{-\infty}^{+\infty}\left|h(y)\right|dy
\end{align*}
when we cancel $(\Lambda\epsilon)^{\frac{1}{p}-1}$ on both sides, we get
\begin{equation*}
\left(\int_{-\infty}^{+\infty}\left|\int_{-\infty}^{+\infty}\tilde{G}^{\epsilon}(x,t;y)h(y)dy\right|^{p}dx\right)^{\frac{1}{p}}
\leq Ct^{-\frac{1}{2}(1-\frac{1}{p})}\int_{-\infty}^{+\infty}\left|h(y)\right|dy
\end{equation*}
or equivalently
\begin{equation*}
\left|\int_{-\infty}^{+\infty}\tilde{G}^{\epsilon}(x,t;y)h(y)dy\right|_{L^{p}(x)}
\leq Ct^{-\frac{1}{2}(1-\frac{1}{p})}|h|_{L^1}.
\end{equation*}
Similarly, we can recover all other estimates with the $\tilde{G}^{\epsilon}$ term.

For the estimates with term
$e^{\epsilon}(y,t)$, we note that $\tilde{e}(\tilde{y},\tilde{t})=\Lambda\epsilon e^{\epsilon}(y,t)
=\Lambda\epsilon e^{\epsilon}(\frac{\tilde{y}}{\Lambda\epsilon},\frac{\tilde{t}}{\Lambda^2 \epsilon^2})$, thus
\begin{eqnarray*}
\tilde{e}_{\tilde{y}}(\tilde{y},\tilde{t})=e^{\epsilon}_{y}(y,t), \quad
\tilde{e}_{\tilde{t}}(\tilde{y},\tilde{t})=\frac{1}{\Lambda\epsilon}e^{\epsilon}_{t}(y,t), \quad
\tilde{e}_{\tilde{t}\tilde{y}}(\tilde{y},\tilde{t})=\frac{1}{\Lambda^2\epsilon^2}e^{\epsilon}_{ty}(y,t).
\end{eqnarray*}
Using these to substitute, for example the following estimate,
\begin{equation*}
\left|\int_{-\infty}^{+\infty}\tilde{e}_{\tilde{y}}(\tilde{y},\tilde{t})\tilde{h}(\tilde{y})d\tilde{y}\right|
\leq C\tilde{t}^{-\frac{1}{2}}|\tilde{h}|_{L^1},
\end{equation*}
we get
\begin{equation*}
\left|\int_{-\infty}^{+\infty}e^{\epsilon}_{y}(y,t)\frac{1}{\Lambda\epsilon}h(y)\Lambda\epsilon dy\right|
\leq C(\Lambda^2\epsilon^2 t)^{-\frac{1}{2}}|h|_{L^1},
\end{equation*}
and this further simplifies to be
\begin{equation*}
\left|\int_{-\infty}^{+\infty}e^{\epsilon}_{y}(y,t)h(y)dy\right|
\leq C(\Lambda\epsilon)^{-1}t^{-\frac{1}{2}}|h|_{L^1}.
\end{equation*}
Other estimates can be derived similarly.

Now we record these results in the following proposition and lemma which are corresponding versions of Proposition \ref{Pro8.7} and Lemma \ref{Lem8.8}.

\begin{pro} \label{Pro_Ge}
The Green function $G^{\epsilon}$ decomposes as $G^{\epsilon}=E^{\epsilon}+S^{\epsilon}+R^{\epsilon}=E^{\epsilon}+\tilde{G}^{\epsilon}$,
where $\tilde{G}^{\epsilon}=S^{\epsilon}+R^{\epsilon}$,
$E^{\epsilon}(x,t;y)=C\bar{u}^{\epsilon}_{x}(x)e^{\epsilon}(y,t)$ and
$$e^{\epsilon}(y,t):=\frac{1}{\Lambda\epsilon}\left(\mathrm{errfn}\left(\frac{y+f^{\prime}(\bar{u}_{-})t}{\sqrt{4t}}\right)
 -\mathrm{errfn}\left(\frac{y-f^{\prime}(\bar{u}_{-})t}{\sqrt{4t}}\right)\right)$$
then for some $C>0$ independent of $\epsilon$ and all $t>0$,
\begin{equation} \label{9.41}
\left|\int_{-\infty}^{+\infty}\tilde{G}^{\epsilon}(x,t;y)h(y)dy\right|_{L^{p}(x)}\leq Ct^{-\frac{1}{2}(1-\frac{1}{p})}|h|_{L^1},
\end{equation}
\begin{equation} \label{9.42}
\left|\int_{-\infty}^{+\infty}\tilde{G}^{\epsilon}_{y}(x,t;y)h(y)dy\right|_{L^{p}(x)}\leq Ct^{-\frac{1}{2}(1-\frac{1}{p})-\frac{1}{2}}|h|_{L^1},
\end{equation}
\begin{equation} \label{9.43}
\left|\int_{-\infty}^{+\infty}\tilde{G}^{\epsilon}(x,t;y)h(y)dy\right|_{L^{p}(x)}\leq C|h|_{L^p},
\end{equation}
\begin{equation} \label{9.44}
\left|\int_{-\infty}^{+\infty}\tilde{G}^{\epsilon}_{y}(x,t;y)h(y)dy\right|_{L^{p}(x)}\leq Ct^{-\frac{1}{2}}|h|_{L^p}.
\end{equation}
and
\begin{equation} \label{9.45}
\left|\int_{-\infty}^{+\infty}e^{\epsilon}(y,t)h(y)dy\right|\leq C\epsilon^{-1}|h|_{L^1},
\left|\int_{-\infty}^{+\infty}e^{\epsilon}_{y}(y,t)h(y)dy\right|\leq C\epsilon^{-1}t^{-\frac{1}{2}}|h|_{L^1},
\end{equation}
\begin{equation} \label{9.46}
\left|\int_{-\infty}^{+\infty}e^{\epsilon}_{t}(y,t)h(y)dy\right|\leq Ct^{-\frac{1}{2}}|h|_{L^1},
\left|\int_{-\infty}^{+\infty}e^{\epsilon}_{ty}(y,t)h(y)dy\right|\leq Ct^{-1}|h|_{L^1},
\end{equation}
\begin{equation} \label{9.47}
\left|\int_{-\infty}^{+\infty}e^{\epsilon}_{t}(y,t)h(y)dy\right|\leq C|h|_{L^{\infty}},
\left|\int_{-\infty}^{+\infty}e^{\epsilon}_{yt}(y,t)h(y)dy\right|\leq Ct^{-\frac{1}{2}}|h|_{L^{\infty}}.
\end{equation}
\end{pro}

\begin{lem} \label{Lem_e}
For some $C>0$ and all $0<\epsilon \leq 1$, and all $t>0$,
 \begin{eqnarray}
 |e^{\epsilon}(\cdot,t)|_{L^{\infty}}&\leq& \frac{C}{\epsilon},\label{9.56}\\
 |e^{\epsilon}_y(\cdot,t)|_{L^{p}}&\leq& \frac{C}{\epsilon}t^{-\frac{1}{2}(1-\frac{1}{p})},\label{9.57}\\
 |e^{\epsilon}_t(\cdot,t)|_{L^{p}}&\leq& Ct^{-\frac{1}{2}(1-\frac{1}{p})},\label{9.58}\\
 |e^{\epsilon}_{ty}(\cdot,t)|_{L^{p}}&\leq& Ct^{-\frac{1}{2}(1-\frac{1}{p})-\frac{1}{2}},\label{9.59}\\
 |e^{\epsilon}_y(y,t)|&\leq&
 \frac{C}{\epsilon}t^{-\frac{1}{2}}\left(e^{-\frac{(-y-t)^2}{Ct}}+e^{-\frac{(-y+t)^2}{Ct}}\right),\label{9.60}\\
 |e^{\epsilon}_t(y,t)|&\leq& Ct^{-\frac{1}{2}}\left(e^{-\frac{(-y-t)^2}{Ct}}+e^{-\frac{(-y+t)^2}{Ct}}\right),\label{9.61}\\
 |e^{\epsilon}_{ty}(y,t)|&\leq& Ct^{-1}\left(e^{-\frac{(-y-t)^2}{Ct}}+e^{-\frac{(-y+t)^2}{Ct}}\right).\label{9.62}
 \end{eqnarray}
\end{lem}

\section{Nonlinear stability in the small amplitude case} \label{sectionstabilitysmall}

Having the estimates in Proposition \ref{Pro_Ge} and Lemma \ref{Lem_e}, we are now ready to derive the nonlinear stability results of the small amplitude case.
Similarly, define the \textit{perturbation} as $u(x,t):=\tilde{u}(x+\alpha(t),t)-\bar{u}^{\epsilon}(x).$
Following some similar computation as the $\epsilon$-independent case, we have the following representations,
\begin{equation*}
\begin{aligned}
u(x,t)
=\int_{-\infty}^{+\infty}\tilde{G}^{\epsilon}(x,t;y)u_{0}(y)dy
 -\int_{0}^{t}\int_{-\infty}^{+\infty}\tilde{G}^{\epsilon}_{y}(x,t-s;y)\left(N(u)(y,s)+\dot{\alpha}(s)u(y,s)\right)dyds
\end{aligned}
\end{equation*}

\begin{equation*}
\begin{aligned}
\alpha(t)
=-\int_{-\infty}^{+\infty}e^{\epsilon}(y,t)u_0(y)dy
 +\int_{0}^{t}\int_{-\infty}^{+\infty}e^{\epsilon}_y(y,t-s)(N(u)(y,s)+\dot{\alpha}(s)u(y,s))dyds
\end{aligned}
\end{equation*}

\begin{equation*}
\begin{aligned}
\dot{\alpha}(t)
=-\int_{-\infty}^{+\infty}e^{\epsilon}_t(y,t)u_0(y)dy
 \quad +\int_{0}^{t}\int_{-\infty}^{+\infty}e^{\epsilon}_{yt}(y,t-s)(N(u)(y,s)+\dot{\alpha}(s)u(y,s))dyds.
\end{aligned}
\end{equation*}

We are ready to prove Theorem \ref{mainthmsmall} now.

\begin{proof}[Proof of Theorem \ref{mainthmsmall}]
The proof of the first two bounds follows exactly as the proof of the fixed-amplitude($\epsilon$-independent) case. Since the bounds on $\tilde{G}^{\epsilon}$, $\tilde{G}^{\epsilon}_{y}$,
$e^{\epsilon}_{t}$ and $e^{\epsilon}_{yt}$ in Proposition \ref{Pro_Ge} are exactly the same as the bounds on $\tilde{G}$, $\tilde{G}_{y}$, $e_{t}$ and $e_{yt}$ in Proposition \ref{Pro8.7}
in the fixed-amplitude case. Using the bounds on $e^{\epsilon}$ and $e^{\epsilon}_{y}$ we obtain the third bound. It is $\epsilon^{-1}$ times the corresponding bound for $|\alpha(t)|$ because the the bounds on $e^{\epsilon}$ and $e^{\epsilon}_{y}$ are all $\epsilon^{-1}$ times the corresponding bounds on $e$ and $e_{y}$ in the fixed-amplitude case.

To derive the fourth bound, we note that $\tilde{u}(x,t)-\bar{u}^{\epsilon}(x)=u(x-\alpha(t),t)-(\bar{u}^{\epsilon}(x)-\bar{u}^{\epsilon}(x-\alpha(t)))$, so that
$\tilde{u}(\cdot,t)-\bar{u}^{\epsilon}$ is controlled by the sum of $|u|$ and $|\bar{u}^{\epsilon}(x)-\bar{u}^{\epsilon}(x-\alpha(t))|$. By monotonicity of scalar shock profiles of the
first-order scalar profile ODE, $\bar{u}^{\epsilon}(x)-\bar{u}^{\epsilon}(x-\alpha(t))$ has one sign, hence
\begin{equation*}
\left|\bar{u}^{\epsilon}(x)-\bar{u}^{\epsilon}(x-\alpha(t))\right|_{L^1}
=\left|\int_{-\infty}^{+\infty}(\bar{u}^{\epsilon}(x)-\bar{u}^{\epsilon}(x-\alpha(t)))dx\right|=|\alpha(t)||\bar{u}^{\epsilon}_{+}-\bar{u}^{\epsilon}_{-}|,
\end{equation*}
and thus by bounds on $|\alpha(t)|$,
\begin{equation*}
\left|\bar{u}^{\epsilon}(x)-\bar{u}^{\epsilon}(x-\alpha(t))\right|_{L^1}
=2\epsilon |\alpha(t)| \leq 2CE_0,
\end{equation*}
where $E_0=|\tilde{u}-\bar{u}^{\epsilon}|_{L^1\cap L^{\infty}}|_{t=0}$.
Similarly, by the Mean Value Theorem,
\begin{equation*}
\left|\bar{u}^{\epsilon}(x)-\bar{u}^{\epsilon}(x-\alpha(t))\right|
\leq |\alpha(t)||(\bar{u}^{\epsilon})^{\prime}|_{L^{\infty}}\leq (\frac{CE_0}{\epsilon})(\epsilon^2)=CE_0\epsilon,
\end{equation*}
here we used the asymptotic $(\bar{u}^{\epsilon})^{\prime}\sim \epsilon^2 e^{-\theta \epsilon |x|}$. Thus,
$|\bar{u}^{\epsilon}(x)-\bar{u}^{\epsilon}(x-\alpha(t))|_{L^1 \cap L^{\infty}}\leq CE_0$, and so $|\tilde{u}(x,t)-\bar{u}^{\epsilon}(x)|_{L^1 \cup L^{\infty}}\leq CE_0$ for all
$t\geq 0$, for $E_0$ sufficiently small. This verifies the fourth inequality, yielding nonlinear stability.
\end{proof}

This completes the proof, verifying nonlinear stability.

\bibliographystyle{amsalpha}
\bibliography{books}

\begin{thebibliography}{100}
 \bibitem[Da]{Da} C. Dafermos, \emph{Hyperbolic Conservation Laws in Continuum Physics}, Grundlehren der mathematischen Wissenschaften 325, Springer-Verlag, Berlin, Heidelberg, 1999, 2005, 2010.
 \bibitem[He]{He} D. Henry, \emph{Geometric Theory of Semilinear Parabolic Equations}, Lecture Notes in Mathematics 840, Springer-Verlag, New York, 1981.
 \bibitem[Ho]{Ho} P. Howard, \emph{Pointwise Green's function approach to stability for scalar conservation laws}, Comm. Pure Appl. Math. 52(1999), no. 10, 1295-1313.
 \bibitem[LZe]{LZe} T.P. Liu and Y. Zeng, \emph{Time-asymptotic behavior of wave propagation around a viscous shock profile}, Comm. Math. Phys. 290(2009), no.1, 23-82.
 \bibitem[MP]{MP} A. Majda and R. Pego, \emph{Stable Viscosity Matrices for Systems of Conservation Laws}, J. Diff. Eqs. 56(1985), 229-262.
 \bibitem[MaZ]{MaZ} C. Mascia and K. Zumbrun, \emph{Pointwise Green Function Bounds for Shock Profiles of Systems with Real Viscosity}, Arch. Rational Mech. Anal. 169(2003), 177-263.
 \bibitem[MaZ2]{MaZ2} C. Mascia and K. Zumbrun, \emph{Stability of Small-Amplitude Shock Profiles of Symmetric Hyperbolic-Parabolic Systems}, Comm. Pure Appl. Math. 57(2004), no.7, 841-876.
 \bibitem[MaZ3]{MaZ3} C. Mascia and K. Zumbrun, \emph{Stability of Large-Amplitude Viscous Shock Profiles of Hyperbolic-Parabolic Systems}, Arch. Ration. Mech. Anal. 172(2004), no.1, 93-131.
 \bibitem[PZ]{PZ} R. Plaza and K. Zumbrun, \emph{An Evans Function Approach to Spectral Stability of Small-Amplitude Shock Profiles}, Discrete and Continuous Dynamical Systems - Series B, Vol.10 No.4(2004), 885-924.
 \bibitem[Z1]{Z1} K. Zumbrun, \emph{Instantaneous shock location and one-dimensional nonlinear stability of viscous shock waves}, Quart. Appl. Math. 69(2011), 177-202.
 \bibitem[ZH]{ZH} K. Zumbrun and P. Howard, \emph{Pointwise Semigroup Methods and Stability of Viscous Shcok Waves}, Indiana. Univ. Math. J. 47(1998), 741-871.
\end{thebibliography}

\end{document}